\documentclass[a4paper,10pt]{article}

\usepackage[english, french]{babel}
\usepackage[applemac]{inputenc}
\usepackage[T1]{fontenc}

\usepackage{
graphics,
latexsym,
textcomp,
amsfonts,
amssymb,
amsbsy,
amsmath,
amsthm,
geometry,
version,
enumitem,
xspace,
}
\usepackage[hidelinks]{hyperref}

\setlist{nosep}
\setlist[enumerate]{label=\emph{\alph*)}}

\usepackage[all,2cell,cmtip]{xy} \UseAllTwocells 

\SilentMatrices

\newcommand\EquiQuillen{W_{\infty}^{\Delta}}
\newcommand\EnsSimp{\widehat{\Delta}}

\newcommand\DeuxFoncteurStrict[0]{\deux{}foncteur strict}
\newcommand\DeuxFoncteursStricts[0]{\deux{}foncteurs stricts}
\newcommand\DeuxFoncteurLax[0]{\deux{}foncteur lax}
\newcommand\DeuxFoncteursLax[0]{\deux{}foncteurs lax}
\newcommand\DeuxFoncteurCoLax[0]{\deux{}foncteur colax}
\newcommand\DeuxFoncteursCoLax[0]{\deux{}foncteurs colax}
\newcommand\DeuxTransformationLax[0]{transformation}
\newcommand\DeuxTransformationsLax[0]{transformations}
\newcommand\DeuxTransformationCoLax[0]{optransformation}

\newcommand\DeuxTransformationsStrictes[0]{transformations strictes}

\newcommand\TrancheLax[3]{{#1}/\negmedspace/_{\negmedspace {\rm{l}}}^{#2}{#3}}

\newcommand\TrancheCoLax[3]{{#1}/\negmedspace/_{\negmedspace {\rm{c}}}^{#2}{#3}}

\newcommand\OpTrancheCoLax[3]{{#3}{\backslash\mspace{-5mu}\backslash}_{\mspace{-.3mu}{\rm{c}}}^{\mspace{-6.mu}#2}{#1}}

\newcommand\OpTrancheLax[3]{{#3}{\backslash\mspace{-5mu}\backslash}_{\mspace{-.3mu}{\rm{l}}}^{\mspace{-6.mu}#2}{#1}}

\newcommand\DeuxFoncTrancheLax[2]{{#1}/\negmedspace/_{\negmedspace {\rm{l}}}{#2}}

\newcommand\DeuxFoncTrancheCoLax[2]{{#1}/\negmedspace/_{\negmedspace {\rm{c}}}{#2}}

\newcommand\DeuxFoncOpTrancheLax[2]{{#2}{\backslash\mspace{-5mu}\backslash}_{\mspace{-.3mu}{\rm{l}}}^{\mspace{-6.mu}}{#1}}

\newcommand\DeuxFoncOpTrancheCoLax[2]{{#2}{\backslash\mspace{-5mu}\backslash}_{\mspace{-.3mu}{\rm{c}}}^{\mspace{-6.mu}}{#1}}

\newcommand\DeuxFoncTrancheLaxCoq[3]{{#1}/\negmedspace/_{\negmedspace{\rm{l}}}^{#2}{#3}}

\newcommand\Fibre[3]{{#2}^{-1}({#3})}

\newcommand\CompDeuxUn[0]{}

\newcommand\CompDeuxZero[0]{\circ}

\newcommand\DeuxCellStructComp[3]{{#1}_{{#2},{#3}}}

\newcommand\DeuxCellStructId[2]{{#1}_{{#2}}}

\newcommand\Objets[1]{\mathsf{Ob}({#1})}

\newcommand\UnCell[1]{\mathsf{Fl}_{1}({#1})}

\newcommand\TildeLax[1]{\widetilde{#1}}
\newcommand\BarreLax[1]{\overline{#1}}

\newcommand\LaxCanonique[1]{\eta_{{#1}}}
\newcommand\StrictCanonique[1]{\epsilon_{{#1}}}
\newcommand\TransLaxCanonique[0]{\eta}
\newcommand\TransStrictCanonique[0]{\epsilon}

\newcommand\DeuxInt[1]{\int_{#1}}
\newcommand\DeuxIntOp[1]{\int_{#1}^{op}}
\newcommand\DeuxIntCo[1]{\int_{#1}^{co}}

\newcommand\deux{$2$\nobreakdash-}
\newcommand\un{$1$\nobreakdash-}

\newcommand\CatHom[3]{\underline{\Hom}_{#1}(#2, #3)}

\newcommand\Ens{{\mathcal{E} \mspace{-2.mu} \it{ns}}}
\newcommand\Top{{\mathcal{T}\mspace{-2.mu}\it{op}}}

\newcommand\Cat{{\mathcal{C} \mspace{-2.mu} \it{at}}}
\newcommand{\tCat}{\texorpdfstring{\Cat}{Cat}}
\newcommand\DeuxCat{\text{$2$-$\Cat$}}
\newcommand\tDeuxCat{\texorpdfstring{\DeuxCat}{2-Cat}}
\newcommand\lax{\text{\it lax}}
\newcommand\DeuxCatLax{\DeuxCat_{\lax}}

\newcommand\DeuxCatDeuxCat{\underline{\DeuxCat}}

\newcommand\ClasseDeuxLocFond{localisateur fondamental de $\DeuxCat$}
\newcommand\ClasseDeuxLocFondS{localisateurs fondamentaux de $\DeuxCat$}
\newcommand\ClasseDeuxLocFondLax[0]{localisateur fondamental de $\DeuxCatLax$}
\newcommand\ClasseDeuxLocFondLaxS[0]{localisateurs fondamentaux de $\DeuxCatLax$}
\newcommand\ClassesDeuxLocFond{localisateurs fondamentaux de $\DeuxCat$}
\newcommand\ClassesDeuxLocFondLax{localisateurs fondamentaux de $\DeuxCatLax$}

\newcommand\ClasseUnLocFond{localisateur fondamental de $\Cat$}
\newcommand\ClassesUnLocFond{localisateurs fondamentaux de $\Cat$}

\newcommand\UnLocFond[1]{{#1}}
\newcommand\DeuxLocFond[1]{\mathcal{#1}}
\newcommand\DeuxLocFondLax[1]{\mathcal{#1}}
\newcommand\DeuxLocFondLaxInduit[1]{\mathcal{#1}_{\lax}}

\newcommand\UnLocFondMin{\UnLocFond{W}_{\infty}^{1}}
\newcommand\DeuxLocFondMin{\DeuxLocFond{W}_{\infty}^{2}}
\newcommand\DeuxLocFondLaxMin{\DeuxLocFond{W}_{\infty,lax}^{2}}

\newcommand\Localisation[2]{{#2}^{-1}{#1}}

\newcommand\UnNerf{N}
\newcommand\NerfLaxNor{N_{\rm{l,n}}}
\newcommand\NerfLax{N_{\rm{l}}}
\newcommand\NerfCatLaxNor{\underline{N}_{\rm{l,n}}}
\newcommand\NerfHom{N_{2}}
\newcommand\SupUnObjet[1]{sup_{#1}}

\newcommand\SupLaxNorObjet[1]{sup^{\rm{l,n}}_{#1}}

\newcommand\SupLaxObjet[1]{sup^{\rm{l}}_{#1}}

\newcommand\SupCatLaxNorObjet[1]{\underline{sup}^{\rm{l,n}}_{#1}}

\newcommand\SupHom{\underline{sup}}
\newcommand\SupHomObjet[1]{\underline{sup}_{#1}}

\newcommand\op{\text{op}}
\newcommand\co{\text{co}}
\newcommand\coop{\text{coop}}
\newcommand\DeuxCatUnOp[1]{{#1}^{\op}}
\newcommand\DeuxCatDeuxOp[1]{{#1}^{\co}}
\newcommand\DeuxCatToutOp[1]{{#1}^{\coop}}

\newcommand\DeuxFoncUnOp[1]{{#1}^{\op}}
\newcommand\DeuxFoncDeuxOp[1]{{#1}^{\co}}
\newcommand\DeuxFoncToutOp[1]{{#1}^{\coop}}

\newcommand\mathdeuxcat[1]{\mathcal{#1}}
\newcommand\DeuxCatPonct{e}

\newcommand\UnCatPonct{e}

\theoremstyle{plain}
\newtheorem{theo}{Théorème}[section]
\newtheorem{prop}[theo]{Proposition}
\newtheorem{lemme}[theo]{Lemme}
\newtheorem{corollaire}[theo]{Corollaire}
\theoremstyle{remark}
\newtheorem{exemple}[theo]{Exemple}
\newtheorem{rem}[theo]{Remarque}
\theoremstyle{definition}
\newtheorem{df}[theo]{Définition}

\newtheorem{paragr}[theo]{} 
\newtheorem*{nota}{Notations}

\theoremstyle{plain}

\theoremstyle{definition}

\theoremstyle{remark}

\def\xpoint{\futurelet\@let@token\@xpoint}
\def\@xpoint{%
  \ifx\@let@token.\else
    .%
  \fi
  \xspace}

\DeclareMathOperator{\Hom}{\mathsf{Hom}}

\newcommand{\Sd}{Sd\kern 1pt}

\makeatother 

\title{Théories homotopiques des 2-catégories}

\author{Jonathan Chiche}

\date{}

\begin{document}
\xyoption{v2}

\maketitle

\begin{abstract}
Ce texte développe les premiers éléments d'une théorie de l'homotopie des
\deux{}catégories analogue à la théorie de l'homotopie des catégories
développée par Grothendieck dans \emph{Pursuing Stacks}. On y définit la
notion de \emph{localisateur fondamental de $\DeuxCat$}, généralisation
\deux{}catégorique de la notion de localisateur fondamental due à Grothendieck, et l'on montre que les théories homotopiques de $\Cat$ et
$\DeuxCat$ sont équivalentes en un sens remarquablement fort: il existe un
isomorphisme compatible à l'opération de localisation entre les classes
ordonnées par inclusion des localisateurs fondamentaux de $\Cat$ et de
$\DeuxCat$. Cela permet notamment d'en déduire une caractérisation purement
\deux{}catégorique de la notion d'équivalence faible homotopique dans
$\DeuxCat$, sans faire appel aux espaces topologiques ou aux ensembles
simpliciaux. 
%Dans un appendice, Dimitri Ara déduit, de nos résultats et de
%résultats qu'il a obtenus avec Maltsiniotis, l'existence, pour
%essentiellement tout localisateur fondamental $\W$ de $\DeuxCat$, d'une
%structure de catégorie de modèles «~à la \hbox{Thomason}~» sur~$\DeuxCat$
%dont les équivalences faibles sont les éléments de $\W$. Il montre que les
%structures de catégorie de modèles ainsi obtenues modélisent exactement les
%localisations de Bousfield à gauche combinatoires de la théorie de
%l'homotopie classique des ensembles simpliciaux.
\end{abstract}

\selectlanguage{english}
\begin{abstract}
This text develops a homotopy theory of \deux{}categories analogous to Grothendieck's homotopy theory of categories developed in \emph{Pursuing Stacks}. We define the notion of \emph{basic localizer of $\DeuxCat$}, \deux{}categorical generalization of Grothendieck's notion of basic localizer, and we show that the homotopy theories of $\Cat$ and $\DeuxCat$ are equivalent in a remarkably strong sense: there is an isomorphism, compatible with localization, between the ordered classes of basic localizers of $\Cat$ and $\DeuxCat$. It follows that weak homotopy equivalences in $\DeuxCat$ can be characterised in an internal way, without mentioning topological spaces or simplicial sets. 
%In an appendix, Dimitri Ara obtains, from our results and results he has obtained with Maltsiniotis, the existence, for almost every basic localizer $\W$ of~$\DeuxCat$, of an associated "Thomason model structure" on $\DeuxCat$ whose weak equivalences are the elements of $\W$.
%He shows that these model category structures model exactly combinatorial left Bousfield localizations of the classical homotopy theory of simplicial sets.
\end{abstract}

\selectlanguage{french}

\section{Introduction}

Ce travail s'inscrit dans une entreprise de généralisation de la théorie de l'homotopie « de Grothendieck » aux catégories supérieures. 

Les objets de base de la théorie de l'homotopie sont, classiquement, les espaces topologiques ou les CW\nobreakdash-complexes. Les ensembles simpliciaux permettent une approche plus combinatoire. L'équivalence des deux points de vue se précise au moyen de la théorie des catégories de modèles de Quillen : il existe une équivalence de Quillen entre la catégorie des espaces topologiques $\Top$ et celle des ensembles simpliciaux $\EnsSimp$, ces deux catégories se trouvant munies des structures de catégories de modèles dégagées par Quillen dans \cite{QuillenHomotopical}.

Dans \emph{Pursuing Stacks} \cite{Poursuite}, Grothendieck développe une théorie de l'homotopie fondée non pas sur la catégorie des espaces topologiques, non plus que sur celle des ensembles simpliciaux, mais sur la catégorie $\Cat$ des petites catégories, « vue avec un œil de géomètre par l'ensemble d'intuitions, étonnamment riche, provenant des topos » \cite{LettreGrothendieckThomason}. Son travail l'amène à dégager la notion de \emph{catégorie test}, petite catégorie dont le topos des préfaisceaux modélise canoniquement les types d'homotopie, notion dont la catégorie des simplexes $\Delta$ constitue le paradigme historique. S'apercevant qu'il n'a, pour étudier la théorie des catégories test, utilisé qu'un petit nombre de propriétés formelles des équivalences faibles homotopiques classiques de $\Cat$ — définies comme étant les foncteurs dont le nerf est une équivalence faible simpliciale —, il définit la notion de \emph{localisateur fondamental} comme étant une classe de morphismes de $\Cat$ vérifiant ces propriétés, dont la plus importante est le Théorème A de Quillen. À tout localisateur fondamental sont associées diverses notions, non seulement celle de catégorie test et ses variantes, mais également, par exemple, celles de foncteur propre et de foncteur lisse, définies par des propriétés de changement de base analogues à celles des théorèmes de changement de base propre ou lisse en géométrie algébrique. La théorie de l'homotopie de Grothendieck généralise une part fondamentale de la théorie « classique » de l'homotopie simpliciale, dont elle propose une approche conceptuelle remarquablement fructueuse. Elle a été dégagée et développée par Grothendieck dans \cite{Poursuite}, présentée de façon plus « bourbachique » par Maltsiniotis dans \cite{THG} et développée plus avant par Cisinski, notamment dans sa thèse \cite{TheseCisinski} puis dans \cite{PMTH}. C'est à Cisinski que l'on doit la démonstration de deux conjectures fondamentales de Grothendieck en ce domaine : la \emph{minimalité} du localisateur fondamental « classique » de $\Cat$, c'est-à-dire de la classe des morphismes de $\Cat$ dont le nerf est une équivalence faible simpliciale, et l'existence, \emph{pour essentiellement tout localisateur fondamental}\footnote{La seule hypothèse, anodine, est de nature ensembliste.}, d'une structure de catégorie de modèles sur $\Cat$ dont le localisateur fondamental considéré constitue précisément la classe des équivalences faibles, généralisant le résultat de Thomason \cite{Thomason}. Ces structures sont obtenues par « transfert » à partir de structures de catégories de modèles sur la catégorie des ensembles simpliciaux. 

De nombreuses structures catégoriques supérieures et simpliciales sont désormais utilisées tant en théorie de l'homotopie proprement dite qu'en géométrie. Nous introduisons dans cet article la notion de \emph{localisateur fondamental de $\DeuxCat$}, classe de \deux{}foncteurs vérifiant des propriétés formelles analogues à celles de l'axiomatique des localisateurs fondamentaux de Grothendieck, que nous appellerons désormais \emph{localisateurs fondamentaux de $\Cat$}. À tout localisateur fondamental de $\DeuxCat$ devraient se trouver attachées, entre autres, des notions de \deux{}catégorie test, \deux{}foncteur propre et \deux{}foncteur lisse, qui n'ont pas encore été suffisamment étudiées pour que nous les présentions ici. Plus généralement, bien sûr, on peut espérer développer à terme une théorie de l'homotopie « à la Grothendieck » des $n$\nobreakdash-catégories pour $n$ quelconque. Signalons qu'en dépit des obstacles conceptuels, une avancée dans cette direction a été réalisée récemment par Ara et Maltsiniotis, qui dégagent dans \cite{AraMaltsiniotis} un petit nombre de conditions à vérifier pour établir l'existence d'une structure de catégorie de modèles « à la Thomason » sur la catégorie des $n$\nobreakdash-catégories et $n$\nobreakdash-foncteurs stricts. Ils démontrent ces conditions dans le cas $n = 2$, et donc l'existence d'une structure de catégorie de modèles sur $\DeuxCat$ dont les équivalences faibles sont les \DeuxFoncteursStricts{} dont le nerf est une équivalence faible simpliciale. Dans le cas $n = 1$, ils retrouvent le résultat de Thomason qu'avait généralisé Cisinski. 

À partir de la forme absolue de la généralisation du Théorème A de Quillen aux \DeuxFoncteursStricts{} démontrée par Bullejos et Cegarra \cite{BC}, on en obtient facilement une version relative, ce qui permet de dégager une notion de localisateur fondamental de $\DeuxCat$, point de départ d'une théorie de l'homotopie des \deux{}catégories\footnote{Signalons une fois pour toutes que, par « \deux{}catégorie », nous entendons « \deux{}catégorie stricte ».} généralisant les notions et résultats de Grothendieck et Cisinski. Nous consacrons la plus grande partie du présent article à l'étude des relations entre localisateurs fondamentaux de $\Cat$ et de $\DeuxCat$. On explicite notamment un isomorphisme remarquable entre la classe ordonnée par inclusion des localisateurs fondamentaux de $\Cat$ et la classe ordonnée par inclusion des localisateurs fondamentaux de $\DeuxCat$, cet isomorphisme ayant de plus la propriété d'induire des équivalences de catégories entre les catégories localisées de~$\Cat$ et de $\DeuxCat$. Pour établir ce résultat, on définit également la notion de \emph{localisateur fondamental de $\DeuxCatLax$} (la catégorie dont les objets sont les petites \deux{}catégories et dont les morphismes sont les \DeuxFoncteursLax{}), les propriétés homotopiques d'un adjoint à gauche de l'inclusion canonique $\DeuxCat \hookrightarrow \DeuxCatLax$ construit par Bénabou permettant de passer de $\Cat$ à $\DeuxCat$ et réciproquement par l'intermédiaire de $\DeuxCatLax$. Cela confirme l'importance des morphismes non stricts en théorie de l'homotopie des catégories supérieures, importance visible également dans \cite{AraMaltsiniotis} et que l'on peut déjà percevoir dans \cite{BC}, \cite{Cegarra} ou \cite{WHPT} (bien que la démonstration de l'énoncé principal de ce dernier article soit fausse).

Dans \cite{LFM}, Cisinski démontre notamment la conjecture de Grothendieck affirmant que l'intersection de tous les localisateurs fondamentaux de $\Cat$, le \emph{localisateur fondamental minimal de~$\Cat$}, n'est autre que la classe des foncteurs entre petites catégories dont le nerf est une équivalence faible simpliciale. Nos résultats permettent d'en déduire un analogue \deux{}dimensionnel : l'intersection de tous les localisateurs fondamentaux de $\DeuxCat$, le \emph{localisateur fondamental minimal de $\DeuxCat$} n'est autre que la classe des \DeuxFoncteursStricts{} dont le nerf est une équivalence faible simpliciale. En particulier, de même que le résultat de Cisinski fournit une caractérisation purement catégorique des équivalences faibles de $\Cat$, sans faire appel aux espaces topologiques ou aux ensembles simpliciaux, les nôtres fournissent une caractérisation purement \deux{}catégorique des \deux{}foncteurs dont le nerf est une équivalence faible simpliciale. Pour un aperçu de la profondeur de cette propriété de minimalité, on pourra se reporter à \cite{CisinskiKan}. 

Le présent texte fait suite à l'article \cite{ArticleThAMoi}, dont les résultats permettent notamment d'affirmer que les \DeuxFoncteursStricts{} dont le nerf est une équivalence faible simpliciale forment un localisateur fondamental de $\DeuxCat$. Même si l'objectif principal de \cite{ArticleThAMoi} consistait à démontrer la version \deux{}catégorique la plus générale possible du Théorème A de Quillen, et que la notion de localisateur fondamental n'y est mentionnée que dans l'introduction, nous conseillons au lecteur de parcourir~\cite{ArticleThAMoi}, dont nous reprenons ici certains éléments. Indépendamment de cela, le présent texte adopte un point de vue qu'il sera probablement plus facile d'appréhender après avoir parcouru, sinon lu, non seulement \cite{ArticleThAMoi}, mais aussi \cite{LFM} et \cite{THG}, en attendant la publication de \cite{Poursuite}. De plus, le présent article généralise les résultats de \cite{ArticleThAMoi} au cas de localisateurs fondamentaux arbitraires de $\DeuxCat$. 

Après cette introduction, nous rappelons dans la deuxième section des résultats de Quillen, Grothendieck et Cisinski, concernant tous la dimension $1$. 

Dans la troisième section, nous rappelons certaines des notions \deux{}catégoriques introduites dans~\cite{ArticleThAMoi}. On les complète par deux constructions duales d'intégration de \deux{}foncteurs et l'on en souligne une propriété homotopique. Nous concluons la section par un exposé de divers foncteurs nerfs, Carrasco, Cegarra et Garz\'on ayant montré dans \cite{CCG} qu'ils sont tous homotopiquement équivalents. 

Dans la quatrième section, nous étudions les premières propriétés des classes de \DeuxFoncteursStricts{} obtenues à partir d'un localisateur fondamental de $\Cat$ par image réciproque du foncteur «~catégorie des éléments du nerf ». On montrera plus loin que tous les localisateurs fondamentaux de $\DeuxCat$ s'obtiennent ainsi. 

La cinquième section pose la définition de la notion de localisateur fondamental de $\DeuxCat$. C'est une classe de \DeuxFoncteursStricts{} vérifiant un petit nombre de propriétés dont la plus importante est une version \deux{}catégorique du Théorème A de Quillen. Nous dégageons les premières conséquences de la définition, notamment l'invariance par dualité, propriété fondamentale déjà non triviale dans le cas de $\Cat$, dont l'on adapte la démonstration dans celui, plus subtil, de $\DeuxCat$.

Dans la sixième section, nous explicitons l'isomorphisme annoncé entre la classe ordonnée par inclusion des localisateurs fondamentaux de $\Cat$ et la classe ordonnée par inclusion des localisateurs fondamentaux de $\DeuxCat$. Il est compatible à l'opération de localisation. Nous en déduisons une caractérisation purement interne à $\DeuxCat$ de la classe des \DeuxFoncteursStricts{} dont le nerf est une équivalence faible simpliciale : c'est le plus petit localisateur fondamental de $\DeuxCat$.

Nous utilisons les résultats obtenus pour en déduire, dans la septième section, un analogue \deux{}catégorique d'une caractérisation, due à Cisinski, des équivalences faibles classiques de $\Cat$ à l'aide du Théorème B de Quillen. 

Dans un article \cite{Ara} faisant suite au présent texte, Dimitri Ara explique comment nos résultats permettent de déduire de ceux obtenus par lui-même et Maltsiniotis dans \cite{AraMaltsiniotis} et de la théorie développée par Cisinski dans~\cite{PMTH} qu'à tout localisateur fondamental (satisfaisant une condition ensembliste anodine) $\DeuxLocFond{W}$ de $\DeuxCat$ est associée une structure de catégorie de modèles de Quillen sur $\DeuxCat$ dont la classe des équivalences faibles est exactement $\DeuxLocFond{W}$. Il montre que les structures de catégorie de modèles « à la Thomason » ainsi obtenues modélisent exactement les localisations de Bousfield à gauche combinatoires de la structure de catégorie de modèles classique sur les ensembles simpliciaux. Nous renvoyons à \cite{Ara} pour plus de détails.

%C'est Georges Maltsiniotis qui nous a proposé de développer une théorie de l'homotopie des \deux{}catégories analogue à celle de Grothendieck pour les catégories. Au cours de notre travail subséquent, de nombreuses discussions avec Dimitri Ara nous ont été très utiles. Les remarques du rapporteur du présent article nous ont permis d'en améliorer la présentation.

\begin{nota}
Nous notons $\Cat$ la catégorie des petites catégories et $CAT$ la catégorie des catégories (pas forcément petites). Pour toute petite catégorie $A$, $\widehat{A}$ désigne la catégorie des préfaisceaux d'ensembles sur $A$. On notera $\Delta$ la catégorie des simplexes, et donc $\EnsSimp$ la catégorie des ensembles simpliciaux. On note $[m]$ la catégorie associée à l'ensemble ordonné naturellement $\{ 0, \dots, m \}$ pour un entier $m \geq 0$ et $\Delta_{m}$ l'image de cette catégorie, objet de $\Delta$, par le plongement de Yoneda $\Delta \hookrightarrow \EnsSimp$. On notera $\UnNerf : \Cat \to \EnsSimp$ le foncteur nerf « classique ». Pour toute petite catégorie $A$, $\Objets{A}$ (resp. $\UnCell{A}$) désigne les objets (resp. les morphismes) de $A$. On notera $\UnCatPonct$ la catégorie ponctuelle, n'ayant qu'un seul objet et qu'un seul morphisme (l'identité de l'unique objet). On la confondra, dans les notations, avec la \deux{}catégorie ponctuelle, n'ayant qu'un seul objet, qu'une seule \un{}cellule et qu'une seule \deux{}cellule.
\end{nota}

\section{Localisateurs fondamentaux de $\tCat$}

\begin{df}\label{DefWTop}
Une application continue $f : X \to Y$ entre espaces topologiques est une \emph{équivalence faible topologique}, ou plus simplement une \emph{équivalence faible}, si elle induit une bi\-jection au niveau des $\pi_{0}$ et des isomorphismes entre les groupes d'homotopie pour tout choix de point base. Plus précisément, $f$ est une équivalence faible si
$$
\pi_{0}(f) : \pi_{0}(X) \to \pi_{0}(Y)
$$
est une bijection et, pour tout point $x$ de $X$ et tout entier $n \geq 1$,
$$
\pi_{n}(f,x) : \pi_{n}(X,x) \to \pi_{n}(Y,f(x))
$$
est un isomorphisme de groupes.
\end{df}

\begin{df}\label{DefWEnsSimp}
Un morphisme d'ensembles simpliciaux est une \emph{équivalence faible simpliciale}, ou plus simplement une \emph{équivalence faible}, si son image par le foncteur de réalisation géométrique $\vert \bullet \vert : \EnsSimp \to Top$ est une équivalence faible topologique. On notera $\EquiQuillen$ la classe des équivalences faibles simpliciales. 
\end{df}

%\begin{rem}
%(Hors-sujet.) Une structure de catégorie de modèles étant déterminée par ses cofibrations et ses objets fibrants, on pourrait, en vertu des résultats sur le sujet, définir la classe des équivalences faibles simpliciales comme la classe des équivalences faibles de l'unique structure de catégorie de modèles sur $\EnsSimp$ dont les cofibrations sont les monomorphismes — qui sont les morphismes simpliciaux injectifs argument par argument — et les objets fibrants les complexes de Kan.  
%\end{rem}

%\begin{lemme}\label{WQuillenStableSomme}
%Une petite somme d'équivalences faibles simpliciales est une équivalence faible simpliciale.
%\end{lemme}
%
%\begin{proof}
%Grothendieck me l'a dit dans l'ascenseur.   
%\end{proof}

\begin{df}
Un foncteur entre petites catégories est une \emph{équivalence faible catégorique}, ou plus simplement une \emph{équivalence faible}, si son image par le foncteur nerf $\UnNerf : \Cat \to \EnsSimp$ est une équivalence faible simpliciale. On note $\UnLocFondMin$ la classe des équivalences faibles de $\Cat$. 
\end{df}

%\begin{paragr}\label{FoncteurDiagonal}
%Pour toute petite catégorie $A$, le foncteur diagonal 
%$$
%\begin{aligned}
%\delta_{A} : A &\to A \times A
%\\
%a &\mapsto (a,a)
%\end{aligned}
%$$
%induit un foncteur
%$$
%\begin{aligned}
%\delta_{A}^{*} : \widehat{A \times A} &\to \widehat{A}
%\\
%X &\mapsto (a \mapsto X_{a,a})
%\end{aligned}
%$$
%\end{paragr}

%La proposition \ref{LemmeBisimplicialDelta} stipule qu'un morphisme d'ensembles bisimpliciaux qui est une équivalence faible simpliciale « sur les colonnes » ou « sur les lignes » en est une « sur la diagonale ». Pour une démonstration de ce résultat folklorique mais non-trivial, le lecteur pourra se reporter à \cite[p. 94-95]{QuillenK}. 
%  
%\begin{prop}\label{LemmeBisimplicialDelta}
%Soit $f : X \to Y$ un morphisme d'ensembles bisimpliciaux  tel que, pour tout entier $n \geq 0$, le morphisme d'ensembles simpliciaux $f_{n,\bullet} : X_{n,\bullet} \to Y_{n,\bullet}$ (resp. $f_{\bullet,n} : X_{\bullet,n} \to Y_{\bullet,n}$) soit une équivalence faible simpliciale. Alors $\delta_{\Delta}^{*}(f)$ est une équivalence faible simpliciale.   
%\end{prop}

\begin{paragr}
Pour tout préfaisceau d'ensembles $X$ sur une petite catégorie $A$, on note $A/X$ la \emph{catégorie des éléments} de $X$. Ses objets sont les couples $(a, x)$, avec $a$ un objet de $A$ et $x$ un objet de $X(a)$. En vertu du lemme de Yoneda, l'on peut donc considérer $x$ comme un morphisme $a \to X$ de préfaisceaux sur $A$. Les morphismes de $(a,x)$ vers $(a',x')$ sont les morphismes $f : a \to a'$ de $A$ tels que $x' f = x$. Cela permet de définir un foncteur
$$
\begin{aligned}
i_{A} : \widehat{A} &\to \Cat
\\
X &\mapsto A / X
\end{aligned}
$$ 
pour toute petite catégorie $A$.
%Pour tout ensemble simplicial $X$, on note $\Delta / X$ la « catégorie des éléments » de $X$. Ses objets sont donnés par les couples $([m], x)$ avec $m \geq 0$ un entier et $x$ un $m$\nobreakdash-simplexe de $X$. Les morphismes d'un objet $([m], x)$ vers un objet $([n], y)$ sont les applications simpliciales $\varphi : [m] \to [n]$ telles que $y \varphi = x$. Cela définit un foncteur
%$$
%\begin{aligned}
%i_{\Delta} : \EnsSimp &\to \Cat
%\\
%X &\mapsto \Delta / X
%\end{aligned}
%$$ 
\end{paragr}

Le théorème \ref{EquiEnsSimpCat} est attribué à Quillen par Illusie dans \cite{Illusie}. On pourra consulter \cite[volume 2, chapitre 6, section 3]{Illusie}, et plus précisément \cite[volume 2, chapitre 6, section 3, corollaire 3.3.1]{Illusie}.

%\begin{prop}[Quillen]\label{MinouDrouet}
%Il existe un diagramme
%$$
%\UnNerf i_{\Delta} \Longleftarrow K \Longrightarrow 1_{\EnsSimp}
%$$
%dont les flèches sont des morphismes d'endofoncteurs de $\EnsSimp$ qui sont des équivalences faibles simpliciales argument par argument.
%\end{prop}

\begin{theo}[Quillen]\label{EquiEnsSimpCat}
On a l'égalité 
$$
\EquiQuillen = i_{\Delta}^{-1}\UnLocFondMin.
$$
De plus, les foncteurs nerf $\UnNerf : \Cat \to \EnsSimp$ et catégorie des éléments $i_{\Delta} : \EnsSimp \to \Cat$ induisent des équivalences de catégories quasi-inverses l'une de l'autre
$$
\overline{\UnNerf} : \Localisation{\Cat}{\UnLocFondMin} \to \Localisation{\EnsSimp}{\EquiQuillen}
$$
et
$$
\overline{i_{\Delta}} : \Localisation{\EnsSimp}{\EquiQuillen} \to \Localisation{\Cat}{\UnLocFondMin}.
$$
\end{theo}

\begin{df}\label{DefSaturationFaible}
Soit $C$ une petite catégorie. Une classe $S \subset \UnCell{C}$ est dite \emph{faiblement saturée} si elle vérifie les conditions suivantes.
\begin{itemize}
\item[FS1] Les identités des objets de $C$ sont dans $S$.
\item[FS2] Si deux des trois flèches d'un triangle commutatif sont dans $S$, alors la troisième l'est aussi.
\item[FS3] Si $i : X \to Y$ et $r : Y \to X$ sont des morphismes de $C$ vérifiant $ri = 1_{X}$ et si $ir$ est dans $S$, alors il en est de même de $r$ (et donc aussi de $i$ en vertu de ce qui précède). 
\end{itemize}
\end{df}

\begin{rem}
On appellera souvent la condition FS2 propriété de « 2 sur 3 ». 
\end{rem}

Pour tout morphisme $u : A \to B$ de $\Cat$ et tout objet $b$ de $B$, nous noterons $A/b$ la catégorie dont les objets sont les couples $(a, p : u(a) \to b)$, avec $a$ un objet de $A$ et $p$ un morphisme de $B$, et dont les morphismes de $(a,p)$ vers $(a',p')$ sont les morphismes $f : a \to a'$ de $A$ tels que $p' u(f) = p$. 

Pour tout diagramme commutatif
$$
\xymatrix{
A 
\ar[rr]^{u}
\ar[dr]_{w}
&&
B
\ar[dl]^{v}
\\
&
C
}
$$
dans $\Cat$ et tout objet $c$ de $C$, on notera $u/c$ le foncteur défini par
$$
\begin{aligned}
A/c &\to B/c
\\
(a, p) &\mapsto (u(a), p)
\\
f &\mapsto u(f).
\end{aligned}
$$

\begin{df}[Grothendieck]\label{DefUnLocFond}
Un \emph{\ClasseUnLocFond{}} est une classe $\UnLocFond{W}$ de morphismes de $\Cat$ vérifiant les conditions suivantes.
\begin{itemize}
\item[LA] La partie $\UnLocFond{W}$ de $\UnCell{\Cat}$ est faiblement saturée.
\item[LB] Si $A$ est une petite catégorie admettant un objet final, alors le morphisme canonique $A \to \UnCatPonct$ est dans $\UnLocFond{W} $.
\item[LC] Si
$$
\xymatrix{
A 
\ar[rr]^{u}
\ar[dr]_{w}
&&
B
\ar[dl]^{v}
\\
&
C
}
$$
désigne un triangle commutatif dans $\Cat$ et si, pour tout objet $c$ de $C$, le foncteur $u/c$ est dans $\UnLocFond{W}$, alors $u$ est dans $\UnLocFond{W}$. 
\end{itemize}
\end{df}

\begin{exemple}\label{UnLocFondMinUnLocFond}
La classe des équivalences faibles topologiques est faiblement saturée. Par fonctorialité, $\UnLocFondMin$ est donc faiblement saturée. En vertu de \cite[p. 84, corollaire 2]{QuillenK}, elle vérifie la condition LB. La condition LC n'est rien d'autre, dans ce cas, que la forme relative du Théorème A de Quillen \cite[p. 93, Théorème A]{QuillenK}. La classe $\UnLocFondMin$ est donc un \ClasseUnLocFond{}. C'en est même le paradigme justifiant historiquement l'introduction de cette notion. Pour plus de détails, on pourra se reporter à \cite{Poursuite} ou à l'introduction de \cite{THG}. 
\end{exemple}

\begin{paragr}
La notion de localisateur fondamental de $\Cat$ est stable par intersection. On définit le \emph{localisateur fondamental minimal de $\Cat$} comme l'intersection de tous les \ClassesUnLocFond{}. Le théorème \ref{CisinskiGrothendieck} a été conjecturé par Grothendieck. C'est \cite[théorème 2.2.11]{LFM}. 
\end{paragr}

\begin{theo}[Cisinski]\label{CisinskiGrothendieck}
Le localisateur fondamental minimal de $\Cat$ est $\UnLocFondMin$.
\end{theo}

\begin{paragr}
Pour toute petite catégorie $A$, pour tout foncteur $u : A \to \Cat$, on note $\DeuxInt{A}u$ la catégorie (opfibrée sur $A$) dont les objets sont les couples $(a, x)$, avec $a$ un objet de $A$ et $x$ un objet de $u(a)$, et dont les morphismes de $(a,x)$ vers $(a',x')$ sont les couples $(f : a \to a', r : F(f)(x) \to x')$, avec $f$ un morphisme de $A$ et $r$ un morphisme de $F(a')$, les unités et compositions étant définies de façon évidente. On peut étendre de façon naturelle cette construction aux morphismes de foncteurs : pour tous foncteurs $u$ et $v$ de $A$ dans $\Cat$, pour tout morphisme de foncteurs $\sigma : u \Rightarrow v$, on construit un foncteur $\DeuxInt{A} \sigma : \DeuxInt{A}u \to \DeuxInt{A}v$. On notera que le foncteur $i_{A} : \widehat{A} \to \Cat$ n'est rien d'autre qu'une forme duale de cette construction, en considérant la catégorie des ensembles comme une sous-catégorie de $\Cat$. 
\end{paragr}

\begin{paragr}
Soient $I$ et $J$ deux petites catégories et $F(\bullet, \bullet) : I \times J \to Cat$ un foncteur. On peut considérer les foncteurs 
$$
\begin{aligned}
J &\to \Cat
\\
j &\mapsto \DeuxInt{I}F(\bullet, j)
\end{aligned}
$$
et
$$
\begin{aligned}
I &\to \Cat
\\
i &\mapsto \DeuxInt{J}F(i, \bullet).
\end{aligned}
$$
Le lemme \ref{Fubini} est immédiat. 
\end{paragr}

\begin{lemme}\label{Fubini}
Soient $I$ et $J$ deux petites catégories et $F : I \times J \to Cat$ un foncteur. On a des isomorphismes canoniques
$$
\DeuxInt{I \times J} F(\bullet, \bullet) \simeq \DeuxInt{I} \left(i \mapsto \DeuxInt{J} F(i, \bullet)\right) \simeq \DeuxInt{J} \left(j \mapsto \DeuxInt{I} F(\bullet, j)\right).
$$
\end{lemme}

\medbreak

\emph{Jusqu'à la fin de cette section, on suppose fixé un localisateur fondamental $\UnLocFond{W}$ de $\Cat$.} On appellera \emph{$\UnLocFond{W}$\nobreakdash-équivalences faibles}, ou plus simplement \emph{équivalences faibles}, les éléments de $\UnLocFond{W}$. 

\begin{prop}\label{IntegrationWParArguments}
Soient $A$ une petite catégorie, $u$ et $v$ deux foncteurs de $A$ vers $\Cat$ et $\sigma : u \Rightarrow v$ un morphisme de foncteurs. Supposons que, pour tout objet $a$ de $A$, $\sigma_{a} : u(a) \to v(a)$ soit une équivalence faible. Alors $\DeuxInt{A}\sigma : \DeuxInt{A}u \to \DeuxInt{A}v$ est une équivalence faible. 
\end{prop}

\begin{proof}
C'est, par exemple, \cite[proposition 2.3.1]{THG}.
\end{proof}

\begin{df}\label{DfWA}
On dit qu'un morphisme de préfaisceaux sur une petite catégorie $A$ est une \emph{équivalence faible} (de préfaisceaux) si son image par le foncteur $i_{A}$ est une équivalence faible (de $\Cat$). On notera $W_{\widehat{A}}$ la classe des équivalences faibles de préfaisceaux sur $A$. 
\end{df}

\begin{lemme}\label{Queneau}
Soient $A$ et $B$ deux petites catégories et $f(\bullet, \bullet) : X(\bullet, \bullet) \to Y(\bullet, \bullet)$ un morphisme de ${\widehat{A \times B}}$. Supposons que, pour tout objet $a$ de $A$, le morphisme $f(a,\bullet) : X(a,\bullet) \to Y(a,\bullet)$ soit dans $\UnLocFond{W}_{\widehat{B}}$. Alors, $f(\bullet, \bullet)$ est dans $\UnLocFond{W}_{\widehat{A \times B}}$. 
\end{lemme}

\begin{proof}
Supposons que, pour tout objet $a$ de $A$, le morphisme $f(a, \bullet) : X(a,\bullet) \to Y(a,\bullet)$ soit dans $\UnLocFond{W}_{\widehat{B}}$, c'est-à-dire que le foncteur $i_{B} (f(a, \bullet))$ soit dans $\UnLocFond{W}$. Il s'agit de montrer que $f(\bullet, \bullet)$ est dans $\UnLocFond{W}_{\widehat{A \times B}}$, c'est-à-dire que la flèche $i_{A \times B}(f(\bullet, \bullet))$ est dans $\UnLocFond{W}$. Pour cela, il suffit d'invoquer le lemme \ref{Fubini} (ou plutôt l'énoncé dual) qui permet d'identifier
$
i_{A \times B} (f(\bullet, \bullet)) 
$
à
$
i_{A} (i_{B} (f(a,\bullet)))
$
et de constater que le dernier terme est bien dans $\UnLocFond{W}$ en vertu des hypothèses et de la proposition \ref{IntegrationWParArguments}, l'assignation 
$
a \mapsto i_{B} (f(a,\bullet)) 
$
définissant un morphisme de foncteurs de
$
A \to \Cat
,
a \mapsto i_{B}X(a,\bullet)
$
vers
$
A \to \Cat
,
a \mapsto i_{B}Y(a,\bullet)
$.
\end{proof}

\begin{df}\label{DfCatAspherique}
On dit qu'une petite catégorie $A$ est \emph{$\UnLocFond{W}$\nobreakdash-asphérique}, ou plus simplement \emph{asphérique}, si le foncteur canonique $A \to e$ est une équivalence faible. Un morphisme $u : A \to B$ de $\Cat$ est \emph{\UnLocFond{W}\nobreakdash-asphérique}, ou plus simplement \emph{asphérique} si, pour tout objet $b$ de $B$, la catégorie $A/b$ est asphérique. 
\end{df}

\begin{rem}
Un foncteur asphérique est donc une équivalence faible.
\end{rem}

Tout morphisme $u : A \to B$ de $\Cat$ en induit un autre, que l'on notera $u^{*}$, défini par
$$
\begin{aligned}
u^{*} : \widehat{B} &\to \widehat{A}
\\
X &\mapsto (a \mapsto X(u(a))).
\end{aligned}
$$ 

\begin{prop}\label{CasParticulier1.2.9.THG}
Soient $A$ et $B$ deux petites catégories asphériques et $u : A \to B$ un morphisme de $\Cat$. Les deux conditions suivantes sont équivalentes.
\begin{itemize}
\item[(i)]
Le foncteur $u$ est asphérique.
\item[(ii)] 
Si $\varphi$ est une équivalence faible de $\widehat{B}$, alors $u^{*}(\varphi)$ est une équivalence faible de $\widehat{A}$. 
\end{itemize}
\end{prop}

\begin{proof}
C'est une partie de l'énoncé de \cite[proposition 1.2.9]{THG}. 
\end{proof}

\begin{rem}
À toute notion de localisateur fondamental de $\Cat$ se rattachent non seulement celles d'équivalence faible de $\Cat$ (élément du localisateur fondamental) et de petite catégorie asphérique (définition \ref{DfCatAspherique}), mais également, pour toute petite catégorie $A$, d'équivalence faible dans la catégorie $\widehat{A}$ des préfaisceaux sur $A$ (définition \ref{DfWA}) ainsi que de préfaisceau asphérique sur $A$. Il reste à définir cette dernière notion : on dit qu'un préfaisceau $X$ sur $A$ est asphérique si la catégorie $A/X$ est asphérique. Ainsi, $A$ est asphérique si et seulement si le préfaisceau final sur $A$ est asphérique. Voir notamment \cite[1.2.5 et 1.2.6]{THG} pour des détails. On peut alors compléter la proposition \ref{CasParticulier1.2.9.THG} par l'énoncé suivant : sous les mêmes données, $u$ est asphérique si et seulement si, pour tout objet $b$ de $B$, le préfaisceau $u^{*}(b)$ sur $A$ est asphérique (c'est une autre partie de \cite[proposition 1.2.9]{THG}).
\end{rem}

\begin{df}\label{DefTotalementAspherique}
On dit qu'une petite catégorie $A$ est \emph{totalement $\UnLocFond{W}$\nobreakdash-asphérique}, ou plus simplement \emph{totalement asphérique}, si elle est asphérique et si le foncteur diagonal 
$$
\begin{aligned}
\delta_{A} : A &\to A \times A
\\
a &\mapsto (a,a)
\end{aligned}
$$ 
est asphérique. 
\end{df}

\begin{rem}\label{RemarqueTotalementAspherique}
En conservant les mêmes notations, le foncteur $\delta_{A}$ est asphérique si et seulement si, pour tous objets $a$ et $a'$ de $A$, le produit de préfaisceaux représentables $a \times a'$ est un préfaisceau asphérique. De plus, si $A$ est non vide et que le foncteur $\delta_{A}$ est asphérique, alors $A$ est asphérique (voir la démonstration de \cite[proposition 1.6.1]{THG}).
\end{rem}

\begin{prop}\label{DeltaTotalementAspherique}
La catégorie des simplexes $\Delta$ est totalement asphérique. 
\end{prop}

\begin{proof}
La $\UnLocFondMin$-asphéricité de $\Delta$ est évidente en vertu de la remarque \ref{RemarqueTotalementAspherique} (cela résulte de la stabilité des équivalences faibles simpliciales par produit). On peut donc conclure dans le cas d'un localisateur fondamental quelconque de $\Cat$ en vertu de la minimalité de $\UnLocFondMin$ (théorème \ref{CisinskiGrothendieck}). Pour un argument plus élémentaire, n'utilisant pas cette propriété de minimalité, nous renvoyons le lecteur à la démonstration de \cite[proposition 1.6.13]{THG}.
\end{proof}

\begin{prop}\label{LemmeBisimplicial}
Soient $A$ une petite catégorie totalement asphérique et $f$ un morphisme de $\widehat{A \times A}$. Si, pour tout objet $a$ de $A$, le morphisme $f(a,\bullet)$ est dans $\UnLocFond{W}_{\widehat{A}}$, alors $\delta_{A}^{*}(f)$ est dans $\UnLocFond{W}_{\widehat{A}}$.
\end{prop}

\begin{proof}
En vertu des hypothèses et du lemme \ref{Queneau}, $f$ est dans $\UnLocFond{W}_{\widehat{A \times A}}$. Le foncteur diagonal $\delta_{A} : A \to A \times A$ étant asphérique par hypothèse, le résultat découle de la proposition \ref{CasParticulier1.2.9.THG}.
\end{proof} 

\begin{prop}\label{UnSupAspherique}
Pour toute petite catégorie $A$, le foncteur 
$$
\begin{aligned}
\SupUnObjet{A} : \Delta / \UnNerf{A} &\to A
\\
([m], x) &\mapsto x_{m}
\end{aligned}
$$
est asphérique (donc en particulier une équivalence faible).
\end{prop}

\begin{proof}
C'est \cite[proposition 2.2.3]{LFM}, attribuée à Grothendieck. On vérifie facilement que, pour tout objet $a$ de $A$, la catégorie $(\Delta / \UnNerf{A}) / a$ s'identifie canoniquement à $\Delta / \UnNerf{(A/a)}$. Or, la catégorie $A/a$ admet un objet final et, pour toute petite catégorie $C$ admettant un objet final, la catégorie $\Delta / \UnNerf{C}$ est contractile, comme on peut le montrer en en construisant un endomorphisme constant homotope à l'identité. 
\end{proof}

\begin{prop}\label{Sade}
On a l'égalité
$$
\UnLocFond{W} = \UnNerf^{-1} ({i_{\Delta}}^{-1} (W))
$$
\end{prop}

\begin{proof}
C'est une conséquence immédiate de la proposition \ref{UnSupAspherique}, par un argument de « $2$ sur $3$ ». 
\end{proof}

\section{Formalisme $2$-catégorique}

\begin{paragr}
On suppose connues les notions de \deux{}catégorie, \DeuxFoncteurStrict{}, \DeuxFoncteurLax{} et \DeuxFoncteurCoLax{}\footnote{Parfois qualifié d'« oplax » dans la littérature.}. Notre vocabulaire et nos notations (parfois idiosyncrasiques) seront similaires à ceux de \cite{ArticleThAMoi}. Nous noterons $\DeuxCat$ la catégorie dont les objets sont les petites \deux{}catégories et dont les morphismes sont les \DeuxFoncteursStricts{}. Nous noterons $\DeuxCatLax$ la catégorie dont les objets sont les petites \deux{}catégories et dont les morphismes sont les \DeuxFoncteursLax{}. Étant donné un \DeuxFoncteurLax{} $u : \mathdeuxcat{A} \to \mathdeuxcat{B}$ et $f$ et $f'$ deux \un{}cellules de $\mathdeuxcat{A}$ telles que la composée $f'f$ fasse sens, nous noterons $u_{f',f}$ la \deux{}cellule structurale (« de composition ») $u(f')u(f) \Rightarrow u(f'f)$ et, pour tout objet $a$ de $\mathdeuxcat{A}$, nous noterons $u_{a}$ la \deux{}cellule structurale (« d'unité ») $1_{u(a)} \Rightarrow u(1_{a})$. Le \DeuxFoncteurLax{} $u : \mathdeuxcat{A} \to \mathdeuxcat{B}$ est dit \emph{normalisé} si, pour tout objet $a$ de $\mathdeuxcat{A}$, la \deux{}cellule structurale d'unité $u_{a}$ est une identité (en particulier, $u(1_{a}) = 1_{u(a)}$) et si, pour toute \un{}cellule $f : a \to a'$ de $\mathdeuxcat{A}$, les \deux{}cellules structurales de composition $u_{1_{a'}, f}$ et $u_{f, 1_{a}}$ sont des identités. On notera par un exposant « op » (resp. « co », resp. « coop ») le « changement de sens des \un{}cellules » (resp. des \deux{}cellules, resp. des \un{}cellules et des \deux{}cellules). Pour tout couple d'objets $x$ et $y$ d'une \deux{}catégorie $\mathdeuxcat{A}$, on notera $\CatHom{\mathdeuxcat{A}}{x}{y}$ la catégorie dont les objets sont les \un{}cellules de $\mathdeuxcat{A}$ et dont les morphismes sont les \deux{}cellules de $\mathdeuxcat{A}$.
\end{paragr}

\begin{df} 
Soient $u : \mathdeuxcat{A} \to \mathdeuxcat{B}$ un \DeuxFoncteurLax{} et $b$ un objet de $\mathdeuxcat{B}$. La \emph{\deux{}catégorie comma lax de $\mathdeuxcat{A}$ au-dessus de $b$ relativement à $u$} est la \deux{}catégorie $\TrancheLax{\mathdeuxcat{A}}{u}{b}$ définie comme suit :
\begin{itemize}
\item
Les objets de $\TrancheLax{\mathdeuxcat{A}}{u}{b}$ sont les couples $(a, p : u(a) \to b)$, avec $a$ un objet de $\mathdeuxcat{A}$ et $p$ une \un{}cellule de $\mathdeuxcat{B}$.
\item
Si $(a, p : u(a) \to b)$ et $(a', p' : u(a') \to b)$ sont deux objets de $\TrancheLax{\mathdeuxcat{A}}{u}{b}$, les \un{}cellules de $(a,p)$ vers $(a',p')$ dans $\TrancheLax{\mathdeuxcat{A}}{u}{b}$ sont les couples $(f : a \to a', \alpha : p \Rightarrow p' u(f))$, avec $f$ une \un{}cellule de $\mathdeuxcat{A}$ et $\alpha$ une \deux{}cellule de $\mathdeuxcat{B}$. Le diagramme à conserver en tête est le suivant : 
$$
\xymatrix{
u(a) 
\ar[rr]^{u(f)}
\ar[dr]_{p}
&{}
&u(a')
\ar[dl]^{p'}
\\
&b
\utwocell<\omit>{\alpha}
&.
}
$$
\item 
Si $(f, \alpha)$ et $(f', \alpha')$ sont deux \un{}cellules de $(a,p)$ vers $(a',p')$ dans $\TrancheLax{\mathdeuxcat{A}}{u}{b}$, les \deux{}cellules de $(f, \alpha)$ vers $(f', \alpha')$ dans $\TrancheLax{\mathdeuxcat{A}}{u}{b}$ sont les \deux{}cellules $\beta : f \Rightarrow f'$ dans $\mathdeuxcat{A}$ telles que $(p' \CompDeuxZero u(\beta)) \alpha = \alpha'$.
\item
Les diverses composées et identités sont « évidentes ».
\end{itemize}

Soient $u : \mathdeuxcat{A} \to \mathdeuxcat{B}$ un \DeuxFoncteurLax{} et $b$ un objet de $\mathdeuxcat{B}$. La \emph{\deux{}catégorie opcomma lax de $\mathdeuxcat{A}$ au-dessous de $b$ relativement à $u$} est la \deux{}catégorie $\OpTrancheLax{\mathdeuxcat{A}}{u}{b}$ définie par la formule 
$$
\OpTrancheLax{\mathdeuxcat{A}}{u}{b} = \DeuxCatUnOp{(\TrancheLax{(\DeuxCatUnOp{\mathdeuxcat{A}})}{\DeuxFoncUnOp{u}}{b})}.
$$

Soient $u : \mathdeuxcat{A} \to \mathdeuxcat{B}$ un \DeuxFoncteurCoLax{} et $b$ un objet de $\mathdeuxcat{B}$. La \emph{\deux{}catégorie comma colax de $\mathdeuxcat{A}$ au-dessus de $b$ relativement à $u$} est la \deux{}catégorie $\TrancheCoLax{\mathdeuxcat{A}}{u}{b}$ définie par la formule 
$$
\TrancheCoLax{\mathdeuxcat{A}}{u}{b} = \DeuxCatDeuxOp{(\TrancheLax{(\DeuxCatDeuxOp{\mathdeuxcat{A}})}{\DeuxFoncDeuxOp{u}}{b})}.
$$

Soient $u : \mathdeuxcat{A} \to \mathdeuxcat{B}$ un \DeuxFoncteurCoLax{} et $b$ un objet de $\mathdeuxcat{B}$. La \emph{\deux{}catégorie opcomma colax de $\mathdeuxcat{A}$ au-dessous de $b$ relativement à $u$} est la \deux{}catégorie $\OpTrancheCoLax{\mathdeuxcat{A}}{u}{b}$ définie par la formule 
$$
\OpTrancheCoLax{\mathdeuxcat{A}}{u}{b} = \DeuxCatUnOp{(\TrancheCoLax{(\DeuxCatUnOp{\mathdeuxcat{A}})}{\DeuxFoncUnOp{u}}{b})}.
$$

Si $a$ est un objet de la \deux{}catégorie $\mathdeuxcat{A}$, on notera $\TrancheLax{\mathdeuxcat{A}}{}{a}$ (resp. $\OpTrancheLax{\mathdeuxcat{A}}{}{a}$, resp. $\TrancheCoLax{\mathdeuxcat{A}}{}{a}$, resp. $\OpTrancheCoLax{\mathdeuxcat{A}}{}{a}$) la \deux{}catégorie $\TrancheLax{\mathdeuxcat{A}}{1_{\mathdeuxcat{A}}}{a}$ (resp. $\OpTrancheLax{\mathdeuxcat{A}}{1_{\mathdeuxcat{A}}}{a}$, resp. $\TrancheCoLax{\mathdeuxcat{A}}{1_{\mathdeuxcat{A}}}{a}$, resp. $\OpTrancheCoLax{\mathdeuxcat{A}}{1_{\mathdeuxcat{A}}}{a}$).
\end{df}

\medskip

On rappelle maintenant la définition, introduite dans \cite{ArticleThAMoi}, de la notion d'objet final (resp. initial) d'un objet d'une \deux{}catégorie.

\begin{df}\label{DefOF2}
On dira qu'un objet $z$ d'une \deux{}catégorie $\mathdeuxcat{A}$ \emph{admet un objet final} si, pour tout objet $a$ de $\mathdeuxcat{A}$, la catégorie $\CatHom{\mathdeuxcat{A}}{a}{z}$ admet un objet final. On dira qu'il \emph{admet un objet initial} s'il admet un objet final dans $\DeuxCatDeuxOp{\mathdeuxcat{A}}$, autrement dit si, pour tout objet $a$ de $\mathdeuxcat{A}$, la catégorie $\CatHom{\mathdeuxcat{A}}{a}{z}$ admet un objet initial. 
\end{df}

\begin{exemple}
Un objet de la \deux{}catégorie $\Cat$ admet un objet final (resp. initial) en ce sens si et seulement si c'est une catégorie admettant un objet final (resp. initial) au sens usuel ; c'est la raison de l'adoption de cette terminologie, suggérée par Jean Bénabou.
\end{exemple}

\begin{paragr}\label{DefOpAdmet}
Pour des raisons de commodité, l'on dira qu'une \deux{}catégorie $\mathdeuxcat{A}$ \emph{op\nobreakdash-admet} (resp. \emph{co\nobreakdash-admet}, resp. \emph{coop\nobreakdash-admet}) une certaine propriété si $\DeuxCatUnOp{\mathdeuxcat{A}}$ (resp. $\DeuxCatDeuxOp{\mathdeuxcat{A}}$, resp. $\DeuxCatToutOp{\mathdeuxcat{A}}$) vérifie cette propriété. 
\end{paragr}

\begin{exemple}\label{ExemplesOF}
Pour toute petite \deux{}catégorie $\mathdeuxcat{A}$ et tout objet $a$ de $\mathdeuxcat{A}$, la \deux{}catégorie $\TrancheCoLax{\mathdeuxcat{A}}{}{a}$ admet un objet admettant un objet final. Plus précisément, l'objet $(a, 1_{a})$ est tel que, pour tout objet $(a', p : a' \to a)$, le couple $(p, 1_{p})$ définit un objet final de la catégorie $\CatHom{\TrancheCoLax{\mathdeuxcat{A}}{}{a}}{(a',p)}{(a,1_{a})}$. Dualement, la \deux{}catégorie $\TrancheLax{\mathdeuxcat{A}}{}{a}$ admet un objet admettant un objet initial, la \deux{}catégorie $\OpTrancheCoLax{\mathdeuxcat{A}}{}{a}$ op-admet un objet admettant un objet final et la \deux{}catégorie $\OpTrancheLax{\mathdeuxcat{A}}{}{a}$ op-admet un objet admettant un objet initial. 
\end{exemple}

\begin{paragr}
Nous avons rappelé dans \cite{ArticleThAMoi} la construction explicite d'un adjoint à gauche de l'inclusion $\DeuxCat \hookrightarrow \DeuxCatLax$, cas particulier d'une construction bicatégorique de Bénabou. On notera $\TildeLax{\mathdeuxcat{A}}$ (resp. $\TildeLax{u}$) l'image par ce foncteur d'une \deux{}catégorie $\mathdeuxcat{A}$ (resp. d'un \DeuxFoncteurLax{} $u$). En notant $\eta$ et $\epsilon$ respectivement l'unité et la coünité de cette adjonction, on a donc en particulier, pour toute petite \deux{}catégorie $\mathdeuxcat{A}$, un \DeuxFoncteurLax{} 
$$
\LaxCanonique{\mathdeuxcat{A}} : \mathdeuxcat{A} \to \TildeLax{\mathdeuxcat{A}}
$$
et un \DeuxFoncteurStrict{} 
$$
\StrictCanonique{\mathdeuxcat{A}} : \TildeLax{\mathdeuxcat{A}} \to \mathdeuxcat{A}.
$$
On renvoie à \cite[section 1.12]{TheseMoi} ou \cite[section 5]{ArticleThAMoi} pour une description de cette adjonction. Muni des formules, on vérifie sans difficulté le lemme \ref{LemmeTranchesCouniteAspheriques}. C'est \cite[lemme 5.9]{ArticleThAMoi}. 
\end{paragr}

\begin{lemme}\label{LemmeTranchesCouniteAspheriques}
Pour toute petite \deux{}catégorie $\mathdeuxcat{A}$, pour tout objet $a$ de $\mathdeuxcat{A}$, la \deux{}catégorie $\TrancheCoLax{\TildeLax{\mathdeuxcat{A}}}{\StrictCanonique{\mathdeuxcat{A}}}{a}$ admet un objet admettant un objet final.
\end{lemme}

\begin{paragr}
Étant donné deux \DeuxFoncteursLax{} (ou deux \DeuxFoncteursCoLax{}) $u$ et $v$ de $\mathdeuxcat{A}$ vers $\mathdeuxcat{B}$, une \emph{transformation} $\sigma$ de $u$ vers $v$ correspond à la donnée d'une \un{}cellule $\sigma_{a} : u(a) \to v(a)$ dans $\mathdeuxcat{B}$ pour tout objet $a$ de $\mathdeuxcat{A}$ et d'une \deux{}cellule $\sigma_{f} : \sigma_{a'} u(f) \Rightarrow v(f) \sigma_{a}$ dans $\mathdeuxcat{B}$ pour toute \un{}cellule $f : a \to a'$ dans $\mathdeuxcat{A}$, ces données vérifiant les conditions de cohérence bien connues. On dira que la transformation $\sigma$ est \emph{relative aux objets} si $\sigma_{a} = 1_{u(a)} (= 1_{v(a)})$ pour tout objet $a$ de $\mathdeuxcat{A}$. Avec les mêmes données, une \emph{optransformation} de $u$ vers $v$ est une \DeuxTransformationLax{} de $\DeuxFoncUnOp{v}$ vers $\DeuxFoncUnOp{u}$. De façon plus explicite, cela revient à se donner une \un{}cellule $\sigma_{a} : u(a) \to v(a)$ dans $\mathdeuxcat{B}$ pour tout objet $a$ de $\mathdeuxcat{A}$ et une \deux{}cellule $\sigma_{f} : v(f) \sigma_{a} \Rightarrow \sigma_{a'} u(f)$ dans $\mathdeuxcat{B}$ pour toute \un{}cellule $f : a \to a'$ dans $\mathdeuxcat{A}$, ces données vérifiant les conditions de cohérence aussi bien connues que les précédentes. Si $u$ et $v$ sont stricts, nous appellerons \emph{transformation stricte} de $u$ vers $v$ une \DeuxTransformationLax{} (ou, ce qui revient au même dans ce cas précis, une \DeuxTransformationCoLax{}) $\sigma$ de $u$ vers $v$ telle que la \deux{}cellule $\sigma_{f}$ soit une identité pour toute \un{}cellule $f$ de $\mathdeuxcat{A}$. On notera $\DeuxCatDeuxCat$ la \deux{}catégorie dont la catégorie sous-jacente est $\DeuxCat$ et dont les \deux{}cellules sont les \DeuxTransformationsStrictes{} (c'est en fait la \deux{}catégorie sous-jacente à une 3\nobreakdash-catégorie dont les 3\nobreakdash-cellules sont les modifications).
\end{paragr} 

\begin{df}\label{DefMAdjoints}
On dira qu'un \DeuxFoncteurCoLax{} $u : \mathdeuxcat{A} \to \mathdeuxcat{B}$ est un \emph{préadjoint à gauche colax} si, pour tout objet $b$ de $\mathdeuxcat{B}$, la \deux{}catégorie $\TrancheCoLax{\mathdeuxcat{A}}{u}{b}$ admet un objet admettant un objet final.  

On dira qu'un \DeuxFoncteurLax{} $u : \mathdeuxcat{A} \to \mathdeuxcat{B}$ est un \emph{préadjoint à gauche lax} si $\DeuxFoncDeuxOp{u}$ est un préadjoint à gauche colax. Cette condition équivaut à la suivante : pour tout objet $b$ de $\mathdeuxcat{B}$, la \deux{}catégorie $\TrancheLax{\mathdeuxcat{A}}{u}{b}$ admet un objet admettant un objet initial. 

On dira qu'un \DeuxFoncteurCoLax{} $u : \mathdeuxcat{A} \to \mathdeuxcat{B}$ est un \emph{préadjoint à droite colax} si $\DeuxFoncUnOp{u}$ est un préadjoint à gauche colax. Cette condition équivaut à la suivante : pour tout objet $b$ de $\mathdeuxcat{B}$, la \deux{}catégorie $\OpTrancheCoLax{\mathdeuxcat{A}}{u}{b}$ op-admet un objet admettant un objet final. 

On dira qu'un \DeuxFoncteurLax{} $u : \mathdeuxcat{A} \to \mathdeuxcat{B}$ est un \emph{préadjoint à droite lax} si $\DeuxFoncToutOp{u}$ est un préadjoint à gauche colax. Cette condition équivaut à la suivante : pour tout objet $b$ de $\mathdeuxcat{B}$, la \deux{}catégorie $\OpTrancheLax{\mathdeuxcat{A}}{u}{b}$ op-admet un objet admettant un objet initial.
\end{df}

\begin{rem}
On renvoie à \cite[section 1.6]{TheseMoi} pour une discussion relative aux apparitions de cette notion dans la littérature antérieure.
\end{rem}

\begin{exemple}\label{EquiPreadjoint}
Toute équivalence de \deux{}catégories est un préadjoint à gauche lax, un préadjoint à gauche colax, un préadjoint à droite lax et un préadjoint à droite colax. Pour un rappel de la définition ainsi qu'une démonstration, on pourra se reporter à \cite[paragraphe 1.6.11 et proposition 1.6.12]{TheseMoi}.
\end{exemple}

\begin{rem}
À tout morphisme lax (resp. colax) $\mathdeuxcat{A} \to \mathdeuxcat{B}$ vérifiant l'une des conditions de la définition \ref{DefMAdjoints} est associé de façon canonique un morphisme colax (resp. lax) $\mathdeuxcat{B} \to \mathdeuxcat{A}$. Pour une démonstration, nous renvoyons le lecteur à \cite[1.6.16]{TheseMoi}. 
\end{rem}

\begin{df}\label{DefFibre}
Soient $u : \mathdeuxcat{A} \to \mathdeuxcat{B}$ un \DeuxFoncteurStrict{} et $b$ un objet de $\mathdeuxcat{B}$. On appelle \emph{fibre de} $u$ \emph{au-dessus de} $b$ la \deux{}catégorie, que l'on notera $\Fibre{\mathdeuxcat{A}}{u}{b}$, dont les objets sont les objets $a$ de $\mathdeuxcat{A}$ tels que $u(a) = b$, dont les \un{}cellules de $a$ vers $a'$ sont les \un{}cellules $f $ de $a$ vers $a'$ dans $\mathdeuxcat{A}$ telles que $u(f) = 1_{b}$ et dont les \deux{}cellules de $f$ vers $f'$ sont les \deux{}cellules $\alpha$ de $f$ vers $f'$ dans $\mathdeuxcat{A}$ telles que $u(\alpha) = 1_{1_{b}}$, les diverses compositions et unités provenant de celles de $\mathdeuxcat{A}$ de façon « évidente ».
\end{df}

\begin{rem}\label{FibresCoOp}
La \deux{}catégorie $\Fibre{(\DeuxCatUnOp{\mathdeuxcat{A}})}{(\DeuxFoncUnOp{u})}{b}$ (resp. $\Fibre{(\DeuxCatDeuxOp{\mathdeuxcat{A}})}{(\DeuxFoncDeuxOp{u})}{b}$, resp. $\Fibre{(\DeuxCatToutOp{\mathdeuxcat{A}})}{(\DeuxFoncToutOp{u})}{b}$) s'identifie canoniquement à $\DeuxCatUnOp{(\Fibre{\mathdeuxcat{A}}{u}{b})}$ (resp. $\DeuxCatDeuxOp{(\Fibre{\mathdeuxcat{A}}{u}{b})}$, resp. $\DeuxCatToutOp{(\Fibre{\mathdeuxcat{A}}{u}{b})}$).
\end{rem} 

\begin{df}\label{DefPrefibration}
On dira qu'un morphisme $u : \mathdeuxcat{A} \to \mathdeuxcat{B}$ de $\DeuxCat$ est une \emph{préfibration} si, pour tout objet $b$ de $\mathdeuxcat{B}$, le \DeuxFoncteurStrict{} canonique 
$$
\begin{aligned}
J_{b} : \Fibre{\mathdeuxcat{A}}{u}{b} &\to \OpTrancheCoLax{\mathdeuxcat{A}}{u}{b}
\\
a &\mapsto (a, 1_{b})
\\
f &\mapsto (f, 1_{1_{b}})
\\
\alpha &\mapsto \alpha
\end{aligned}
$$
est un préadjoint à gauche lax.

On dira qu'un  \DeuxFoncteurStrict{} $u : \mathdeuxcat{A} \to \mathdeuxcat{B}$ est une \emph{préopfibration} si $\DeuxFoncUnOp{u}$ est une préfibration, autrement dit si, pour tout objet $b$ de $\mathdeuxcat{B}$, le morphisme canonique $\Fibre{\mathdeuxcat{A}}{u}{b} \to \TrancheCoLax{\mathdeuxcat{A}}{u}{b}$ est un préadjoint à droite lax. 

On dira qu'un  \DeuxFoncteurStrict{} $u : \mathdeuxcat{A} \to \mathdeuxcat{B}$ est une \emph{précofibration} si $\DeuxFoncDeuxOp{u}$ est une préfibration, autrement dit si, pour tout objet $b$ de $\mathdeuxcat{B}$, le morphisme canonique $\Fibre{\mathdeuxcat{A}}{u}{b} \to \OpTrancheLax{\mathdeuxcat{A}}{u}{b}$ est un préadjoint à gauche colax. 

On dira qu'un  \DeuxFoncteurStrict{} $u : \mathdeuxcat{A} \to \mathdeuxcat{B}$ est une \emph{précoopfibration} si $\DeuxFoncToutOp{u}$ est une préfibration, autrement dit si, pour tout objet $b$ de $\mathdeuxcat{B}$, le morphisme canonique $\Fibre{\mathdeuxcat{A}}{u}{b} \to \TrancheLax{\mathdeuxcat{A}}{u}{b}$ est un préadjoint à droite colax.   
\end{df}

\begin{paragr}\label{DefDeuxInt}
On rappelle noter $\DeuxCatDeuxCat$ la \deux{}catégorie dont la catégorie sous-jacente est $\DeuxCat$ et dont les \deux{}cellules sont les \DeuxTransformationsStrictes{}. On a déjà défini dans \cite[section 3]{ArticleThAMoi} (voir \cite[section 1.10]{TheseMoi} pour davantage de détails), pour tout \DeuxFoncteurStrict{} $F : \mathdeuxcat{A} \to \DeuxCatDeuxCat$, une \deux{}catégorie $\DeuxInt{\mathdeuxcat{A}}F$ comme suit. 

Les objets de $\DeuxInt{\mathdeuxcat{A}}F$ sont les couples $(a, x)$, avec $a$ un objet de $\mathdeuxcat{A}$ et $x$ un objet de la \deux{}catégorie $F(a)$. 

Les \un{}cellules de $(a, x)$ vers $(a', x')$ dans $\DeuxInt{\mathdeuxcat{A}}F$ sont les couples
$$
(f : a \to a', r : F(f)(x) \to x')
$$
dans lesquels $f$ est une \un{}cellule de $a$ vers $a'$ dans $\mathdeuxcat{A}$ et $r$ une \un{}cellule de $F(f)(x)$ vers $x'$ dans $F(a')$. 

Les \deux{}cellules de $(f : a \to a', r : F(f)(x) \to x')$ vers $(g : a \to a', s : F(g)(x) \to x')$ dans $\DeuxInt{\mathdeuxcat{A}}F$ sont les couples
$$
(\gamma : f \Rightarrow g, \varphi : r \Rightarrow s (F(\gamma))_{x})
$$
dans lesquels $\gamma$ est une \deux{}cellule de $f$ vers $g$ dans $\mathdeuxcat{A}$ et $\varphi$ une \deux{}cellule de $r$ vers $s (F(\gamma))_{x}$ dans $F(a')$.

Les diverses unités et compositions sont définies de façon « évidente ». 
\end{paragr}

\begin{paragr}\label{DefDeuxIntOp}
De façon duale, pour tout \DeuxFoncteurStrict{} $F : \DeuxCatUnOp{\mathdeuxcat{A}} \to \DeuxCatDeuxCat$, on définit une \deux{}catégorie $\DeuxIntOp{\mathdeuxcat{A}} F$ par la formule
$$
\DeuxIntOp{\mathdeuxcat{A}} F = \DeuxCatToutOp{\left( \DeuxInt{\DeuxCatToutOp{\mathdeuxcat{A}}} \DeuxFoncDeuxOp{(?^{coop} \circ F)} \right)}.
$$
En particulier, les objets sont les couples $(a,x)$, avec $a$ un objet de $\mathdeuxcat{A}$ et $x$ un objet de $F(a)$. Les \un{}cellules de $(a,x)$ vers $(a',x')$ sont les couples $(f : a \to a', r : x \to F(f)(x'))$, avec $f$ une \un{}cellule de $\mathdeuxcat{A}$ et $r$ une \un{}cellule de $F(a)$. Les \deux{}cellules de $(f,r)$ vers $(g,s)$ sont les couples $(\gamma : f \Rightarrow g, \varphi : (F(\gamma))_{x'} r \Rightarrow s)$, avec $\gamma$ une \deux{}cellule de $\mathdeuxcat{A}$ et $\varphi$ une \deux{}cellule de $F(a)$.
\end{paragr}

\begin{rem}
En considérant la catégorie des ensembles comme une sous-catégorie de $\Cat$, on peut considérer la restriction de $\DeuxIntOp{\mathdeuxcat{A}}$ à la catégorie $\widehat{A}$. Ce n'est rien d'autre que le foncteur $i_{A}$.  
\end{rem}

\begin{paragr}\label{DefDeuxIntCo}
De façon duale, pour tout \DeuxFoncteurStrict{} $F : \DeuxCatDeuxOp{\mathdeuxcat{A}} \to \DeuxCatDeuxCat$, on définit une \deux{}catégorie $\DeuxIntCo{\mathdeuxcat{A}} F$ par la formule
$$
\DeuxIntCo{\mathdeuxcat{A}} F = \DeuxCatDeuxOp{\left(\DeuxInt{\DeuxCatDeuxOp{\mathdeuxcat{A}}} \left(\DeuxFoncDeuxOp{?} \circ F\right)\right)}.
$$ 
En particulier, les objets sont les couples $(a,x)$, avec $a$ un objet de $\mathdeuxcat{A}$ et $x$ un objet de $F(a)$. Les \un{}cellules de $(a,x)$ vers $(a', x')$ sont les couples $(f : a \to a', r : F(f)(x) \to x')$, $f$ et $r$ étant des \un{}cellules de $\mathdeuxcat{A}$ et $F(a')$ respectivement. Les \deux{}cellules de $(f,r)$ vers $(g,s)$ sont les couples $(\gamma : f \Rightarrow g, \varphi : r (F(\gamma))_{x} \Rightarrow s)$, $\gamma$ et $\varphi$ étant des \deux{}cellules dans $\mathdeuxcat{A}$ et $F(a')$ respectivement.
\end{paragr}

\begin{paragr}\label{JaKa}
Soient $F : \mathdeuxcat{A} \to \DeuxCatDeuxCat$ un \DeuxFoncteurStrict{} et $a$ un objet de ${\mathdeuxcat{A}}$. La projection canonique
$$
\begin{aligned}
P_{F} : \DeuxInt{\mathdeuxcat{A}}F &\to \mathdeuxcat{A}
\\
(a,x) &\mapsto a
\\
(f,r) &\mapsto f
\\
(\gamma, \varphi) &\mapsto \gamma
\end{aligned}
$$ 
est un \DeuxFoncteurStrict{}. Des calculs ne présentant guère de difficulté (voir \cite[proposition 1.10.7]{TheseMoi} ou \cite[proposition 3.2]{ArticleThAMoi}) permettent de vérifier que le \DeuxFoncteurStrict{} canonique $J_{a} : \Fibre{\left(\DeuxInt{\mathdeuxcat{A}}F\right)}{P_{F}}{a} \to \TrancheLax{\left(\DeuxInt{\mathdeuxcat{A}}F\right)}{P_{F}}{a}$ est un préadjoint à droite colax. Autrement dit, $P_{F}$ est une précoopfibration. On vérifie que le \DeuxFoncteurStrict{}
$$
\begin{aligned}
K_{a} : \TrancheLax{\left(\DeuxInt{\mathdeuxcat{A}}F\right)}{P_{F}}{a} &\longrightarrow \Fibre{\left(\DeuxInt{\mathdeuxcat{A}}F\right)}{P_{F}}{a}
\\
((a',x'), p : a' \to a) &\longmapsto (a, F(p)(x'))
\\
((f : a' \to a'', r : F(f)(x') \to x''), \sigma : p \Rightarrow p'f) &\longmapsto (1_{a}, F(p')(r) (F(\sigma))_{x'})
\\
(\gamma, \varphi) &\longmapsto (1_{1_{a}}, F(p')(\varphi) \CompDeuxZero (F(\sigma))_{x'})
\end{aligned}
$$
est une rétraction de $J_{a}$. De plus, des calculs supplémentaires et tout aussi passionnants montrent qu'il s'agit d'un préadjoint à gauche lax\footnote{Comme on l'a déjà signalé, l'existence d'un morphisme lax allant dans le sens opposé à celui de $J_{a}$ résulte de propriétés générales des préadjoints. En revanche, le caractère strict de ce morphisme comme sa propriété d'être lui-même un préadjoint ne sont pas vérifiés en général.}. Des résultats analogues sont bien entendu valables pour les constructions duales introduites dans les paragraphes \ref{DefDeuxIntOp} et \ref{DefDeuxIntCo}.
\end{paragr}

%\begin{df}
%On dira qu'une \DeuxTransformationLax{} $\sigma : u \Rightarrow v$ est \emph{relative aux objets} si, pour tout objet $a$ de $\mathdeuxcat{A}$, $\sigma_{a} = 1_{u(a)} = 1_{v(a)}$. 
%\end{df}

\begin{df}\label{DefinitionNerfs}
Soient $\mathdeuxcat{A}$ une petite \deux{}catégorie et $m \geq 0$ un entier. 
\begin{itemize}
\item[$(i)$]
On note 
$$
FonLax([m], \mathdeuxcat{A})
$$
l'ensemble des \DeuxFoncteursLax{} de $[m]$ vers $\mathdeuxcat{A}$ et $\NerfLax{\mathdeuxcat{A}}$ l'ensemble simplicial
$$
\begin{aligned}
\NerfLax{\mathdeuxcat{A}} : \Delta^{op} &\to \Ens
\\
[m] &\mapsto FonLax([m], \mathdeuxcat{A})
\end{aligned}
$$
dont les faces et dégénérescences sont définies de la façon « évidente ».

Cela permet de définir un foncteur \emph{nerf lax}
$$
\begin{aligned}
\NerfLax : \DeuxCatLax &\to \EnsSimp
\\
\mathdeuxcat{A} &\mapsto \NerfLax \mathdeuxcat{A}
\\
u &\mapsto \NerfLax(u).
\end{aligned}
$$

\item[$(ii)$]
On note 
$$
FonLaxNor([m], \mathdeuxcat{A})
$$
l'ensemble des \DeuxFoncteursLax{} normalisés de $[m]$ vers $\mathdeuxcat{A}$ et $\NerfLaxNor{\mathdeuxcat{A}}$ l'ensemble simplicial 
$$
\begin{aligned}
\NerfLaxNor{\mathdeuxcat{A}} : \Delta^{op} &\to \Ens
\\
[m] &\mapsto FonLaxNor([m], \mathdeuxcat{A})
\end{aligned}
$$
dont les faces et dégénérescences sont définies de la façon « évidente ». 

Cela permet de définir un foncteur \emph{nerf lax normalisé}
$$
\begin{aligned}
\NerfLaxNor : \DeuxCat &\to \EnsSimp
\\
\mathdeuxcat{A} &\mapsto \NerfLaxNor \mathdeuxcat{A}
\\
u &\mapsto \NerfLaxNor(u).
\end{aligned}
$$

\item[$(iii)$]
On note 
$$
\underline{FonLaxNor}([m], \mathdeuxcat{A})
$$
la catégorie dont les objets sont les \DeuxFoncteursLax{} normalisés de $[m]$ vers $\mathdeuxcat{A}$ et dont les morphismes sont les \DeuxTransformationsLax{} relatives aux objets entre tels \DeuxFoncteursLax{} normalisés et $\NerfCatLaxNor{\mathdeuxcat{A}}$ l'objet simplicial de $\Cat$ 
$$
\begin{aligned}
\NerfCatLaxNor{\mathdeuxcat{A}} : \Delta^{op} &\to \Cat
\\
[m] &\mapsto \underline{FonLaxNor}([m], \mathdeuxcat{A})
\end{aligned}
$$
dont les faces et dégénérescences sont définies de la façon « évidente ». 

Cela permet de définir un foncteur \emph{nerf lax normalisé catégorique}
$$
\begin{aligned}
\NerfCatLaxNor : \DeuxCat &\to \CatHom{CAT}{\Delta^{op}}{\Cat} 
\\
\mathdeuxcat{A} &\mapsto \NerfCatLaxNor \mathdeuxcat{A}
\\
u &\mapsto \NerfCatLaxNor(u).
\end{aligned}
$$

\item[$(iv)$]
On note 
$$
\underline{Fon}([m], \mathdeuxcat{A})
$$
la catégorie dont les objets sont les \DeuxFoncteursStricts{} de $[m]$ vers $\mathdeuxcat{A}$ et dont les morphismes sont les \DeuxTransformationsLax{} relatives aux objets entre tels morphismes de $\DeuxCat$. Autrement dit, $\underline{Fon}([m], \mathdeuxcat{A})$ est la catégorie 
$$
\coprod_{\substack{(a_{0}, \dots, a_{m}) \in (\Objets{\mathdeuxcat{A}})^{m+1}}} \CatHom{\mathdeuxcat{A}}{a_{0}}{a_{1}} \times \dots \times \CatHom{\mathdeuxcat{A}}{a_{m-1}}{a_{m}}.
$$
On note 
$\NerfHom{\mathdeuxcat{A}}$ l'objet simplicial de $\Cat$ 
$$
\begin{aligned}
\NerfHom{\mathdeuxcat{A}} : \Delta^{op} &\to \Cat
\\
[m] &\mapsto \underline{Fon}([m], \mathdeuxcat{A})
\end{aligned}
$$
dont les faces et dégénérescences sont définies de la façon « évidente ». 

Cela permet de définir un foncteur \emph{nerf strict}
$$
\begin{aligned}
\NerfHom : \DeuxCat &\to \CatHom{CAT}{\Delta^{op}}{\Cat} 
\\
\mathdeuxcat{A} &\mapsto \NerfHom \mathdeuxcat{A}
\\
u &\mapsto \NerfHom(u).
\end{aligned}
$$
\end{itemize}
\end{df}

\begin{rem}
Insistons sur le fait que nous avons considéré le domaine de définition du nerf lax $\NerfLax$ et du nerf lax normalisé $\NerfLaxNor$ comme étant $\DeuxCatLax$ et $\DeuxCat$ respectivement. En particulier, bien que le nerf lax normalisé soit fonctoriel sur les \DeuxFoncteursLax{} normalisés, nous ne le considérons pas, sauf mention contraire, comme un foncteur de domaine la catégorie dont les objets sont les petites \deux{}catégories et dont les morphismes sont les \DeuxFoncteursLax{} normalisés, ce qui serait possible. Ce nerf n'est toutefois pas fonctoriel sur les \DeuxFoncteursLax{} généraux, propriété du nerf $\NerfLax$ que nous utiliserons de façon cruciale. Le nerf $\NerfLaxNor$ est cependant plus familier aux catégoriciens. On pourra par exemple se reporter à \cite[section 10]{Street}. On peut le construire de la façon suivante : pour tout entier positif $m$, il existe une \deux{}catégorie dont les objets sont les entiers de $0$ à $m$, dont les \un{}cellules de $i$ vers $j$ sont les sous-ensembles de $\{i, \dots, j\}$ contenant $i$ et $j$ et dont les \deux{}cellules entre telles \un{}cellules parallèles sont définies par l'opposée de la relation d'inclusion. Cela fournit un foncteur $\Delta \to \DeuxCat$ et donc, par un procédé classique de Kan, une adjonction entre $\EnsSimp$ et $\DeuxCat$, dont l'adjoint à droite n'est autre que le nerf $\NerfLaxNor$.
\end{rem}

\begin{paragr}
Pour tout entier $m \geq 0$, toute petite \deux{}catégorie $\mathdeuxcat{A}$ et tout \DeuxFoncteurLax{} $x : [m] \to \mathdeuxcat{A}$, on notera $x_{i}$ l'image de $0 \leq i \leq m$ par $x$, $x_{j,i}$ l'image du morphisme $i \to j$ de $[m]$ par $x$ et $x_{k,j,i}$ la \deux{}cellule structurale $x_{j \to k, i \to j} : x_{k,j} x_{j,i} \Rightarrow x_{k,i}$ pour $0 \leq i \leq j \leq k \leq m$. Afin d'éviter toute confusion avec l'objet $x_{i}$, on notera $(x)_{i}$ la \deux{}cellule structurale d'unité de $x$ associée à l'objet $i$ de $[m]$. On a donc $(x)_{i} : 1_{x_{i}} \Rightarrow x_{i,i}$. 

%Sous ces mêmes données, pour toute \DeuxTransformationLax{} ou \DeuxTransformationCoLax{} $\sigma : x \Rightarrow y$ de $x$ vers un \DeuxFoncteurLax{} $y : [m] \to \mathdeuxcat{A}$, on notera $\sigma_{j,i}$\index[not]{0sigmaji@$\sigma_{j,i}$} la composante de $\sigma$ en le morphisme $i \to j$ de $[m]$. C'est donc une \deux{}cellule de $\sigma_{j} x_{j,i}$ vers $y_{j,i} \sigma_{i}$ si $\sigma$ est une \DeuxTransformationLax{} et de $y_{j,i} \sigma_{i}$ vers $\sigma_{j} x_{j,i}$ si $\sigma$ est une \DeuxTransformationCoLax{}. On notera parfois $\sigma_{i}$ la \deux{}cellule $\sigma_{i,i-1}$.    
\end{paragr}

\begin{paragr}
Soit $\mathdeuxcat{A}$ une petite \deux{}catégorie. On définit un \DeuxFoncteurLax{}
$$
\SupLaxObjet{\mathdeuxcat{A}} : \Delta / \NerfLax{\mathdeuxcat{A}} \to \mathdeuxcat{A}
$$
comme suit. Pour tout objet $([m], x)$ de $\Delta / \NerfLax{\mathdeuxcat{A}}$, 
$$
\SupLaxObjet{\mathdeuxcat{A}}([m], x) = x_{m}.
$$

Pour tout morphisme simplicial $\varphi : [m] \to [n]$ définissant un morphisme de $([m], x)$ vers $([n], y)$ dans $\Delta /\NerfLax{\mathdeuxcat{A}}$, 
$$
\SupLaxObjet{\mathdeuxcat{A}} (\varphi : ([m], x) \to ([n], y)) = y_{n,\varphi(m)}.
$$

Pour tout objet $([m], x)$ de $\Delta / \NerfLax{\mathdeuxcat{A}}$, 
$$
\DeuxCellStructId{\SupLaxObjet{\mathdeuxcat{A}}}{([m], x)} = \DeuxCellStructId{(x)}{m} : 1_{x_{m}} \Rightarrow x(1_{m}). 
$$

Pour tout couple de morphismes composables $\varphi : ([m], x) \to ([n], y)$ et $\psi : ([n], y) \to ([p], z)$ de $\Delta / \NerfLax{\mathdeuxcat{A}}$, 
$$
\DeuxCellStructComp{\SupLaxObjet{\mathdeuxcat{A}}}{\psi}{\varphi} = z_{p, \psi(n), \psi(\varphi(m))}.
$$
\end{paragr}

\begin{paragr}
Soit $\mathdeuxcat{A}$ une petite \deux{}catégorie. On définit un \DeuxFoncteurLax{} normalisé
$$
\SupLaxNorObjet{\mathdeuxcat{A}} : \Delta / \NerfLaxNor{\mathdeuxcat{A}} \to \mathdeuxcat{A}
$$
par la condition de commutativité du diagramme
$$
\xymatrix{
\DeuxIntOp{\Delta}\NerfLax{\mathdeuxcat{A}}
\ar[rd]_{\SupLaxObjet{\mathdeuxcat{A}}}
&&
\DeuxIntOp{\Delta}\NerfLaxNor{\mathdeuxcat{A}}
\ar[ll]_{i_{l,n}^{l}\mathdeuxcat{A}}
\ar[ld]^{\SupLaxNorObjet{\mathdeuxcat{A}}}
\\
&
\mathdeuxcat{A}
&,
}
$$
dans lequel $i_{l,n}^{l}\mathdeuxcat{A}$ désigne l'inclusion canonique.
\end{paragr}

\begin{paragr}
Soit $\mathdeuxcat{A}$ une petite \deux{}catégorie. On définit un \DeuxFoncteurLax{} normalisé
$$
\SupCatLaxNorObjet{\mathdeuxcat{A}} : \DeuxIntOp{\Delta} \NerfCatLaxNor{\mathdeuxcat{A}} \to \mathdeuxcat{A}
$$
comme suit. 

Pour tout objet $([m], x)$ de $\DeuxIntOp{\Delta} \NerfCatLaxNor{\mathdeuxcat{A}}$, 
$$
\SupCatLaxNorObjet{\mathdeuxcat{A}}([m], x) = x_{m}.
$$

Pour tout morphisme $(\varphi, \alpha) : ([m], x) \to ([n], y)$ de $\DeuxIntOp{\Delta} \NerfCatLaxNor{\mathdeuxcat{A}}$, 
$$
\SupCatLaxNorObjet{\mathdeuxcat{A}} (\varphi, \alpha) = y_{n,\varphi(m)}.
$$

Pour tout objet $([m], x)$ de $\DeuxIntOp{\Delta} \NerfCatLaxNor{\mathdeuxcat{A}}$
$$
\DeuxCellStructId{\SupCatLaxNorObjet{\mathdeuxcat{A}}}{([m],x)} = \DeuxCellStructId{(x)}{m} = 1_{1_{x_{m}}}.
$$

Pour tout couple de morphismes composables $(\varphi, \alpha) : ([m], x) \to
([n], y)$ et $(\psi, \beta) : ([n], y) \to ([p], z)$ de $\DeuxIntOp{\Delta} \NerfCatLaxNor{\mathdeuxcat{A}}$, 
$$
\DeuxCellStructComp{\SupCatLaxNorObjet{\mathdeuxcat{A}}}{(\psi, \beta)}{(\varphi, \alpha)} = z_{p, \psi(n), \psi(\varphi(m))} \CompDeuxUn (z_{p, \psi(n)} \CompDeuxZero \beta_{n, \varphi(m)}).
$$
\end{paragr}

\begin{paragr}
Soit $\mathdeuxcat{A}$ une petite \deux{}catégorie. On définit un \DeuxFoncteurLax{} normalisé
$$
\SupHomObjet{\mathdeuxcat{A}} : \DeuxIntOp{\Delta} \NerfHom{\mathdeuxcat{A}} \to \mathdeuxcat{A}
$$
par la condition de commutativité du diagramme
$$
\xymatrix{
\DeuxIntOp{\Delta}\NerfCatLaxNor{\mathdeuxcat{A}}
\ar[dr]_{\SupCatLaxNorObjet{\mathdeuxcat{A}}}
&&
\DeuxIntOp{\Delta}\NerfHom{\mathdeuxcat{A}}
\ar[ll]_{i_{hom}^{\underline{l,n}}\mathdeuxcat{A}}
\ar[ld]^{\SupHomObjet{\mathdeuxcat{A}}}
\\
&
\mathdeuxcat{A}
&,
}
$$
dans lequel $i_{hom}^{\underline{l,n}}\mathdeuxcat{A}$ désigne l'inclusion canonique.
\end{paragr}

\begin{lemme}\label{DiagrammeSups}
Pour toute petite \deux{}catégorie $\mathdeuxcat{A}$, le diagramme
$$
\xymatrix{
\DeuxIntOp{\Delta}\NerfLax{\mathdeuxcat{A}}
\ar@/_1.5pc/[rrrdd]_{\SupLaxObjet{\mathdeuxcat{A}}}
&&
\DeuxIntOp{\Delta}\NerfLaxNor{\mathdeuxcat{A}}
\ar[ll]_{i_{l,n}^{l}\mathdeuxcat{A}}
\ar@/_0.5pc/[rdd]_{\SupLaxNorObjet{\mathdeuxcat{A}}}
\ar[rr]^{i_{l,n}^{\underline{l,n}}\mathdeuxcat{A}}
&&
\DeuxIntOp{\Delta}\NerfCatLaxNor{\mathdeuxcat{A}}
\ar@/^0.5pc/[ldd]^{\SupCatLaxNorObjet{\mathdeuxcat{A}}}
&&
\DeuxIntOp{\Delta}\NerfHom{\mathdeuxcat{A}}
\ar[ll]_{i_{hom}^{\underline{l,n}}\mathdeuxcat{A}}
\ar@/^1.5pc/[llldd]^{\SupHomObjet{\mathdeuxcat{A}}}
\\
\\
&&&
\mathdeuxcat{A}
}
$$
est commutatif, les flèches horizontales désignant les inclusions canoniques. 
\end{lemme}

\begin{proof}
En vertu des définitions, il suffit de vérifier l'égalité 
$
\SupLaxNorObjet{\mathdeuxcat{A}} = \SupCatLaxNorObjet{\mathdeuxcat{A}} \phantom{a} i_{l,n}^{\underline{l,n}}\mathdeuxcat{A}
$, ce qui ne pose aucune difficulté. 
\end{proof}

\section{De $\tCat$ à $\tDeuxCat$}

À toute classe $\UnLocFond{S}$ de morphismes de $\Cat$, on peut associer une classe $\NerfLaxNor^{-1} (i_{\Delta}^{-1} (\UnLocFond{S}))$ de morphismes de $\DeuxCat$. On étudie maintenant les premières propriétés des classes de \DeuxFoncteursStricts{} obtenues par un tel procédé à partir d'un \ClasseUnLocFond{}, avant d'en entreprendre une étude plus axiomatique dans la section suivante. 

\begin{prop}[Carrasco-Cegarra-Garz\'on]\label{ResultatCCG}
Pour toute petite \deux{}catégorie $\mathdeuxcat{A}$ et tout entier $m \geq 0$, les inclusions naturelles de catégories
$$
i_{l,n}^{l}\mathdeuxcat{A} : \DeuxIntOp{\Delta} \NerfLaxNor{\mathdeuxcat{A}} \hookrightarrow \DeuxIntOp{\Delta} \NerfLax{\mathdeuxcat{A}}
$$
$$
i_{l,n}^{\underline{l,n}}\mathdeuxcat{A} : \DeuxIntOp{\Delta} \NerfLaxNor{\mathdeuxcat{A}} \hookrightarrow \DeuxIntOp{\Delta}\NerfCatLaxNor{\mathdeuxcat{A}}
$$
et
$$
i_{hom}^{\underline{l,n}}\mathdeuxcat{A} : \DeuxIntOp{\Delta}\NerfHom{\mathdeuxcat{A}} \hookrightarrow \DeuxIntOp{\Delta}\NerfCatLaxNor{\mathdeuxcat{A}}
$$
sont dans $\UnLocFondMin$ (et donc des équivalences faibles pour tout \ClasseUnLocFond{}).
\end{prop}

\begin{proof}
C'est une reformulation d'une partie des résultats de \cite{CCG}.
\end{proof}

\begin{rem}\label{Nini}
En particulier, pour tout morphisme $u : \mathdeuxcat{A} \to \mathdeuxcat{B}$ de $\DeuxCat$, il existe un diagramme commutatif dans $\Cat$ 
$$
\xymatrix{
\Delta / \NerfLaxNor \mathdeuxcat{A}
\ar[rr]^{\Delta / \NerfLaxNor (u)}
\ar[d]
&& 
\Delta / \NerfLaxNor \mathdeuxcat{B}
\ar[d]
\\
\Delta / \NerfLax \mathdeuxcat{A}
\ar[rr]_{\Delta / \NerfLax (u)}
&& 
\Delta / \NerfLax \mathdeuxcat{B}
}
$$
dont les flèches verticales sont des $\UnLocFondMin$\nobreakdash-équivalences faibles, donc des équivalences faibles pour tout \ClasseUnLocFond{}. En particulier, pour tout \ClasseUnLocFond{}, pour tout morphisme $u$ de $\DeuxCat$, le foncteur $\Delta / \NerfLax{(u)}$ est une équivalence faible si et seulement si le foncteur $\Delta / \NerfLaxNor{(u)}$ en est une. 
\end{rem}

\emph{Soit $\UnLocFond{W}$ un \ClasseUnLocFond{} fixé.}  

\begin{df}\label{DefDeuxW}
On notera $\DeuxLocFond{W}$ la classe des morphismes de $\DeuxCat$ dont l'image par le foncteur $\Delta / \NerfLax(\bullet)$ (ou, de façon équivalente en vertu de la remarque \ref{Nini}, l'image par le foncteur $\Delta / \NerfLaxNor(\bullet)$) est une équivalence faible.

On appellera les éléments de $\DeuxLocFond{W}$ des $\DeuxLocFond{W}$\nobreakdash-\emph{équivalences faibles}, ou plus simplement des \emph{équivalences faibles}. 
\end{df}

\begin{rem}\label{DeuxLocFondSat}
Par fonctorialité, la classe $\DeuxLocFond{W}$ est faiblement saturée.
\end{rem}

\begin{df}
On dira qu'une petite \deux{}catégorie $\mathdeuxcat{A}$ est $\DeuxLocFond{W}$\nobreakdash-\emph{asphérique}, ou plus simplement \emph{asphérique}, si le morphisme canonique $\mathdeuxcat{A} \to \DeuxCatPonct$ est une équivalence faible.
\end{df}

%\begin{exemple}[Bourbaki]\label{PointAsph}
%La \deux{}catégorie ponctuelle $\DeuxCatPonct$ est asphérique.
%\end{exemple}

\begin{lemme}\label{Rutebeuf}
Une petite \deux{}catégorie $\mathdeuxcat{A}$ est $\DeuxLocFond{W}$\nobreakdash-asphérique si et seulement si la catégorie $\Delta/\NerfLaxNor\mathdeuxcat{A}$ est $\UnLocFond{W}$\nobreakdash-asphérique.  
\end{lemme}

\begin{proof}
C'est immédiat, la catégorie $\Delta$ étant $\UnLocFond{W}$\nobreakdash-asphérique (elle admet un objet final) et la classe $\UnLocFond{W}$ vérifiant la propriété de $2$ sur $3$. 
\end{proof}

%\begin{lemme}\label{EquivalenceViaPoint}
%Soit $A$ une catégorie admettant un objet final $*$. Pour tout objet $a$ de $A$, on notera $p_{a}$ l'unique morphisme de $a$ vers $*$ dans $A$. Soit $S \subset \UnCell{A}$ vérifiant les propriétés suivantes.
%\begin{itemize}
%\item[(i)] Si $u \in S$ et $vu \in S$, alors $v \in S$.
%\item[(ii)] Si $vu = 1_{*}$ et $uv \in S$, alors $v \in S$ (et donc $u \in S$ en vertu de la propriété précédente).
%\end{itemize}
%Alors, pour tout diagramme commutatif 
%$$
%\xymatrix{
%a
%\ar[rr]^{p_{a}}
%\ar[dr]_{s}
%&&{*}
%\ar[dl]^{q}
%\\
%&a'
%}
%$$
%dans $A$ tel que $s \in S$, $q$ et $p_{a'}$ sont aussi dans $S$. Si, de plus, $S$ est stable par composition, alors $p_{a}$ est aussi dans $S$. 
%\end{lemme}
%
%\begin{proof}
%De $s \in S$ et $qp_{a'}s = qp_{a} = s \in S$, l'on déduit $qp_{a'} \in S$. Comme de plus $p_{a'}q = 1_{*}$, on a bien le résultat.
%\end{proof}

\begin{rem}\label{EndoConstantW}
Rappelons qu'un endofoncteur $u : A \to A$ est dit \emph{constant} s'il existe un foncteur $a : e \to A$ tel que $u = ap_{A}$. Autrement dit, $u$ se factorise par la catégorie ponctuelle. En vertu de la saturation faible des localisateurs fondamentaux, une petite catégorie admettant un endofoncteur constant qui est une équivalence faible est asphérique.
\end{rem}

\begin{lemme}\label{OFAspherique1}
Une petite \deux{}catégorie admettant un objet admettant un objet final est asphérique.
\end{lemme}

\begin{proof}
Soit $\mathdeuxcat{A}$ une petite \deux{}catégorie admettant un objet $z$ tel que, pour tout objet $a$ de $\mathdeuxcat{A}$, la catégorie $\CatHom{\mathdeuxcat{A}}{a}{z}$ admette un objet final. Pour montrer le résultat désiré, il suffit, en vertu du lemme \ref{Rutebeuf}, de vérifier que la catégorie $\Delta/\NerfLaxNor \mathdeuxcat{A}$ est asphérique. 

Pour tout objet $a$ de $\mathdeuxcat{A}$, on notera $p_{a} : a \to z$ l'objet final de $\CatHom{\mathdeuxcat{A}}{a}{z}$ et, pour toute \un{}cellule $f : a \to z$, on notera $\varphi_{f} : f \Rightarrow p_{a}$ l'unique \deux{}cellule de $f$ vers $p_{a}$ dans $\mathdeuxcat{A}$. Pour tout \DeuxFoncteurLax{} normalisé $x : [m] \to \mathdeuxcat{A}$, on définit un \DeuxFoncteurLax{} normalisé $D(x) : [m+1] \to \mathdeuxcat{A}$ comme suit. Pour tout objet $i$ de $[m]$, $D(x)_{i} = x_{i}$ ; de plus, $D(x)_{m+1} = z$. Pour tout couple d'entiers $0 \leq i \leq j \leq m$, $D(x)_{j,i} = x_{j,i}$ ; de plus, $D(x)_{m+1, i} = p_{x_{i}}$ pour tout objet $i$ de $[m]$. Pour tout triplet d'entiers $0 \leq i \leq j \leq k \leq m$, $D(x)_{k,j,i} = x_{k,j,i}$ ; de plus, $D(x)_{m+1, j, i} = \varphi_{p_{x_{j}} \CompDeuxZero D(x)_{j,i}}$ pour tout couple d'entiers $0 \leq i \leq j \leq m$. Pour tout morphisme simplicial $\varphi : [m] \to [n]$, on définit un morphisme simplicial $D(\varphi) : [m+1] \to [n+1]$ par $D(\varphi)(i) = \varphi(i)$ si $i \leq m$ et $D(\varphi)(m+1) = n+1$. Cela permet de définir un endofoncteur
$$
\begin{aligned}
D : \Delta/\NerfLaxNor{\mathdeuxcat{A}} &\to \Delta/\NerfLaxNor{\mathdeuxcat{A}} 
\\
([m], x : [m] \to \mathdeuxcat{A}) &\mapsto ([m+1], D(x) : [m+1] \to \mathdeuxcat{A})
\\
\varphi &\mapsto D(\varphi).
\end{aligned}
$$
Considérons en outre l'endofoncteur constant
$$
\begin{aligned}
Z : \Delta/\NerfLaxNor{\mathdeuxcat{A}} &\to \Delta/\NerfLaxNor{\mathdeuxcat{A}} 
\\
([m], x : [m] \to \mathdeuxcat{A}) &\mapsto ([0], z)
\\
\varphi &\mapsto 1_{[0]}.
\end{aligned}
$$ 
Pour tout objet $([m], x)$ de $\Delta/\NerfLaxNor{\mathdeuxcat{A}}$, posons 
$$
\begin{aligned}
\iota_{([m], x)} : [m] &\to [m+1]
\\
i &\mapsto i
\end{aligned}
$$
et 
$$
\begin{aligned}
\omega_{([m], x)} : [0] &\to [m+1]
\\
0 &\mapsto m+1.
\end{aligned}
$$
Cela définit des transformations naturelles $\iota : 1_{\Delta/\NerfLaxNor{\mathdeuxcat{A}}} \Rightarrow D$ et $\omega : Z \Rightarrow D$. Il en résulte que $Z$ est une équivalence faible. Comme c'est un endofoncteur constant de $\Delta/\NerfLaxNor{\mathdeuxcat{A}}$, cette catégorie est asphérique (voir la remarque \ref{EndoConstantW}). 
\end{proof}

%\begin{rem}\label{OFContractile}
%Voici une autre démonstration du lemme \ref{OFAspherique1}. Soit $z$ un objet d'une \deux{}catégorie $\mathdeuxcat{A}$ tel que, pour tout objet $a$ de $\mathdeuxcat{A}$, la catégorie $\CatHom{\mathdeuxcat{A}}{a}{z}$ admette un objet final, que l'on notera $p_{a}$. Pour toute \un{}cellule $f : a \to z$ de $\mathdeuxcat{A}$, l'on notera $\varphi_{f}$ l'unique \deux{}cellule de $f$ vers $p_{a}$. On définit un \deux{}endofoncteur constant (c'est-à-dire, rappelons-le, se factorisant par la \deux{}catégorie ponctuelle $\DeuxCatPonct$) $Z$ de $\mathdeuxcat{A}$ par 
%$$
%\begin{aligned}
%Z : \mathdeuxcat{A} &\to \mathdeuxcat{A}
%\\
%a &\mapsto z
%\\
%f &\mapsto 1_{z}
%\\
%\alpha &\mapsto 1_{1_{z}}
%\end{aligned}
%$$
%On pose alors $\sigma_{a} = p_{a}$ pour tout objet $a$ et $\sigma_{f} = \varphi_{p_{a'} f}$ pour toute \un{}cellule $f : a \to a'$ de $\mathdeuxcat{A}$. Cela définit une \DeuxTransformationLax{} de $1_{\mathdeuxcat{A}}$ vers $Z$. En vertu du lemme \ref{DeuxTransFoncLax}, il existe une homotopie élémentaire de $1_{\mathdeuxcat{A}}$ vers $Z$. Le lemme \ref{HomotopieW} permet donc d'affirmer que $Z$ est une équivalence faible. On conclut grâce au lemme \ref{EquivalenceViaPoint}. 
%\end{rem}

\begin{theo}\label{ThABC}
Soit
$$
\xymatrix{
\mathdeuxcat{A} 
\ar[rr]^{u}
\ar[dr]_{w}
&&\mathdeuxcat{B}
\ar[dl]^{v}
\\
&\mathdeuxcat{C}
}
$$
un triangle commutatif dans $\DeuxCat$. Supposons que, pour tout objet $c$ de $\mathdeuxcat{C}$, le \DeuxFoncteurStrict{} induit par ces données
$$
\DeuxFoncTrancheCoLax{u}{c} : \TrancheCoLax{\mathdeuxcat{A}}{w}{c} \to \TrancheCoLax{\mathdeuxcat{B}}{v}{c}
$$
soit une équivalence faible. Alors, $u$ est une équivalence faible.
\end{theo}

\begin{proof}
En utilisant les propositions \ref{LemmeBisimplicial} et \ref{DeltaTotalementAspherique}, le résultat se démontre de façon tout à fait analogue à \cite[théorème 2.32]{ArticleThAMoi}.
\end{proof}

\begin{rem}
Il est bien entendu possible de remplacer le \DeuxFoncteurStrict{} $\DeuxFoncTrancheCoLax{u}{c}$ par $\DeuxFoncTrancheLax{u}{c}$, $\DeuxFoncOpTrancheCoLax{u}{c}$ ou $\DeuxFoncOpTrancheLax{u}{c}$ dans l'énoncé du théorème \ref{ThABC}. On démontrera plus loin ces variantes, ainsi que des généralisations d'icelles, dans un cadre plus conceptuel.
\end{rem}

\section{Localisateurs fondamentaux de $\tDeuxCat$}

\begin{df}\label{def:loc_fond_dCat}
Un \emph{\ClasseDeuxLocFond{}} est une classe $\DeuxLocFond{W}$ de morphismes de $\DeuxCat$ vérifiant les propriétés suivantes. 
\begin{itemize}
\item[LF1] La classe $\DeuxLocFond{W}$ est faiblement saturée.

\item[LF2] Si une petite \deux{}catégorie $\mathdeuxcat{A}$ admet un objet admettant un objet final, alors le morphisme canonique $\mathdeuxcat{A} \to \DeuxCatPonct{}$ est dans $\DeuxLocFond{W}$.

\item[LF3] Si 
$$
\xymatrix{
\mathdeuxcat{A} 
\ar[rr]^{u}
\ar[dr]_{w}
&&\mathdeuxcat{B}
\ar[dl]^{v}
\\
&\mathdeuxcat{C}
}
$$
désigne un triangle commutatif dans $\DeuxCat$ et si, pour tout objet $c$ de $\mathcal{C}$, le \DeuxFoncteurStrict{} induit par ces données
$$
\DeuxFoncTrancheCoLax{u}{c} : \TrancheCoLax{\mathdeuxcat{A}}{w}{c} \to \TrancheCoLax{\mathdeuxcat{B}}{v}{c} 
$$
est dans $\DeuxLocFond{W}$, alors $u$ est dans $\DeuxLocFond{W}$.
\end{itemize}
\end{df}

\begin{rem}\label{Levet}
Si $\UnLocFond{W}$ est un \ClasseUnLocFond{}, alors la classe $\DeuxLocFond{W} = \NerfLaxNor^{-1} (i_{\Delta}^{-1} (W))$ est un \ClasseDeuxLocFond{} en vertu de la remarque \ref{DeuxLocFondSat}, du lemme \ref{OFAspherique1} et du théorème \ref{ThABC}. En particulier, $\DeuxLocFondMin = \NerfLaxNor^{-1} (i_{\Delta}^{-1} (\UnLocFondMin)) = \NerfLaxNor^{-1} (\EquiQuillen)$ est un \ClasseDeuxLocFond{}.
\end{rem}

Dans la suite, on commettra souvent l'abus de considérer $\Cat$ comme une sous-catégorie (pleine) de $\DeuxCat$.

\begin{rem}\label{Proust}
Si $\DeuxLocFond{W}$ est un \ClasseDeuxLocFond{}, alors $\DeuxLocFond{W} \cap \UnCell{\Cat}$ est un \ClasseUnLocFond{}. 
\end{rem}

L'objectif principal de cet article consiste à montrer que \emph{tous} les \ClassesUnLocFond{} et \emph{tous} les \ClasseDeuxLocFondS{} s'obtiennent les premiers à partir des seconds et les seconds à partir des premiers par les opérations figurant dans les énoncés des remarques \ref{Levet} et \ref{Proust} respectivement, que ces deux opérations sont inverses l'une de l'autre et qu'elles induisent des équivalences de catégories entre les catégories homotopiques associées de $\Cat$ et de $\DeuxCat$.

\medbreak

\emph{On suppose fixé un localisateur fondamental $\DeuxLocFond{W}$ de $\DeuxCat$.}

\medskip

On appellera les éléments de $\DeuxLocFond{W}$ des $\DeuxLocFond{W}\emph{-équivalences}$ ou, plus simplement, si cela n'introduit aucune ambiguïté, des \emph{équivalences faibles}. 

\begin{df}
On dira qu'une petite \deux{}catégorie $\mathdeuxcat{A}$ est \emph{$\DeuxLocFond{W}$-asphérique}, ou plus simplement \emph{asphérique}, si le morphisme canonique $\mathdeuxcat{A} \to e$ est dans $\DeuxLocFond{W}$. On dira qu'un \DeuxFoncteurStrict{} $u : \mathdeuxcat{A} \to \mathdeuxcat{B}$ est \emph{$\DeuxLocFond{W}$-colax-asphérique}, ou plus simplement \emph{colax-asphérique} si, pour tout objet $b$ de $\mathdeuxcat{B}$, la \deux{}catégorie $\TrancheCoLax{\mathdeuxcat{A}}{u}{b}$ est asphérique. 
\end{df}

\begin{exemple}\label{CommaCoLaxAspherique}
Pour toute petite \deux{}catégorie $\mathdeuxcat{A}$ et tout objet $a$ de $\mathdeuxcat{A}$, la \deux{}catégorie $\TrancheCoLax{\mathdeuxcat{A}}{}{a}$ est asphérique, en vertu de l'exemple \ref{ExemplesOF} et de l'axiome LF2. La condition LF3 et un argument de « 2 sur 3 » permettent donc d'affirmer qu'un \DeuxFoncteurStrict{} colax-asphérique est une équivalence faible. 
\end{exemple}

\begin{rem}
La condition LF2 de la définition \ref{def:loc_fond_dCat} est équivalente, modulo LF1 et LF3, à la suivante : pour toute petite \deux{}catégorie $\mathdeuxcat{A}$ et tout objet $a$ de $\mathdeuxcat{A}$, le morphisme canonique $\TrancheCoLax{\mathdeuxcat{A}}{}{a} \to \DeuxCatPonct$ est dans $\DeuxLocFond{W}$. L'énoncé analogue pour les localisateurs fondamentaux de $\Cat$ est évident puisque, si une petite catégorie $A$ admet un objet final $z$, les catégories $A$ et $A/z$ sont canoniquement isomorphes. Le cas de $\DeuxCat$ est un peu plus subtil et nous le laissons en exercice (instructif, même si nous ne l'utiliserons pas).
\end{rem}

\begin{lemme}\label{Reynaldo}
Une $\UnLocFondMin$\nobreakdash-équivalence faible est une $\DeuxLocFond{W}$\nobreakdash-équivalence faible. Une petite catégorie  $\UnLocFondMin$\nobreakdash-asphérique est $\DeuxLocFond{W}$\nobreakdash-asphérique.
\end{lemme}

\begin{proof}
La seconde assertion est conséquence de la première, qui résulte du théorème \ref{CisinskiGrothendieck} et de la remarque \ref{Proust}.
\end{proof}

Le lemme \ref{LemmeDeGeorges} est immédiat. 

\begin{lemme}\label{LemmeDeGeorges}
Soient $\mathdeuxcat{A}$ une petite \deux{}catégorie, $z$ un objet de $\mathdeuxcat{A}$ et $q : \DeuxCatPonct \to \mathdeuxcat{A}$ le morphisme de $\DeuxCat$ défini par l'objet $z$. Alors, pour tout objet $x$ de $\mathdeuxcat{A}$, $\TrancheCoLax{\DeuxCatPonct}{q}{x}$ est la \deux{}catégorie associée à la catégorie $\DeuxCatUnOp{(\CatHom{\mathdeuxcat{A}}{z}{x})}$.
\end{lemme}

\begin{lemme}\label{OFAspherique2}
Une petite \deux{}catégorie op-admettant un objet admettant un objet initial est asphérique.  
\end{lemme}

\begin{proof}
Soient $\mathdeuxcat{A}$ une petite \deux{}catégorie op-admettant un objet admettant un objet initial et soit $z$ un objet de $\mathdeuxcat{A}$ tel que, pour tout objet $x$ de $\mathdeuxcat{A}$, la catégorie $\CatHom{\mathdeuxcat{A}}{z}{x}$ admette un objet initial. Considérons les \deux{}foncteurs $p_{\mathdeuxcat{A}} : \mathdeuxcat{A} \to \DeuxCatPonct$ et $q_{\mathdeuxcat{A}} : \DeuxCatPonct \to \mathdeuxcat{A}$, le second étant défini par $q_\mathdeuxcat{A}(*)= z$. En vertu du lemme \ref{LemmeDeGeorges}, pour tout objet $x$ de $\mathdeuxcat{A}$, la \deux{}catégorie $\TrancheCoLax{\DeuxCatPonct}{q_\mathdeuxcat{A}}{x}$ admet un objet admettant un objet final (elle admet également un objet admettant un objet initial). On en déduit que $\TrancheCoLax{\DeuxCatPonct}{q_\mathdeuxcat{A}}{x}$ est asphérique pour tout objet $x$ de $\mathdeuxcat{A}$. Par conséquent, $q_{\mathdeuxcat{A}}$ une équivalence faible. Comme $p_{\mathdeuxcat{A}} q_{\mathdeuxcat{A}} = 1_{\DeuxCatPonct}$, il résulte de la saturation faible de $\DeuxLocFond{W}$ que $p_{\mathdeuxcat{A}}$ est une équivalence faible. Par définition, $\mathdeuxcat{A}$ est donc asphérique.
\end{proof}

\begin{lemme}\label{OFAspherique2Bis}
Une petite \deux{}catégorie op-admettant un objet admettant un objet final est asphérique.  
\end{lemme}

\begin{proof}
Soit $\mathdeuxcat{A}$ une petite \deux{}catégorie op-admettant un objet admettant un objet final et soit $z$ un objet de $\mathdeuxcat{A}$ tel que, pour tout objet $x$ de $\mathdeuxcat{A}$, la catégorie $\CatHom{\mathdeuxcat{A}}{z}{x}$ admette un objet final. Considérons les \deux{}foncteurs $p_{\mathdeuxcat{A}} : \mathdeuxcat{A} \to \DeuxCatPonct$ et $q_{\mathdeuxcat{A}} : \DeuxCatPonct \to \mathdeuxcat{A}$, le second étant défini par $q_\mathdeuxcat{A}(*)= z$. En vertu du lemme \ref{LemmeDeGeorges}, pour tout objet $x$ de $\mathdeuxcat{A}$, la \deux{}catégorie $\TrancheCoLax{\DeuxCatPonct}{q_\mathdeuxcat{A}}{x}$ op-admet un objet admettant un objet initial (elle op-admet également un objet admettant un objet final).  On en déduit que $\TrancheCoLax{\DeuxCatPonct}{q_\mathdeuxcat{A}}{x}$ est asphérique pour tout objet $x$ de $\mathdeuxcat{A}$, en vertu du lemme \ref{OFAspherique2}. Par conséquent, $q_{\mathdeuxcat{A}}$ est une équivalence faible. La saturation faible de $\DeuxLocFond{W}$ permet d'en conclure que $\mathdeuxcat{A}$ est asphérique. 
\end{proof}

\begin{corollaire}\label{OpTrancheLaxOpTrancheCoLaxAspheriques}
Pour toute petite \deux{}catégorie $\mathdeuxcat{A}$ et tout objet $a$ de $\mathdeuxcat{A}$, les \deux{}catégories $\OpTrancheLax{\mathdeuxcat{A}}{}{a}$ et $\OpTrancheCoLax{\mathdeuxcat{A}}{}{a}$ sont asphériques. 
\end{corollaire} 

\begin{proof}
En vertu de l'exemple \ref{ExemplesOF}, cela résulte des lemmes \ref{OFAspherique2} et \ref{OFAspherique2Bis}. 
\end{proof}

\begin{lemme}\label{PreadjointW}
Un préadjoint à gauche colax est une équivalence faible.
\end{lemme}

\begin{proof}
En vertu des définitions, cela résulte des conditions LF2 et LF3.
\end{proof}

%\begin{rem}
%L'énoncé du lemme \ref{PreadjointW} constitue une généralisation, à tout localisateur fondamental de $\DeuxCat$, d'un quart de l'énoncé de \cite[corollaire 2.35]{ArticleThAMoi}. On retrouvera les trois autres quarts plus loin.
%\end{rem}

\begin{prop}\label{PrefibrationIntegraleW}
Soient $\mathdeuxcat{A}$ une petite \deux{}catégorie et $F : \DeuxCatDeuxOp{\mathdeuxcat{A}} \to \DeuxCatDeuxCat$ un \DeuxFoncteurStrict{} tel que, pour tout objet $a$ de $\mathdeuxcat{A}$, la \deux{}catégorie $F(a)$ soit asphérique. Alors, la projection canonique
$$
P_{F} : \DeuxIntCo{\mathdeuxcat{A}} F \to \mathdeuxcat{A}
$$
est colax-asphérique (donc en particulier une équivalence faible). 
\end{prop}

\begin{proof}
La fibre $\Fibre{(\DeuxIntCo{\mathdeuxcat{A}} F)}{P_{F}}{a}$ de $P_{F}$ au-dessus de $a$ s'identifiant à $F(a)$, elle est asphérique en vertu des hypothèses. On sait de plus (voir le paragraphe \ref{JaKa}) qu'il existe un \DeuxFoncteurStrict{} 
$$
K_{a} : \TrancheCoLax{\left(\DeuxIntCo{\mathdeuxcat{A}} F\right)}{P_{F}}{a} \to \Fibre{\left(\DeuxIntCo{\mathdeuxcat{A}} F\right)}{P_{F}}{a}
$$  
qui est un préadjoint à gauche colax, donc une équivalence faible en vertu du lemme \ref{PreadjointW}. Par conséquent, la \deux{}catégorie $\TrancheCoLax{(\DeuxIntCo{\mathdeuxcat{A}} F)}{P_{F}}{a}$ est asphérique, ce qui permet de conclure (voir l'exemple \ref{CommaCoLaxAspherique}). 
\end{proof}

\begin{paragr}
Soit $\mathdeuxcat{A}$ une petite \deux{}catégorie. On définit une \deux{}catégorie $S_{1}(\mathdeuxcat{A})$ comme suit :
$$
\Objets{S_{1}(\mathdeuxcat{A})} = \UnCell{\mathdeuxcat{A}}.
$$
Si $k : b \to a$ et $k' : b' \to a'$ sont deux \un{}cellules de $\mathdeuxcat{A}$, une \un{}cellule de $k$ vers $k'$ dans $S_{1}(\mathdeuxcat{A})$ est un triplet 
$$
(f : b' \to b, g : a \to a', \alpha : k' \Rightarrow gkf),
$$
ce que représente le diagramme
$$
\xymatrix{
b 
\ar[r]^{k}
&a
\ar[d]^{g}
\\
b'
\ar[u]^{f}
\ar[r]_{k'}
\ar@{}[ur] | (.54) {\big\Uparrow \alpha}
&a'
&. 
}
$$
Si $(f,g,\alpha)$ et $(f',g',\alpha')$ sont deux \un{}cellules parallèles de $k$ vers $k'$ dans $S_{1}(\mathdeuxcat{A})$, les \deux{}cellules de $(f,g,\alpha)$ vers $(f',g',\alpha')$ dans $S_{1}(\mathdeuxcat{A})$ sont les couples 
$$
(\varphi : f \Rightarrow f', \gamma : g \Rightarrow g')
$$ 
tels que 
$$
(\gamma \CompDeuxZero k \CompDeuxZero \varphi) \alpha = \alpha'.
$$

Les diverses compositions et identités de $S_{1}(\mathdeuxcat{A})$ se définissent de façon « évidente » à partir de celles de $\mathdeuxcat{A}$.

On vérifie l'existence d'isomorphismes canoniques 
$$
S_{1}(\mathdeuxcat{A}) \simeq  {\DeuxIntCo{\DeuxCatUnOp{\mathdeuxcat{A}}}} (\OpTrancheLax{\mathdeuxcat{A}}{}{a})  \simeq 
{\DeuxIntCo{\mathdeuxcat{A}}} \DeuxCatUnOp{(\TrancheLax{\mathdeuxcat{A}}{}{a})}.
$$
Il existe un diagramme de projections canoniques
$$
\xymatrix{
\DeuxCatUnOp{\mathdeuxcat{A}}
&S_{1}(\mathdeuxcat{A})
\ar[l]_{s_{1}^{\mathdeuxcat{A}}}
\ar[r]^{t_{1}^{\mathdeuxcat{A}}}
&\mathdeuxcat{A}.
}
$$
La fibre de la flèche de gauche au-dessus de $a \in \Objets{\mathdeuxcat{A}}$ s'identifie à $\OpTrancheLax{\mathdeuxcat{A}}{}{a}$, donc est asphérique en vertu du corollaire \ref{OpTrancheLaxOpTrancheCoLaxAspheriques} ; la fibre de la flèche de droite au-dessus de $a \in \Objets{\mathdeuxcat{A}}$ s'identifie à $\DeuxCatUnOp{(\TrancheLax{\mathdeuxcat{A}}{}{a})}$, donc est asphérique en vertu du lemme \ref{OFAspherique2}, puisqu'elle op-admet un objet admettant un objet initial. En particulier, en vertu de la proposition \ref{PrefibrationIntegraleW}, ces deux flèches sont des équivalences faibles. 
\end{paragr}

\begin{paragr}
Pour tout morphisme $u : \mathdeuxcat{A} \to \mathdeuxcat{B}{}$ de $\DeuxCat$, on en définit un autre par 
$$
\begin{aligned}
S_{1}(u) : S_{1}(\mathdeuxcat{A}) &\to S_{1}(\mathdeuxcat{B}{})
\\
k &\mapsto u(k)
\\
(f,g,\alpha) &\mapsto (u(f), u(g), u(\alpha))
\\
(\varphi, \gamma) &\mapsto (u(\varphi), u(\gamma)).
\end{aligned}
$$

Cette définition rend le diagramme suivant commutatif :
$$
\xymatrix{
\DeuxCatUnOp{\mathdeuxcat{A}}
\ar[d]_{\DeuxFoncUnOp{u}}
&S_{1}(\mathdeuxcat{A})
\ar[l]_{s_{1}^{\mathdeuxcat{A}}}
\ar[d]^{S_{1}(u)}
\ar[r]^{t_{1}^{\mathdeuxcat{A}}}
&\mathdeuxcat{A}
\ar[d]^{u}
\\
\DeuxCatUnOp{\mathdeuxcat{B}}
&S_{1}(\mathdeuxcat{B}{})
\ar[l]^{s_{1}^{\mathdeuxcat{B}{}}}
\ar[r]_{t_{1}^{\mathdeuxcat{B}{}}}
&\mathdeuxcat{B}{}
&.
}
$$

Comme les flèches horizontales de ce diagramme sont toutes des équivalences faibles, en vertu de ce qui précède, on déduit la proposition \ref{DeuxFoncUnOpW} de deux applications consécutives d'un argument de « $2$ sur $3$ ». 
\end{paragr}

\begin{prop}\label{DeuxFoncUnOpW}
Un morphisme $u$ de $\DeuxCat$ est une équivalence faible si et seulement si $\DeuxFoncUnOp{u}$ est une équivalence faible.
\end{prop}

\begin{corollaire}\label{DeuxCatUnOpAsph}
Une petite \deux{}catégorie $\mathdeuxcat{A}$ est asphérique si et seulement si $\DeuxCatUnOp{\mathdeuxcat{A}}$ l'est. 
\end{corollaire}

\begin{corollaire}\label{OFAspherique3}
Une petite \deux{}catégorie admettant un objet admettant un objet initial est asphérique.
\end{corollaire}

\begin{prop}\label{TheoremeAOpTrancheCoLax}
Soit 
$$
\xymatrix{
\mathdeuxcat{A} 
\ar[rr]^{u}
\ar[dr]_{w}
&&\mathdeuxcat{B}
\ar[dl]^{v}
\\
&\mathdeuxcat{C}
}
$$
un diagramme commutatif dans $\DeuxCat$. Supposons que, pour tout objet $c$ de $\mathdeuxcat{C}$, le morphisme de $\DeuxCat$ induit
$$
\DeuxFoncOpTrancheCoLax{u}{c} : \OpTrancheCoLax{\mathdeuxcat{A}}{w}{c} \to \OpTrancheCoLax{\mathdeuxcat{B}}{v}{c} 
$$
soit une équivalence faible. Alors $u$ est une équivalence faible.
\end{prop}

\begin{proof}
En vertu des hypothèses et par définition, le \DeuxFoncteurStrict{} $\DeuxFoncUnOp{(\DeuxFoncTrancheCoLax{\DeuxFoncUnOp{u}}{c})}$ est une équivalence faible pour tout objet $c$ de $\mathdeuxcat{C}$. Par conséquent, en vertu de la proposition \ref{DeuxFoncUnOpW}, le \DeuxFoncteurStrict{} $\DeuxFoncTrancheCoLax{\DeuxFoncUnOp{u}}{c}$ est une équivalence faible pour tout objet $c$ de $\mathdeuxcat{C}$. En vertu de la condition LF3, le \DeuxFoncteurStrict{} $\DeuxFoncUnOp{u}$ est donc une équivalence faible. Une nouvelle invocation de la proposition \ref{DeuxFoncUnOpW} permet de conclure.
\end{proof} 

\begin{corollaire}
Soit $u : \mathdeuxcat{A} \to \mathdeuxcat{B}$ un morphisme de $\DeuxCat$. Supposons que, pour tout objet $b$ de $\mathdeuxcat{B}$, la \deux{}catégorie $\OpTrancheCoLax{\mathdeuxcat{A}}{u}{b}$ soit asphérique. Alors $u$ est une équivalence faible.
\end{corollaire}

\begin{paragr}
Soit $\mathdeuxcat{A}$ une petite \deux{}catégorie. On définit une \deux{}catégorie $S_{2}(\mathdeuxcat{A})$ comme suit :
$$
\Objets{S_{2}(\mathdeuxcat{A})} = \UnCell{\mathdeuxcat{A}}.
$$
Si $k : b \to a$ et $k' : b' \to a'$ sont deux \un{}cellules de $\mathdeuxcat{A}$, une \un{}cellule de $k$ vers $k'$ dans $S_{2}(\mathdeuxcat{A})$ est un triplet 
$$
(f : b \to b', g : a \to a', \alpha : k' f \Rightarrow gk),
$$
ce que représente le diagramme
$$
\xymatrix{
b 
\ar[r]^{k}
\ar[d]_{f}
&a
\ar[d]^{g}
\\
b'
\ar[r]_{k'}
\ar@{}[ur] | (.54) {\big\Uparrow \alpha}
&a'
&.
}
$$
Si $(f,g,\alpha)$ et $(f',g',\alpha')$ sont deux \un{}cellules parallèles de $k$ vers $k'$ dans $S_{2}(\mathdeuxcat{A})$, les \deux{}cellules de $(f,g,\alpha)$ vers $(f',g',\alpha')$ dans $S_{2}(\mathdeuxcat{A})$ sont les couples 
$$
(\varphi : f' \Rightarrow f, \gamma : g \Rightarrow g')
$$ 
tels que 
$$
(\gamma \CompDeuxZero k) \CompDeuxUn \alpha \CompDeuxUn (k' \CompDeuxZero \varphi) = \alpha'.
$$

Les diverses compositions et identités de $S_{2}(\mathdeuxcat{A})$ se définissent de façon « évidente » à partir de celles de $\mathdeuxcat{A}$.

On vérifie l'existence d'isomorphismes canoniques 
$$
S_{2}(\mathdeuxcat{A}) \simeq \DeuxCatUnOp{\left({\DeuxIntCo{\DeuxCatToutOp{\mathdeuxcat{A}}}} \DeuxCatUnOp{(\OpTrancheLax{\mathdeuxcat{A}}{}{a})}\right)} \simeq
{\DeuxIntCo{\mathdeuxcat{A}}} \DeuxCatDeuxOp{(\TrancheCoLax{\mathdeuxcat{A}}{}{a})}.
$$
Il existe un diagramme de projections canoniques
$$
\xymatrix{
\DeuxCatDeuxOp{\mathdeuxcat{A}}
&S_{2}(\mathdeuxcat{A})
\ar[l]_{s_{2}^{\mathdeuxcat{A}}}
\ar[r]^{t_{2}^{\mathdeuxcat{A}}}
&\mathdeuxcat{A}.
}
$$
La fibre de la flèche de droite au-dessus de $a \in \Objets{\mathdeuxcat{A}}$ s'identifie à $\DeuxCatDeuxOp{(\TrancheCoLax{\mathdeuxcat{A}}{}{a})}$, donc est asphérique (elle admet un objet admettant un objet initial). En vertu de la proposition \ref{PrefibrationIntegraleW}, $t_{2}^{\mathdeuxcat{A}}$ est donc une équivalence faible. La fibre de la flèche de gauche au-dessus de $a \in \Objets{\mathdeuxcat{A}}$ s'identifie à $\OpTrancheLax{\mathdeuxcat{A}}{}{a}$, donc est asphérique. En vertu de la proposition \ref{PrefibrationIntegraleW}, la projection canonique
$$
\DeuxIntCo{\DeuxCatToutOp{\mathdeuxcat{A}}} 
\DeuxCatUnOp{(\OpTrancheLax{\mathdeuxcat{A}}{}{a})} \to \DeuxCatToutOp{\mathdeuxcat{A}}
$$ 
est une équivalence faible. Comme elle s'identifie à $\DeuxFoncUnOp{\left(s_{2}^{\mathdeuxcat{A}}\right)}$, $s_{2}^{\mathdeuxcat{A}}$ est une équivalence faible en vertu du corollaire \ref{DeuxFoncUnOpW}.
\end{paragr}

\begin{paragr}
Pour tout morphisme $u : \mathdeuxcat{A} \to \mathdeuxcat{B}{}$ de $\DeuxCat$, on en définit un autre par
$$
\begin{aligned}
S_{2}(u) : S_{2}(\mathdeuxcat{A}) &\to S_{2}(\mathdeuxcat{B}{})
\\
k &\mapsto u(k)
\\
(f,g,\alpha) &\mapsto (u(f), u(g), u(\alpha))
\\
(\varphi, \gamma) &\mapsto (u(\varphi), u(\gamma)).
\end{aligned}
$$

Cette définition rend le diagramme suivant commutatif :
$$
\xymatrix{
\DeuxCatDeuxOp{\mathdeuxcat{A}}
\ar[d]_{\DeuxFoncDeuxOp{u}}
&S_{2}(\mathdeuxcat{A})
\ar[l]_{s_{2}^{\mathdeuxcat{A}}}
\ar[d]^{S_{2}(u)}
\ar[r]^{t_{2}^{\mathdeuxcat{A}}}
&\mathdeuxcat{A}
\ar[d]^{u}
\\
\DeuxCatDeuxOp{\mathdeuxcat{B}}
&S_{2}(\mathdeuxcat{B}{})
\ar[l]^{s_{2}^{\mathdeuxcat{B}{}}}
\ar[r]_{t_{2}^{\mathdeuxcat{B}{}}}
&\mathdeuxcat{B}{}
&.
}
$$
\end{paragr}

Comme les flèches horizontales de ce diagramme sont toutes des équivalences faibles, en vertu de ce qui précède, on en déduit la proposition \ref{DeuxFoncDeuxOpW} par deux applications consécutives d'un argument de « $2$ sur $3$ ».

\begin{prop}\label{DeuxFoncDeuxOpW}
Un morphisme $u$ de $\DeuxCat$ est une équivalence faible si et seulement si $\DeuxFoncDeuxOp{u}$ en est une.
\end{prop}

\begin{corollaire}\label{DeuxCatDeuxOpAsph}
Une petite \deux{}catégorie $\mathdeuxcat{A}$ est asphérique si et seulement si $\DeuxCatDeuxOp{\mathdeuxcat{A}}$ est asphérique.
\end{corollaire}

\begin{prop}\label{DeuxFoncToutOpW}
Un morphisme $u$ de $\DeuxCat$ est une équivalence faible si et seulement si $\DeuxFoncToutOp{u}$ en est une.
\end{prop}

\begin{proof}
Cela résulte des propositions \ref{DeuxFoncUnOpW} et \ref{DeuxFoncDeuxOpW}.
\end{proof}

\begin{corollaire}\label{DeuxCatToutOpAsph}
Une petite \deux{}catégorie $\mathdeuxcat{A}$ est asphérique si et seulement si $\DeuxCatToutOp{\mathdeuxcat{A}}$ est asphérique.
\end{corollaire}

\begin{prop}\label{TheoremeATrancheLax}
Soit 
$$
\xymatrix{
\mathdeuxcat{A} 
\ar[rr]^{u}
\ar[dr]_{w}
&&\mathdeuxcat{B}
\ar[dl]^{v}
\\
&\mathdeuxcat{C}
}
$$
un diagramme commutatif dans $\DeuxCat$. Supposons que, pour tout objet $c$ de $\mathdeuxcat{C}$, le morphisme
$$
\DeuxFoncTrancheLax{u}{c} : \TrancheLax{\mathdeuxcat{A}}{w}{c} \to \TrancheLax{\mathdeuxcat{B}}{v}{c} 
$$
soit une équivalence faible. Alors $u$ est une équivalence faible.
\end{prop}

\begin{proof}
En vertu des hypothèses et par définition, le \DeuxFoncteurStrict{} $\DeuxFoncDeuxOp{(\DeuxFoncTrancheCoLax{\DeuxFoncDeuxOp{u}}{c})}$ est une équivalence faible pour tout objet $c$ de $\mathdeuxcat{C}$. Par conséquent, en vertu de la proposition \ref{DeuxFoncDeuxOpW}, le \DeuxFoncteurStrict{} $\DeuxFoncTrancheCoLax{\DeuxFoncDeuxOp{u}}{c}$ est une équivalence faible pour tout objet $c$ de $\mathdeuxcat{C}$. En vertu de la condition LF3, le \DeuxFoncteurStrict{} $\DeuxFoncDeuxOp{u}$ est donc une équivalence faible. Une nouvelle invocation de la proposition \ref{DeuxFoncDeuxOpW} permet de conclure.
\end{proof} 

\begin{prop}\label{TheoremeAOpTrancheLax}
Soit 
$$
\xymatrix{
\mathdeuxcat{A} 
\ar[rr]^{u}
\ar[dr]_{w}
&&\mathdeuxcat{B}
\ar[dl]^{v}
\\
&\mathdeuxcat{C}
}
$$
un diagramme commutatif dans $\DeuxCat$. Supposons que, pour tout objet $c$ de $\mathdeuxcat{C}$, le morphisme
$$
\DeuxFoncOpTrancheLax{u}{c} : \OpTrancheLax{\mathdeuxcat{A}}{w}{c} \to \OpTrancheLax{\mathdeuxcat{B}}{v}{c} 
$$
soit une équivalence faible. Alors $u$ est une équivalence faible.
\end{prop}

\begin{proof}
En vertu des hypothèses et par définition, le \DeuxFoncteurStrict{} $\DeuxFoncUnOp{(\DeuxFoncTrancheLax{\DeuxFoncUnOp{u}}{c})}$ est une équivalence faible pour tout objet $c$ de $\mathdeuxcat{C}$. Par conséquent, en vertu de la proposition \ref{DeuxFoncUnOpW}, le \DeuxFoncteurStrict{} $\DeuxFoncTrancheLax{\DeuxFoncUnOp{u}}{c}$ est une équivalence faible pour tout objet $c$ de $\mathdeuxcat{C}$. En vertu de la proposition \ref{TheoremeATrancheLax}, le \DeuxFoncteurStrict{} $\DeuxFoncUnOp{u}$ est donc une équivalence faible. Une nouvelle invocation de la proposition \ref{DeuxFoncUnOpW} permet de conclure.
\end{proof} 

\begin{prop}\label{PreadjointsW}
Un préadjoint à gauche lax (resp. un préadjoint à droite lax, resp. un préadjoint à droite colax) est une équivalence faible.
\end{prop}

\begin{proof}
C'est une conséquence immédiate de l'exemple \ref{ExemplesOF}, du corollaire \ref{OFAspherique3} et de la proposition \ref{TheoremeATrancheLax} (resp. de l'exemple \ref{ExemplesOF}, du lemme \ref{OFAspherique2} et de la proposition \ref{TheoremeAOpTrancheLax}, resp. de l'exemple \ref{ExemplesOF}, du lemme \ref{OFAspherique2Bis} et de la proposition \ref{TheoremeAOpTrancheCoLax}). 
\end{proof}

%\begin{rem}
%Le lemme \ref{PreadjointW} et la proposition \ref{PreadjointsW} constituent la généralisation, à tout localisateur fondamental de $\DeuxCat$, de \cite[corollaire 2.35]{ArticleThAMoi}.
%\end{rem}

\begin{prop}\label{PrefibrationFibresAspheriquesW}
Une préfibration (resp. une préopfibration, resp. une précofibration, resp. une précoopfibration) à fibres asphériques est une équivalence faible.
\end{prop}

\begin{proof}
Soit $u : \mathdeuxcat{A} \to \mathdeuxcat{B}$ une préfibration à fibres asphériques. Par définition, pour tout objet $b$ de $\mathdeuxcat{B}$, le \DeuxFoncteurStrict{} canonique $J_{b} : \Fibre{\mathdeuxcat{A}}{u}{b} \to \OpTrancheCoLax{\mathdeuxcat{A}}{u}{b}$ est un préadjoint à gauche lax. En vertu de la proposition \ref{PreadjointsW}, c'est donc une équivalence faible. Par conséquent, $\OpTrancheCoLax{\mathdeuxcat{A}}{u}{b}$ est asphérique pour tout objet $b$ de $\mathdeuxcat{B}$. On conclut par une invocation de la proposition \ref{TheoremeAOpTrancheCoLax}. Les trois autres assertions s'en déduisent en invoquant respectivement les propositions \ref{DeuxFoncUnOpW}, \ref{DeuxFoncDeuxOpW} et \ref{DeuxFoncToutOpW} et leur corollaire respectif \ref{DeuxCatUnOpAsph}, \ref{DeuxCatDeuxOpAsph} et \ref{DeuxCatToutOpAsph}.
%assaisonné de la remarque \ref{FibresCoOp}. 
\end{proof}

\begin{rem}
La proposition \ref{PrefibrationFibresAspheriquesW} constitue la généralisation, à tout localisateur fondamental de $\DeuxCat$, de \cite[lemme 2.43]{ArticleThAMoi}.
\end{rem}

\begin{df}\label{Zonzon}
Étant donné un diagramme commutatif 
$$
\xymatrix{
\mathdeuxcat{A} 
\ar[rr]^{u}
\ar[dr]_{w}
&&\mathdeuxcat{B}
\ar[dl]^{v}
\\
&\mathdeuxcat{C}
}
$$
dans $\DeuxCat$, on dira que $u$ est \emph{lax-asphérique au-dessus de $\mathdeuxcat{C}$} (resp. \emph{lax-opasphérique au-dessus de $\mathdeuxcat{C}$}, resp. \emph{colax-asphérique au-dessus de $\mathdeuxcat{C}$}, resp. \emph{colax-opasphérique au-dessus de $\mathdeuxcat{C}$}) si, pour tout objet $c$ de $\mathdeuxcat{C}$, le \DeuxFoncteurStrict{} $\DeuxFoncTrancheLax{u}{c}$ (resp. $\DeuxFoncOpTrancheLax{u}{c}$, resp. $\DeuxFoncTrancheCoLax{u}{}{c}$, resp. $\DeuxFoncOpTrancheCoLax{u}{c}$) est une équivalence faible. Si $v = 1_{\mathdeuxcat{B}}$, on dira simplement\footnote{En accord avec l'emploi du terme « colax-asphérique », déjà introduit.} que $u$ est \emph{lax-asphérique} (resp. \emph{lax-opasphérique}, resp. \emph{colax-asphérique}, resp. \emph{colax-opasphérique}).
\end{df}

\begin{rem}\label{Zozo}
En vertu des résultats ci-dessus, sous les données de la définition \ref{Zonzon}, le \DeuxFoncteurStrict{} $u$ est une équivalence faible pour peu qu'il soit lax-asphérique, lax-opasphérique, colax-asphérique ou colax-opasphérique au-dessus de $\mathdeuxcat{C}$. En faisant $v = 1_{\mathdeuxcat{B}}$, on obtient le cas particulier suivant : pour tout \DeuxFoncteurStrict{} $u : \mathdeuxcat{A} \to \mathdeuxcat{B}$, si, pour tout objet $b$ de $\mathdeuxcat{B}$, la \deux{}catégorie $\TrancheLax{\mathdeuxcat{A}}{u}{b}$ (resp. $\OpTrancheLax{\mathdeuxcat{A}}{u}{b}$, resp. $\TrancheCoLax{\mathdeuxcat{A}}{u}{b}$, resp. $\OpTrancheCoLax{\mathdeuxcat{A}}{u}{b}$) est asphérique, alors $u$ est une équivalence faible. 
\end{rem}

\begin{prop}\label{2.1.10.THG}
Soit
$$
\xymatrix{
\mathdeuxcat{A} 
\ar[rr]^{u}
\ar[dr]_{w}
&&\mathdeuxcat{B}
\ar[dl]^{v}
\\
&\mathdeuxcat{C}
}
$$
un diagramme commutatif dans $\DeuxCat$. On suppose que, pour tout objet $c$ de $\mathdeuxcat{C}$, le \DeuxFoncteurStrict{} induit entre les fibres $u_{c} : \Fibre{\mathdeuxcat{A}}{w}{c} \to \Fibre{\mathdeuxcat{B}}{v}{c}$ est une équivalence faible.
\begin{itemize}
\item[\rm{(a)}]
Si $v$ et $w$ sont des préfibrations, alors $u$ est colax-opasphérique au-dessus de $\mathdeuxcat{C}$. 
\item[\rm{(b)}]
Si $v$ et $w$ sont des préopfibrations, alors $u$ est colax-asphérique au-dessus de $\mathdeuxcat{C}$. 
\item[\rm{(c)}]
Si $v$ et $w$ sont des précofibrations, alors $u$ est lax-opasphérique au-dessus de $\mathdeuxcat{C}$.
\item[\rm{(d)}]
Si $v$ et $w$ sont des précoopfibrations, alors $u$ est lax-asphérique au-dessus de $\mathdeuxcat{C}$. 
\end{itemize}
En particulier, dans n'importe lequel de ces quatre cas, $u$ est une équivalence faible.
\end{prop}

\begin{proof}
Plaçons-nous dans le premier cas, les trois autres s'en déduisant par dualité. Pour tout objet $c$ de $\mathdeuxcat{C}$, on a un carré commutatif
$$
\xymatrix{
\Fibre{\mathdeuxcat{A}}{w}{c}
\ar[r]^{u_{c}}
\ar[d]_{J_{c}}
&
\Fibre{\mathdeuxcat{B}}{v}{c}
\ar[d]^{J_{c}}
\\
\OpTrancheCoLax{\mathdeuxcat{A}}{w}{c}
\ar[r]_{\DeuxFoncOpTrancheCoLax{u}{c}}
&
\OpTrancheCoLax{\mathdeuxcat{B}}{v}{c}
}
$$
dont les flèches verticales sont des préadjoints à gauche lax, donc des équivalences faibles (proposition \ref{PreadjointsW}). Comme, par hypothèse, $u_{c}$ est une équivalence faible, il en est de même de $\DeuxFoncOpTrancheCoLax{u}{c}$.
\end{proof}

\emph{On ne suppose plus fixé de \ClasseDeuxLocFond{}.}

\begin{theo}\label{DeuxLocFondHuitDef}
Soit $\DeuxLocFond{W}$ une partie de $\UnCell{\DeuxCat}$. Les conditions suivantes sont équivalentes.
\begin{itemize}
\item[(i)] 
$\DeuxLocFond{W}$ est un \ClasseDeuxLocFond{}.
\item[(ii)]
Les conditions suivantes sont vérifiées. 
\begin{itemize}
\item[LF1$'$] 
La partie $\DeuxLocFond{W}$ de $\UnCell{\DeuxCat}$ est faiblement saturée. 
\item[LF2$'$] 
Si une petite \deux{}catégorie $\mathdeuxcat{A}$ admet un objet admettant un objet initial, alors $\mathdeuxcat{A} \to \DeuxCatPonct$ est dans $\DeuxLocFond{W}$. 
\item[LF3$'$] Si 
$$
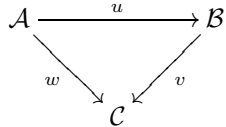
\xymatrix{
\mathdeuxcat{A} 
\ar[rr]^{u}
\ar[dr]_{w}
&&\mathdeuxcat{B}
\ar[dl]^{v}
\\
&\mathdeuxcat{C}
}
$$
désigne un triangle commutatif dans $\DeuxCat$ et si, pour tout objet $c$ de $\mathdeuxcat{C}$, le \deux{}foncteur strict 
$$
\DeuxFoncTrancheLax{u}{c} : \TrancheLax{\mathdeuxcat{A}}{w}{c} \to \TrancheLax{\mathdeuxcat{B}}{v}{c} 
$$
est dans $\DeuxLocFond{W}$, alors $u$ est dans $\DeuxLocFond{W}$.
\end{itemize}
\item[(iii)]
Les conditions suivantes sont vérifiées.
\begin{itemize}
\item[LF1$''$] 
La partie $\DeuxLocFond{W}$ de $\UnCell{\DeuxCat}$ est faiblement saturée. 
\item[LF2$''$] 
Si une petite \deux{}catégorie $\mathdeuxcat{A}$ op-admet un objet admettant un objet final, alors $\mathdeuxcat{A} \to \DeuxCatPonct$ est dans $\DeuxLocFond{W}$. 
\item[LF3$''$] Si 
$$
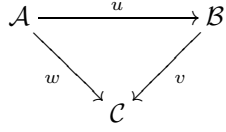
\xymatrix{
\mathdeuxcat{A} 
\ar[rr]^{u}
\ar[dr]_{w}
&&\mathdeuxcat{B}
\ar[dl]^{v}
\\
&\mathdeuxcat{C}
}
$$
désigne un triangle commutatif dans $\DeuxCat$ et si, pour tout objet $c$ de $\mathdeuxcat{C}$, le \deux{}foncteur strict
$$
\DeuxFoncOpTrancheCoLax{u}{c} : \OpTrancheCoLax{\mathdeuxcat{A}}{w}{c} \to \OpTrancheCoLax{\mathdeuxcat{B}}{v}{c} 
$$
est dans $\DeuxLocFond{W}$, alors $u$ est dans $\DeuxLocFond{W}$.
\end{itemize}
\item[(iv)]
Les conditions suivantes sont vérifiées. 
\begin{itemize}
\item[LF1$'''$] 
La partie $\DeuxLocFond{W}$ de $\UnCell{\DeuxCat}$ est faiblement saturée. 
\item[LF2$'''$] 
Si une petite \deux{}catégorie $\mathdeuxcat{A}$ op-admet un objet admettant un objet initial, alors $\mathdeuxcat{A} \to \DeuxCatPonct$ est dans $\DeuxLocFond{W}$. 
\item[LF3$'''$] Si 
$$
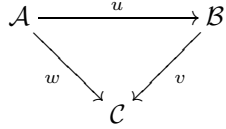
\xymatrix{
\mathdeuxcat{A} 
\ar[rr]^{u}
\ar[dr]_{w}
&&\mathdeuxcat{B}
\ar[dl]^{v}
\\
&\mathdeuxcat{C}
}
$$
désigne un triangle commutatif dans $\DeuxCat$ et si, pour tout objet $c$ de $\mathdeuxcat{C}$, le \deux{}foncteur strict
$$
\DeuxFoncOpTrancheLax{u}{c} : \OpTrancheLax{\mathdeuxcat{A}}{w}{c} \to \OpTrancheLax{\mathdeuxcat{B}}{v}{c} 
$$
est dans $\DeuxLocFond{W}$, alors $u$ est dans $\DeuxLocFond{W}$.
\end{itemize}
\item[(v)]
Les conditions suivantes sont vérifiées.
\begin{itemize}
\item[LF$\alpha$] 
La partie $\DeuxLocFond{W}$ de $\UnCell{\DeuxCat}$ est faiblement saturée. 
\item[LF$\beta$] 
Le morphisme canonique $[1] \to \DeuxCatPonct$ est dans $\DeuxLocFond{W}$. 
\item[LF$\gamma$] Si 
$$
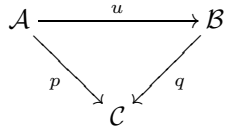
\xymatrix{
\mathdeuxcat{A} 
\ar[rr]^{u}
\ar[dr]_{p}
&&\mathdeuxcat{B}
\ar[dl]^{q}
\\
&\mathdeuxcat{C}
}
$$
désigne un triangle commutatif dans $\DeuxCat$, si $p$ et $q$ sont des précoopfibrations et si, pour tout objet $c$ de $\mathdeuxcat{C}$, le \DeuxFoncteurStrict{} induit entre les fibres
$$
\Fibre{\mathdeuxcat{A}}{p}{c} \to \Fibre{\mathdeuxcat{B}}{q}{c} 
$$
est dans $\DeuxLocFond{W}$, alors $u$ est dans $\DeuxLocFond{W}$.
\end{itemize}
\item[(vi)]
Les conditions suivantes sont vérifiées. 
\begin{itemize}
\item[LF$\alpha'$] 
La partie $\DeuxLocFond{W}$ de $\UnCell{\DeuxCat}$ est faiblement saturée. 
\item[LF$\beta'$] 
Le morphisme canonique $[1] \to \DeuxCatPonct$ est dans $\DeuxLocFond{W}$. 
\item[LF$\gamma'$] Si 
$$
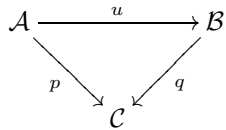
\xymatrix{
\mathdeuxcat{A} 
\ar[rr]^{u}
\ar[dr]_{p}
&&\mathdeuxcat{B}
\ar[dl]^{q}
\\
&\mathdeuxcat{C}
}
$$
désigne un triangle commutatif dans $\DeuxCat$, si $p$ et $q$ sont des précofibrations et si, pour tout objet $c$ de $\mathdeuxcat{C}$, le \DeuxFoncteurStrict{} induit entre les fibres
$$
\Fibre{\mathdeuxcat{A}}{p}{c} \to \Fibre{\mathdeuxcat{B}}{q}{c} 
$$
est dans $\DeuxLocFond{W}$, alors $u$ est dans $\DeuxLocFond{W}$.
\end{itemize}
\item[(vii)]
Les conditions suivantes sont vérifiées.
\begin{itemize}
\item[LF$\alpha''$] 
La partie $\DeuxLocFond{W}$ de $\UnCell{\DeuxCat}$ est faiblement saturée. 
\item[LF$\beta''$] 
Le morphisme canonique $[1] \to \DeuxCatPonct$ est dans $\DeuxLocFond{W}$. 
\item[LF$\gamma''$] Si 
$$
\xymatrix{
\mathdeuxcat{A} 
\ar[rr]^{u}
\ar[dr]_{p}
&&\mathdeuxcat{B}
\ar[dl]^{q}
\\
&\mathdeuxcat{C}
}
$$
désigne un triangle commutatif dans $\DeuxCat$, si $p$ et $q$ sont des préopfibrations et si, pour tout objet $c$ de $\mathdeuxcat{C}$, le \DeuxFoncteurStrict{} induit entre les fibres
$$
\Fibre{\mathdeuxcat{A}}{p}{c} \to \Fibre{\mathdeuxcat{B}}{q}{c} 
$$
est dans $\DeuxLocFond{W}$, alors $u$ est dans $\DeuxLocFond{W}$.
\end{itemize}
\item[(viii)]
Les conditions suivantes sont vérifiées.
\begin{itemize}
\item[LF$\alpha'''$] 
La partie $\DeuxLocFond{W}$ de $\UnCell{\DeuxCat}$ est faiblement saturée. 
\item[LF$\beta'''$] 
Le morphisme canonique $[1] \to \DeuxCatPonct$ est dans $\DeuxLocFond{W}$. 
\item[LF$\gamma'''$] Si 
$$
\xymatrix{
\mathdeuxcat{A} 
\ar[rr]^{u}
\ar[dr]_{p}
&&\mathdeuxcat{B}
\ar[dl]^{q}
\\
&\mathdeuxcat{C}
}
$$
désigne un triangle commutatif dans $\DeuxCat$, si $p$ et $q$ sont des préfibrations et si, pour tout objet $c$ de $\mathdeuxcat{C}$, le \DeuxFoncteurStrict{} induit entre les fibres
$$
\Fibre{\mathdeuxcat{A}}{p}{c} \to \Fibre{\mathdeuxcat{B}}{q}{c} 
$$
est dans $\DeuxLocFond{W}$, alors $u$ est dans $\DeuxLocFond{W}$.
\end{itemize}
\end{itemize}
\end{theo}

\begin{proof}
L'implication $(i) \Rightarrow (v)$ résulte de la proposition \ref{2.1.10.THG}. 

Montrons l'implication $(v) \Rightarrow (iii)$. 

Soit $\mathdeuxcat{A}$ une petite \deux{}catégorie. Les flèches du triangle commutatif
$$
\xymatrix{
[1] \times \mathdeuxcat{A} 
\ar[rr]^{pr_{2}}
\ar[dr]_{pr_{2}}
&&\mathdeuxcat{A}
\ar[dl]^{1_{\mathdeuxcat{A}}}
\\
&\mathdeuxcat{A}
}
$$
sont des précoopfibrations et, pour tout objet $a$ de $\mathdeuxcat{A}$, le \DeuxFoncteurStrict{} induit entre les fibres au-dessus de $a$ s'identifie au morphisme canonique $[1] \to \DeuxCatPonct$, qui est dans $\DeuxLocFond{W}$ par hypothèse. Ainsi, pour toute petite \deux{}catégorie $\mathdeuxcat{A}$, la projection canonique $[1] \times \mathdeuxcat{A} \to \mathdeuxcat{A}$ est dans $\DeuxLocFond{W}$. 

Soit $\mathdeuxcat{A}$ une petite \deux{}catégorie op-admettant un objet admettant un objet final. On construit facilement (voir la démonstration de \cite[lemme 2.27]{ArticleThAMoi}) un endomorphisme constant de $\mathdeuxcat{A}$ homotope à $1_{\mathdeuxcat{A}}$. La \deux{}catégorie $\mathdeuxcat{A}$ est donc contractile. En vertu de ce qui précède et de \cite[lemme 1.4.8]{THG}, le morphisme canonique $\mathdeuxcat{A} \to \DeuxCatPonct$ est dans $\DeuxLocFond{W}$. La condition LF$2''$ est donc vérifiée.

Pour tout morphisme $u : \mathdeuxcat{A} \to \mathdeuxcat{B}$ de $\DeuxCat$, on peut considérer les \DeuxFoncteursStricts{}
$$
\begin{aligned}
\DeuxCatUnOp{\mathdeuxcat{B}} &\to \DeuxCatDeuxCat
\\
b &\mapsto \OpTrancheCoLax{\mathdeuxcat{A}}{u}{b}
\end{aligned}
$$
et 
$$
\begin{aligned}
\mathdeuxcat{A} &\to \DeuxCatDeuxCat
\\
a &\mapsto \DeuxCatUnOp{\left(\TrancheCoLax{\mathdeuxcat{B}}{}{u(a)}\right)}.
\end{aligned}
$$
Considérons la \deux{}catégorie $S(u)$ définie comme suit. Ses objets sont les triplets $(b, a , k : b \to u(a))$ avec $b$ un objet de $\mathdeuxcat{B}$, $a$ un objet de $\mathdeuxcat{A}$ et $k$ une \un{}cellule de $\mathdeuxcat{B}$. Les \un{}cellules de $(b,a,k)$ vers $(b',a',k')$ sont les triplets $(f : b' \to b, g : a \to a', \alpha : u(g) k f \Rightarrow k')$, $f$ étant une \un{}cellule de $\mathdeuxcat{B}$, $g$ une \un{}cellule de $\mathdeuxcat{A}$ et $\alpha$ une \deux{}cellule de $\mathdeuxcat{B}$. Les \deux{}cellules de $(f,g,\alpha)$ vers $(f',g',\alpha')$ sont les couples $(\varphi : f \Rightarrow f', \gamma : g \Rightarrow g')$ avec $\varphi$ une \deux{}cellule de $\mathdeuxcat{B}$ et $\gamma$ une \deux{}cellule de $\mathdeuxcat{A}$ telles que $\alpha' \CompDeuxUn (u(\gamma) \CompDeuxZero k \CompDeuxZero \varphi) = \alpha$. On vérifie l'existence d'isomorphismes canoniques
$$
S(u) \simeq \DeuxInt{\DeuxCatUnOp{\mathdeuxcat{B}}} \OpTrancheCoLax{\mathdeuxcat{A}}{u}{b} \simeq \DeuxInt{\mathdeuxcat{A}} \DeuxCatUnOp{\left(\TrancheCoLax{\mathdeuxcat{B}}{}{u(a)}\right)}.
$$
Les projections canoniques
$$
\begin{aligned}
s_{u} : S(u) &\to \DeuxCatUnOp{\mathdeuxcat{B}}
\\
(b,a,k) &\mapsto b
\\
(f,g,\alpha) &\mapsto f
\\
(\varphi, \gamma) &\mapsto \varphi
\end{aligned}
$$
et
$$
\begin{aligned}
t_{u} : S(u) &\to \mathdeuxcat{A}
\\
(b,a,k) &\mapsto a
\\
(f,g,\alpha) &\mapsto g
\\
(\varphi, \gamma) &\mapsto \gamma
\end{aligned}
$$
sont donc des précoopfibrations. De plus, pour tout objet $a$ de $\mathdeuxcat{A}$, la \deux{}catégorie $\DeuxCatUnOp{\left(\TrancheCoLax{\mathdeuxcat{B}}{}{u(a)}\right)}$ op-admet un objet admettant un objet final. En vertu de ce qui précède, le morphisme canonique $\DeuxCatUnOp{\left(\TrancheCoLax{\mathdeuxcat{B}}{}{u(a)}\right)} \to \DeuxCatPonct$ est donc dans $\DeuxLocFond{W}$. Or, c'est à ce morphisme que s'identifie le \DeuxFoncteurStrict{} induit entre les fibres au-dessus de $a$ par le diagramme
$$
\xymatrix{
S(u)
\ar[rr]^{t_{u}}
\ar[dr]_{t_{u}}
&&\mathdeuxcat{A}
\ar[dl]^{1_{\mathdeuxcat{A}}}
\\
&\mathdeuxcat{A}
&.
}
$$
Comme $t_{u}$ et $1_{\mathdeuxcat{A}}$ sont des précoopfibrations, $t_{u}$ est dans $\DeuxLocFond{W}$. 

Soit 
$$
\xymatrix{
\mathdeuxcat{A}
\ar[r]^{u}
\ar[d]_{v}
&
\mathdeuxcat{B}
\ar[d]^{w}
\\
\mathdeuxcat{A'}
\ar[r]_{u'}
&
\mathdeuxcat{B'}
}
$$
un carré commutatif dans $\DeuxCat$. On définit un \DeuxFoncteurStrict{}
$$
\begin{aligned}
S(v,w) : S(u) &\to S(u')
\\
(b,a,k) &\mapsto (w(b), v(a), w(k))
\\
(f, g, \alpha) &\mapsto (w(f), v(g), w(\alpha))
\\
(\varphi, \gamma) &\mapsto (w(\varphi), v(\gamma)).
\end{aligned}
$$
Cela fournit un diagramme commutatif
$$
\xymatrix{
\DeuxCatUnOp{\mathdeuxcat{B}}
\ar[d]_{\DeuxFoncUnOp{w}}
&
S(u)
\ar[l]_{s_{u}}
\ar[r]^{t_{u}}
\ar[d]^{S(v,w)}
&
\mathdeuxcat{A}
\ar[d]^{v}
\\
\DeuxCatUnOp{\mathdeuxcat{B}'}
&
S(u')
\ar[l]^{s_{u'}}
\ar[r]_{t_{u'}}
&
\mathdeuxcat{A'}
}
$$
dans lequel $t_{u}$ et $t_{u'}$ sont dans $\DeuxLocFond{W}$. 

Soit
$$
\xymatrix{
\mathdeuxcat{A}
\ar[rr]^{u}
\ar[dr]_{w}
&&\mathdeuxcat{B}
\ar[dl]^{v}
\\
&\mathdeuxcat{C}
}
$$
un diagramme commutatif dans $\DeuxCat$ tel que, pour tout objet $c$ de $\mathdeuxcat{C}$, le \DeuxFoncteurStrict{} $\DeuxFoncTrancheCoLax{u}{c} : \TrancheCoLax{\mathdeuxcat{A}}{w}{c} \to \TrancheCoLax{\mathdeuxcat{B}}{v}{c}$ soit dans $\DeuxLocFond{W}$. En vertu de ce qui précède, cela fournit un diagramme commutatif 
$$
\xymatrix{
\DeuxCatUnOp{\mathdeuxcat{C}}
\ar[d]_{1_{\DeuxCatUnOp{\mathdeuxcat{C}}}}
&
S(w)
\ar[l]_{s_{w}}
\ar[r]^{t_{w}}
\ar[d]^{S(u,1_{\mathdeuxcat{C}})}
&
\mathdeuxcat{A}
\ar[d]^{u}
\\
\DeuxCatUnOp{\mathdeuxcat{C}}
&
S(v)
\ar[l]^{s_{v}}
\ar[r]_{t_{v}}
&
\mathdeuxcat{B}
&.
}
$$
En particulier, on a un diagramme commutatif
$$
\xymatrix{
S(w)
\ar[rr]^{S(u, 1_{\mathdeuxcat{C}})}
\ar[dr]_{s_{w}}
&&
S(v)
\ar[dl]^{s_{v}}
\\
&
\DeuxCatUnOp{\mathdeuxcat{C}}
}
$$
dans lequel $s_{w}$ et $s_{v}$ sont des précoopfibrations. Pour tout objet $c$ de $\mathdeuxcat{C}$, le \DeuxFoncteurStrict{} induit par ce diagramme entre les fibres au-dessus de $c$ s'identifie à $\DeuxFoncOpTrancheCoLax{u}{c}$, qui est dans $\DeuxLocFond{W}$ par hypothèse. Par conséquent, $S(u, 1_{\mathdeuxcat{C}})$ est dans $\DeuxLocFond{W}$ en vertu de la condition LF$\gamma$. Comme $t_{w}$ et $t_{v}$ sont dans $\DeuxLocFond{W}$, qui est faiblement saturée, $u$ est dans $\DeuxLocFond{W}$, ce qui termine la démonstration de la condition LF$3''$, et donc de l'implication $(v) \Rightarrow (iii)$.

Montrons l'implication $(iii) \Rightarrow (i)$. Notons $\DeuxFoncUnOp{\DeuxLocFond{W}}$ la partie de $\UnCell{\DeuxCat}$ définie par
$$
\DeuxFoncUnOp{\DeuxLocFond{W}} = \{ u \in \UnCell{\DeuxCat} \vert \DeuxFoncUnOp{u} \in \DeuxLocFond{W} \}.
$$
Vérifions que la classe $\DeuxFoncUnOp{\DeuxLocFond{W}}$ constitue un \ClasseDeuxLocFond{}. La saturation faible de $\DeuxFoncUnOp{\DeuxLocFond{W}}$ est immédiate. Soit $\mathdeuxcat{A}$ une petite \deux{}catégorie admettant un objet admettant un objet final. Alors, $\DeuxCatUnOp{\mathdeuxcat{A}}$ op-admet un objet admettant un objet final, donc le morphisme $\DeuxCatUnOp{\mathdeuxcat{A}} \to \DeuxCatPonct$ est dans $\DeuxLocFond{W}$, c'est-à-dire que le morphisme $\mathdeuxcat{A} \to \DeuxCatPonct$ est dans $\DeuxFoncUnOp{\DeuxLocFond{W}}$. Soit 
$$
\xymatrix{
\mathdeuxcat{A}
\ar[rr]^{u}
\ar[dr]_{w}
&&\mathdeuxcat{B}
\ar[dl]^{v}
\\
&\mathdeuxcat{C}
}
$$
un diagramme commutatif dans $\DeuxCat$ tel que, pour tout objet $c$ de $\mathdeuxcat{C}$, le \DeuxFoncteurStrict{} $\DeuxFoncTrancheCoLax{u}{c}$ soit dans $\DeuxFoncUnOp{\DeuxLocFond{W}}$, c'est-à-dire tel que le \DeuxFoncteurStrict{} $\DeuxFoncUnOp{(\DeuxFoncTrancheCoLax{u}{c})}$ soit dans $\DeuxLocFond{W}$, c'est-à-dire tel que le \DeuxFoncteurStrict{} $\DeuxFoncUnOp{(\DeuxFoncTrancheCoLax{\DeuxFoncUnOp{(\DeuxFoncUnOp{u})}}{c})}$ soit dans $\DeuxLocFond{W}$. Pour tout objet $c$ de $\mathdeuxcat{C}$, le \DeuxFoncteurStrict{} $\DeuxFoncOpTrancheCoLax{\DeuxFoncUnOp{u}}{c}$ est donc dans $\DeuxLocFond{W}$. La condition LF3$'$ implique que $\DeuxFoncUnOp{u}$ est dans $\DeuxLocFond{W}$, c'est-à-dire que $u$ est dans $\DeuxFoncUnOp{\DeuxLocFond{W}}$, qui est donc bien un \ClasseDeuxLocFond{}. On en déduit $\DeuxLocFond{W} = \DeuxFoncUnOp{\DeuxLocFond{W}}$ en vertu de la proposition \ref{DeuxFoncUnOpW}. L'implication considérée en résulte. 

Les autres se démontrent de façon analogue ou se déduisent de ce qui précède par un argument de dualité.   
\end{proof}

\section{Correspondances fondamentales}

\emph{On suppose fixé un localisateur fondamental $\DeuxLocFond{W}$ de $\DeuxCat$.}

\begin{prop}\label{CouniteColaxAspherique}
Pour toute petite \deux{}catégorie $\mathdeuxcat{A}$, le \DeuxFoncteurStrict{} $\StrictCanonique{\mathdeuxcat{A}} : \TildeLax{\mathdeuxcat{A}} \to \mathdeuxcat{A}$ est une équivalence faible.
\end{prop}

\begin{proof}
C'est une conséquence immédiate du lemme \ref{LemmeTranchesCouniteAspheriques} et de la condition LF2.
\end{proof}

\begin{rem}
La proposition \ref{CouniteColaxAspherique} constitue la généralisation, à tout localisateur fondamental de $\DeuxCat$, de \cite[proposition 5.10]{ArticleThAMoi}.
\end{rem}

\begin{lemme}\label{BarreTildeW}
Soit $u : \mathdeuxcat{A} \to \mathdeuxcat{B}$ un \deux{}foncteur strict. Les propositions suivantes sont équivalentes. 
\begin{itemize}
\item[(i)]
$u$ est une équivalence faible.
\item[(ii)]
$\TildeLax{u}$ est une équivalence faible.
\end{itemize}
\end{lemme}

\begin{proof}
Cela résulte de l'égalité $\StrictCanonique{\mathdeuxcat{B}} \TildeLax{u} = u \StrictCanonique{\mathdeuxcat{A}}$, de la proposition \ref{CouniteColaxAspherique} et d'un argument de « 2 sur 3 ». 
\end{proof}

\begin{df}\label{DfEquiLax}
On dira qu'un morphisme $u$ de $\DeuxCatLax$ est une $\DeuxLocFond{W}$\emph{-équivalence faible lax} ou, plus simplement, une \emph{équivalence faible lax}, voire, en l'absence d'ambiguïté, une \emph{équivalence faible}, si $\TildeLax{u}$ est dans $\DeuxLocFond{W}$. On notera $\DeuxLocFondLaxInduit{W}$ la classe des $\DeuxLocFond{W}$\nobreakdash-équivalences faibles lax.
\end{df}

Dans la suite, on commettra l'abus sans conséquence de considérer $\DeuxCat$ comme une sous-catégorie (non pleine) de $\DeuxCatLax$. 

\begin{rem}\label{EquiStricteEquiLax}
En vertu du lemme \ref{BarreTildeW}, un \deux{}foncteur strict est une équivalence faible si et seulement si c'est une équivalence faible lax.
\end{rem}

\begin{rem}\label{SatFaibleEquiLax}
Par fonctorialité, la classe des équivalences faibles lax est faiblement saturée.
\end{rem}

\begin{lemme}\label{UniteW}
Pour toute petite \deux{}catégorie $\mathdeuxcat{A}$, le \DeuxFoncteurLax{} $\LaxCanonique{\mathdeuxcat{A}} : \mathdeuxcat{A} \to \TildeLax{\mathdeuxcat{A}}$ est une équivalence faible.
\end{lemme}

\begin{proof}
Cela résulte de l'égalité $\StrictCanonique{\mathdeuxcat{A}} \LaxCanonique{\mathdeuxcat{A}} = 1_{\mathdeuxcat{A}}$, du fait que $\StrictCanonique{\mathdeuxcat{A}}$ est une équivalence faible et de la saturation faible de la classe des équivalences faibles lax. 
\end{proof}

\begin{rem}
Le lemme \ref{UniteW} constitue la généralisation, à tout localisateur fondamental de $\DeuxCat$, de \cite[proposition 5.11]{ArticleThAMoi}.
\end{rem}

\begin{theo}\label{EqCatLocDeuxCatDeuxCatLax}
L'inclusion $\DeuxCat \hookrightarrow \DeuxCatLax$ induit une équivalence de catégories entre les catégories localisées $\Localisation{\DeuxCat}{\DeuxLocFond{W}}$ et $\Localisation{\DeuxCatLax}{\DeuxLocFondLaxInduit{W}}$.
\end{theo}

\begin{proof}
C'est une conséquence du fait que les composantes des transformations naturelles $\TransLaxCanonique$ et $\TransStrictCanonique$ sont dans $\DeuxLocFondLaxInduit{W}$ et $\DeuxLocFond{W}$ respectivement.
\end{proof}

\begin{rem}
Le théorème \ref{EqCatLocDeuxCatDeuxCatLax} constitue la généralisation, à tout localisateur fondamental de $\DeuxCat$, de \cite[théorème 5.15]{ArticleThAMoi}.
\end{rem}

\begin{theo}\label{TheoremeALaxTrancheLax}
Soit
$$
\xymatrix{
\mathdeuxcat{A} 
\ar[rr]^{u}
\ar[dr]_{w}
&&\mathdeuxcat{B}
\ar[dl]^{v}
\\
&\mathdeuxcat{C}
}
$$
un diagramme commutatif dans $\DeuxCatLax$. Supposons que, pour tout objet $c$ de $\mathdeuxcat{C}$, le \DeuxFoncteurLax{} $\DeuxFoncTrancheLax{u}{c} : \TrancheLax{\mathdeuxcat{A}}{w}{c} \to \TrancheLax{\mathdeuxcat{B}}{v}{c}$ soit une équivalence faible. Alors $u$ est une équivalence faible.
\end{theo}

\begin{proof}
On esquisse la démonstration. Le lecteur pourra se reporter à \cite[théorème 6.6]{ArticleThAMoi} pour les détails d'une démonstration d'un énoncé plus fort dans le cas particulier des équivalences faibles définies par le nerf (voir la remarque \ref{ThAPlusFort}). Supposant donné un diagramme commutatif de \DeuxFoncteursLax{} tel que celui de l'énoncé, on peut lui associer, pour tout objet $c$ de $\mathdeuxcat{C}$, un \DeuxFoncteurStrict{} $\DeuxFoncTrancheLax{\TildeLax{u}}{c} : \TrancheLax{\TildeLax{\mathdeuxcat{A}}}{\BarreLax{w}}{c} \to \TrancheLax{\TildeLax{\mathdeuxcat{B}}}{\BarreLax{v}}{c}$ induit par le diagramme commutatif de \DeuxFoncteursStricts{}
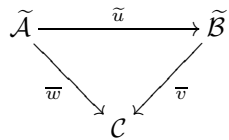
$$
\xymatrix{
\TildeLax{\mathdeuxcat{A}}
\ar[rr]^{\TildeLax{u}}
\ar[dr]_{\BarreLax{w}}
&&\TildeLax{\mathdeuxcat{B}}
\ar[dl]^{\BarreLax{v}}
\\
&\mathdeuxcat{C}
}
$$
dans lequel on a posé $\BarreLax{v} = \StrictCanonique{\mathdeuxcat{C}} \TildeLax{v}$ et $\BarreLax{w} = \StrictCanonique{\mathdeuxcat{C}} \TildeLax{w}$.  

On peut alors vérifier la commutativité du diagramme
$$
\xymatrix{
\TrancheLax{\TildeLax{\mathdeuxcat{A}}}{\BarreLax{w}}{c}
\ar[r]^{\DeuxFoncTrancheLax{\TildeLax{u}}{c}}
&\TrancheLax{\TildeLax{\mathdeuxcat{B}}}{\BarreLax{v}}{c}
\\
\TrancheLax{\mathdeuxcat{A}}{w}{c}
\ar[u]^{\DeuxFoncTrancheLax{\LaxCanonique{\mathdeuxcat{A}}}{c}}
\ar[r]_{\DeuxFoncTrancheLax{u}{c}}
&\TrancheLax{\mathdeuxcat{B}}{v}{c}
\ar[u]_{\DeuxFoncTrancheLax{\LaxCanonique{\mathdeuxcat{B}}}{c}}
&.
}
$$

Les \DeuxFoncteursLax{} $\DeuxFoncTrancheLax{\LaxCanonique{\mathdeuxcat{A}}}{c}$ et $\DeuxFoncTrancheLax{\LaxCanonique{\mathdeuxcat{B}}}{c}$ admettent comme rétraction les \DeuxFoncteursStricts{} $\DeuxFoncTrancheLax{\StrictCanonique{\mathdeuxcat{A}}}{c}$ et $\DeuxFoncTrancheLax{\StrictCanonique{\mathdeuxcat{B}}}{c}$ respectivement ; ces \DeuxFoncteursStricts{} sont des préadjoints à gauche colax, donc des équivalences faibles. Les \DeuxFoncteursLax{} $\DeuxFoncTrancheLax{\LaxCanonique{\mathdeuxcat{A}}}{c}$ et $\DeuxFoncTrancheLax{\LaxCanonique{\mathdeuxcat{B}}}{c}$ sont donc des équivalences faibles en vertu des remarques \ref{EquiStricteEquiLax} et \ref{SatFaibleEquiLax}. Par conséquent, $\DeuxFoncTrancheLax{u}{c}$ est une équivalence faible si et seulement si $\DeuxFoncTrancheLax{\TildeLax{u}}{c}$ en est une, en vertu de ces mêmes remarques. La proposition \ref{TheoremeATrancheLax} permet de conclure.
\end{proof}

\begin{rem}\label{ThAPlusFort}
Le résultat plus général \cite[théorème 6.6]{ArticleThAMoi}, que nous avons déjà mentionné, reste valable pour un localisateur fondamental arbitraire de $\DeuxCat$. Plus précisément, soit
$$
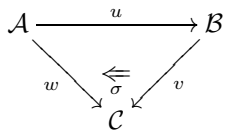
\xymatrix{
\mathdeuxcat{A}  
\ar[rr]^{u} 
\ar[dr]_{w} 
&&\mathdeuxcat{B}
\dtwocell<\omit>{<7.3>\sigma} 
\ar[dl]^{v}
\\ 
& 
\mathdeuxcat{C}
&{}
}
$$
un diagramme dans lequel $u$, $v$ et $w$ sont des \DeuxFoncteursLax{} et $\sigma$ est une \DeuxTransformationCoLax{}. Supposons que, pour tout objet $c$ de $\mathdeuxcat{C}$, le \DeuxFoncteurLax{} 
$$
\DeuxFoncTrancheLaxCoq{u}{\sigma}{c} : \TrancheLax{\mathdeuxcat{A}}{w}{c} \to \TrancheLax{\mathdeuxcat{B}}{v}{c}
$$ 
induit par ces données soit une équivalence faible. Alors $u$ est une équivalence faible. La preuve est l'exact analogue de celle de \cite[théorème 6.6]{ArticleThAMoi}. On peut bien entendu énoncer trois versions duales de ce résultat. 
\end{rem}

\emph{On ne suppose plus fixé de localisateur fondamental de $\DeuxCat$.}

\begin{df}
On dira qu'une classe $\DeuxLocFondLax{W}$ de \DeuxFoncteursLax{} est un \emph{\ClasseDeuxLocFondLax{}} si les conditions suivantes sont vérifiées.
\begin{itemize}
\item[$\rm{LF1}_{lax}$] La classe $\DeuxLocFondLax{W}$ est faiblement saturée. 

\item[$\rm{LF2}_{lax}$] Si une petite \deux{}catégorie $\mathdeuxcat{A}$ admet un objet admettant un objet initial, alors le morphisme canonique $\mathdeuxcat{A} \to e$ est dans $\DeuxLocFondLax{W}$.

\item[$\rm{LF3}_{lax}$] Si 
$$
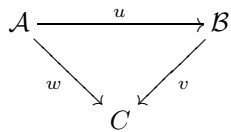
\xymatrix{
\mathdeuxcat{A} 
\ar[rr]^{u}
\ar[dr]_{w}
&&\mathdeuxcat{B}{}
\ar[dl]^{v}
\\
&C
}
$$
désigne un triangle commutatif dans $\DeuxCatLax$ et si, pour tout objet $c$ de $\mathdeuxcat{C}$, le \DeuxFoncteurLax{} 
$
\DeuxFoncTrancheLax{u}{c} : \TrancheLax{\mathdeuxcat{A}}{w}{c} \to \TrancheLax{\mathdeuxcat{B}}{v}{c} 
$
est dans $\DeuxLocFondLax{W}$, alors $u$ est dans $\DeuxLocFondLax{W}$.
\end{itemize}
\end{df}

\begin{rem}\label{RemDeuxLocFondLaxInterCat}
Si $\DeuxLocFondLax{W}$ est un \ClasseDeuxLocFondLax{}, alors $\DeuxLocFondLax{W} \cap \UnCell{\DeuxCat}$ est un \ClasseDeuxLocFond{} et $\DeuxLocFondLax{W} \cap \UnCell{\Cat}$ est un \ClasseUnLocFond{}. 
\end{rem}

\begin{prop}\label{DeuxLocFondInduitLax}
Si $\DeuxLocFond{W}$ est un \ClasseDeuxLocFond{}, alors $\DeuxLocFond{W}_{lax}$ est un \ClasseDeuxLocFondLax{}{}.
\end{prop}

\begin{proof}
En vertu de la remarque \ref{SatFaibleEquiLax}, la classe $\DeuxLocFond{W}_{lax}$ vérifie la condition $\rm{LF1}_{lax}$.

Montrons que $\DeuxLocFond{W}_{lax}$ vérifie $\rm{LF2}_{lax}$. Soit $\mathdeuxcat{A}$ une petite \deux{}catégorie admettant un objet admettant un objet initial. Il s'agit de montrer que le morphisme canonique $\mathdeuxcat{A} \to e$ est dans $\DeuxLocFond{W}_{lax}$. Comme c'est un \DeuxFoncteurStrict{}, il est dans $\DeuxLocFond{W}$ si et seulement s'il est dans $\DeuxLocFond{W}_{lax}$ en vertu de la remarque \ref{EquiStricteEquiLax}. Comme il est dans $\DeuxLocFond{W}$ en vertu du corollaire \ref{OFAspherique3}, il est dans $\DeuxLocFond{W}_{lax}$.

La propriété $\rm{LF3}_{lax}$ résulte du théorème \ref{TheoremeALaxTrancheLax}.
\end{proof}

\begin{lemme}\label{Baba}
Si $\DeuxLocFondLax{W}$ est un \ClasseDeuxLocFondLax{}, alors, pour toute
petite \deux{}caté\-gorie $\mathdeuxcat{A}$, $\StrictCanonique{\mathdeuxcat{A}}$ et $\LaxCanonique{\mathdeuxcat{A}}$ sont dans $\DeuxLocFondLax{W}$.
\end{lemme}

\begin{proof}
Comme la classe $\DeuxLocFondLax{W} \cap \UnCell{\DeuxCat}$ est un \ClasseDeuxLocFond{}, on a $\StrictCanonique{\mathdeuxcat{A}} \in \DeuxLocFondLax{W} \cap \UnCell{\DeuxCat}$ en vertu de la proposition \ref{CouniteColaxAspherique}, donc en particulier $\StrictCanonique{\mathdeuxcat{A}} \in \DeuxLocFondLax{W}$. On en déduit $\LaxCanonique{\mathdeuxcat{A}} \in \DeuxLocFondLax{W}$ en vertu de l'égalité $\StrictCanonique{\mathdeuxcat{A}} \LaxCanonique{\mathdeuxcat{A}} = 1_{\mathdeuxcat{A}}$ et de la condition $\rm{LF1}_{lax}$.
\end{proof}

\begin{lemme}\label{Bibi}
Soit $\DeuxLocFondLax{W}$ un \ClasseDeuxLocFondLax{}. Un morphisme $u : \mathdeuxcat{A} \to \mathdeuxcat{B}$ de $\DeuxCatLax$ est dans $\DeuxLocFondLax{W}$ si et seulement si $\TildeLax{u}$ l'est. 
\end{lemme}

\begin{proof}
Cela résulte de l'égalité $\LaxCanonique{\mathdeuxcat{B}} u = \TildeLax{u} \LaxCanonique{\mathdeuxcat{A}}$, du lemme \ref{Baba} et de la condition $\rm{LF1}_{lax}$.
\end{proof}

\begin{df}
Pour toute classe $\DeuxLocFondLax{W}$ de morphismes de $\DeuxCat$, on notera 
$$
\DeuxLocFondLaxInduit{W} = \{ u \in \UnCell{\DeuxCatLax} \vert \TildeLax{u} \in \DeuxLocFondLax{W} \}.
$$ 
\end{df}

\begin{theo}\label{IsoDeuxLocFondDeuxLocFondLax}
Les applications 
$$
\begin{aligned}
\mathcal{P} (\UnCell{\DeuxCat}) &\to \mathcal{P} (\UnCell{\DeuxCatLax})
\\
\DeuxLocFond{W} &\mapsto \DeuxLocFond{W}_{lax}
\end{aligned}
$$
et
$$
\begin{aligned}
\mathcal{P} (\UnCell{\DeuxCatLax}) &\to \mathcal{P} (\UnCell{\DeuxCat})
\\
\DeuxLocFond{W} &\mapsto \DeuxLocFond{W} \cap \UnCell{\DeuxCat}
\end{aligned}
$$
induisent des isomorphismes inverses l'un de l'autre entre la classe des \ClasseDeuxLocFondS{} ordonnée par inclusion et la classe des \ClasseDeuxLocFondLaxS{} ordonnée par inclusion. De plus, pour tout localisateur fondamental $\DeuxLocFond{W}$ de $\DeuxCat$, les catégories localisées $\Localisation{\DeuxCat}{\DeuxLocFond{W}}$ et $\Localisation{\DeuxCatLax}{\DeuxLocFondLaxInduit{W}}$ sont équivalentes et, pour tout localisateur fondamental $\DeuxLocFondLax{W}$ de $\DeuxCatLax$, les catégories localisées $\Localisation{\DeuxCatLax}{\DeuxLocFondLax{W}}$ et $\Localisation{\DeuxCat}{(\DeuxLocFondLax{W} \cap \UnCell{\DeuxCat})}$ sont équivalentes, ces équivalences de catégories localisées étant induites par l'inclusion $\DeuxCat \hookrightarrow \DeuxCatLax$.
\end{theo}

\begin{proof}
Ces applications respectent manifestement la relation d'inclusion. Il s'agit de vérifier que, pour tout localisateur fondamental $\DeuxLocFond{W}$ de $\DeuxCat$, 
$$
\DeuxLocFond{W}_{lax} \cap \UnCell{\DeuxCat} = \DeuxLocFond{W}
$$ 
et que, pour tout localisateur fondamental $\DeuxLocFondLax{W}$ de $\DeuxCatLax$, 
$$
\DeuxLocFondLax{W} = (\DeuxLocFondLax{W} \cap \UnCell{\DeuxCat})_{lax}.
$$

Soit donc $\DeuxLocFond{W}$ un \ClasseDeuxLocFond{}. Un \DeuxFoncteurStrict{} est dans $\DeuxLocFondLaxInduit{W} \cap \UnCell{\DeuxCat}$ si et seulement s'il est dans $\DeuxLocFondLaxInduit{W}$, donc si et seulement s'il est dans $\DeuxLocFond{W}$ (remarque \ref{EquiStricteEquiLax}), ce qui montre l'égalité $\DeuxLocFond{W}_{lax} \cap \UnCell{\DeuxCat} = \DeuxLocFond{W}$. 

Soit $\DeuxLocFondLax{W}$ un \ClasseDeuxLocFondLax{}. Un \DeuxFoncteurLax{} $u$ est dans $(\DeuxLocFondLax{W} \cap \UnCell{\DeuxCat})_{lax}$ si et seulement si $\TildeLax{u}$ est dans $\DeuxLocFondLax{W} \cap \UnCell{\DeuxCat}$, donc si et seulement si $\TildeLax{u}$ est dans $\DeuxLocFondLax{W}$, donc si et seulement si $u$ est dans $\DeuxLocFondLax{W}$ (lemme \ref{Bibi}). Cela montre l'égalité $\DeuxLocFondLax{W} = (\DeuxLocFondLax{W} \cap \UnCell{\DeuxCat})_{lax}$. 

La dernière assertion de l'énoncé se déduit du théorème \ref{EqCatLocDeuxCatDeuxCatLax}. 
\end{proof}

Le lemme \ref{SupLaxNaturel} se vérifie sans difficulté. 

\begin{lemme}\label{SupLaxNaturel}
Pour tout morphisme $u : \mathdeuxcat{A} \to \mathdeuxcat{B}$ de $\DeuxCatLax$, le diagramme
$$
\xymatrix{
\Delta/\NerfLax\mathdeuxcat{A}
\ar[rr]^{\SupLaxObjet{\mathdeuxcat{A}}}
\ar[d]_{\Delta/\NerfLax{(u)}}
&&\mathdeuxcat{A}
\ar[d]^{u}
\\
\Delta/\NerfLax\mathdeuxcat{B}
\ar[rr]_{\SupLaxObjet{\mathdeuxcat{B}}}
&&\mathdeuxcat{B}
}
$$
est commutatif.
\end{lemme}

\begin{prop}[Del Hoyo]\label{LemmeDelHoyo}
Pour toute petite \deux{}catégorie $\mathdeuxcat{A}$, pour tout objet $a$ de $\mathdeuxcat{A}$, la catégorie $\TrancheLax{(\DeuxIntOp{\Delta} \NerfHom \mathdeuxcat{A})}{\SupHom_{\mathdeuxcat{A}}}{a}$ est $\UnLocFondMin$\nobreakdash-asphérique.
\end{prop}

\begin{proof}
Voir la démonstration de \cite[théorème 9.2.4]{TheseDelHoyo} ou \cite[théorème 7.3]{ArticleDelHoyo}.
\end{proof}

\begin{prop}\label{SupHomAspherique}
Pour tout \ClasseDeuxLocFond{}, pour toute petite \deux{}catégorie $\mathdeuxcat{A}$, le \DeuxFoncteurLax{} normalisé
$$
\SupHom_{\mathdeuxcat{A}} : \DeuxIntOp{\Delta} \NerfHom \mathdeuxcat{A} \to \mathdeuxcat{A}
$$
est lax\nobreakdash-asphérique (donc en particulier une équivalence faible). 
\end{prop}

\begin{proof}
C'est une conséquence immédiate de la proposition \ref{LemmeDelHoyo} et du lemme \ref{Reynaldo}.
\end{proof}

%\begin{rem}
%La proposition \ref{SupHomAspherique} constitue la généralisation, à tout localisateur fondamental de $\DeuxCat$, de \cite[proposition 7.7]{ArticleThAMoi}.
%\end{rem}

\begin{prop}\label{SupLaxW}
Pour tout localisateur fondamental de $\DeuxCatLax$, pour toute petite \deux{}catégorie $\mathdeuxcat{A}$, le \DeuxFoncteurLax{} $\SupLaxObjet{\mathdeuxcat{A}} : \Delta/ \NerfLax\mathdeuxcat{A} \to \mathdeuxcat{A}$ est une équivalence faible. 
\end{prop}

\begin{proof}
C'est une conséquence du lemme \ref{DiagrammeSups}, de la proposition \ref{ResultatCCG}, de la remarque \ref{SatFaibleEquiLax} et de la proposition \ref{SupHomAspherique}.
\end{proof}

\begin{theo}\label{EqCatLocCatDeuxCatLax}
Pour tout localisateur fondamental $\DeuxLocFondLax{W}$ de $\DeuxCatLax$, l'inclusion $\Cat \hookrightarrow \DeuxCatLax$ induit une équivalence de catégories localisées
$$
\Localisation{\DeuxCatLax}{\DeuxLocFondLax{W}} \simeq \Localisation{\Cat}{(\DeuxLocFond{W} \cap \UnCell{\Cat})}.
$$
\end{theo}

\begin{proof}
Cela résulte directement de la proposition \ref{SupLaxW}, de la remarque \ref{Proust} et de la proposition \ref{UnSupAspherique}.
\end{proof}

%\begin{lemme}\label{Zouzou}
%Soit $u$ un \DeuxFoncteurStrict{} qui est une équivalence faible pour tout \ClasseDeuxLocFond{}. Alors, les foncteurs $\Delta / \NerfLax{(u)}$ et $\Delta / \NerfLaxNor{(u)}$ sont dans $\UnLocFondMin$.
%\end{lemme}
%
%\begin{proof}
%On a déjà remarqué que la classe $\NerfLaxNor^{-1} i_{\Delta}^{-1} W$ était un \ClasseDeuxLocFond{} pour tout localisateur fondamental $\UnLocFond{W}$ de $\Cat$. C'est donc en particulier le cas pour $W = \UnLocFondMin$. En vertu de l'hypothèse, le foncteur $\Delta / \NerfLaxNor{(u)}$ est donc dans $\UnLocFondMin$. En vertu de la remarque \ref{Nini}, le foncteur $\Delta / \NerfLax{(u)}$ est aussi dans $\UnLocFondMin$.  
%\end{proof}

\begin{lemme}\label{EquiDefEquiLax}
Pour tout localisateur fondamental $\UnLocFond{W}$ de $\Cat$, pour tout morphisme $u$ de $\DeuxCatLax$, $\Delta/\NerfLax(\TildeLax{u})$ est dans $\UnLocFond{W}$ si et seulement si $\Delta/\NerfLax(u)$ l'est.
\end{lemme}

\begin{proof}
Soient $\UnLocFond{W}$ un \ClasseUnLocFond{} et $u : \mathdeuxcat{A} \to \mathdeuxcat{B}$ un morphisme de $\DeuxCatLax$. Les \DeuxFoncteursStricts{} $\StrictCanonique{\mathdeuxcat{A}}$ et $\StrictCanonique{\mathdeuxcat{B}}$ sont des équivalences faibles pour tout \ClasseDeuxLocFond{} (proposition \ref{CouniteColaxAspherique}). Comme $\NerfLaxNor^{-1} (i_{\Delta}^{-1} (W))$ est un \ClasseDeuxLocFond{}, $\Delta/\NerfLaxNor(\StrictCanonique{\mathdeuxcat{A}})$ et $\Delta/\NerfLaxNor(\StrictCanonique{\mathdeuxcat{B}})$ sont dans $\UnLocFond{W}$. Il en est donc de même de $\Delta/\NerfLax(\StrictCanonique{\mathdeuxcat{A}})$ et $\Delta/\NerfLax(\StrictCanonique{\mathdeuxcat{B}})$ (remarque \ref{Nini}). La saturation faible de $\UnLocFond{W}$ permet d'en déduire que les sections $\Delta/\NerfLax(\LaxCanonique{\mathdeuxcat{A}})$ et $\Delta/\NerfLax(\LaxCanonique{\mathdeuxcat{B}})$ sont dans $\UnLocFond{W}$. On conclut par un argument de 2 sur 3 après avoir appliqué le foncteur $i_{\Delta} \NerfLax$ au diagramme commutatif
\[
\xymatrix{
\TildeLax{\mathdeuxcat{A}}
\ar[r]^{\TildeLax{u}}
&\TildeLax{\mathdeuxcat{B}}
\\
\mathdeuxcat{A}
\ar[u]^{\LaxCanonique{\mathdeuxcat{A}}}
\ar[r]_{u}
&\mathdeuxcat{B}
\ar[u]_{\LaxCanonique{\mathdeuxcat{B}}}
&.
}
\]
\end{proof}

\begin{lemme}\label{Gabuzomeu}
Pour tout localisateur fondamental $\DeuxLocFondLax{W}$ de $\DeuxCatLax$, un morphisme $u$ de $\DeuxCatLax$ est une équivalence faible si et seulement si $\Delta / \NerfLax{(u)}$ en est une. 
\end{lemme}

\begin{proof}
C'est une conséquence immédiate du lemme \ref{SupLaxNaturel} et de la proposition \ref{SupLaxW}.
\end{proof}

\begin{lemme}\label{Gaston}
Pour tout localisateur fondamental $\UnLocFond{W}$ de $\Cat$,
$$
{\NerfLax}^{-1} (i_{\Delta}^{-1} (\UnLocFond{W})) = ({\NerfLaxNor}^{-1} (i_{\Delta}^{-1} (\UnLocFond{W})))_{lax}.
$$
\end{lemme}

\begin{proof}
Cela résulte de la suite d'équivalences suivante, pour tout morphisme $u$ de $\DeuxCatLax$.
\begin{align*}
u \in {\NerfLax}^{-1} (i_{\Delta}^{-1} (\UnLocFond{W})) &\Longleftrightarrow \Delta / \NerfLax{(u)} \in \UnLocFond{W}
\\
&\Longleftrightarrow \Delta / \NerfLax{(\TildeLax{u})} \in \UnLocFond{W} \phantom{bla} \text{(lemme \ref{EquiDefEquiLax})}
\\
&\Longleftrightarrow \Delta / \NerfLaxNor{(\TildeLax{u})} \in \UnLocFond{W} \phantom{bla} \text{(remarque \ref{Nini})}
\\
&\Longleftrightarrow \TildeLax{u} \in {\NerfLaxNor}^{-1} (i_{\Delta}^{-1} (\UnLocFond{W}))
\\
&\Longleftrightarrow u \in ({\NerfLaxNor}^{-1} (i_{\Delta}^{-1} (\UnLocFond{W})))_{lax}
\qedhere
\end{align*}
\end{proof}

\begin{lemme}\label{Leroux}
Pour tout localisateur fondamental $\UnLocFond{W}$ de $\Cat$, la classe
$
{\NerfLax}^{-1} (i_{\Delta}^{-1} (\UnLocFond{W})) 
$
est un \ClasseDeuxLocFondLax{}.
\end{lemme}

\begin{proof}
Cela résulte du lemme \ref{Gaston}, de la remarque \ref{Levet} et de la proposition \ref{DeuxLocFondInduitLax}.
\end{proof}

\begin{lemme}\label{DeuxLocFondLaxInterCat}
Pour tout localisateur fondamental $\DeuxLocFondLax{W}$ de $\DeuxCatLax$,  
$$
\DeuxLocFondLax{W} = {\NerfLax}^{-1} (i_{\Delta}^{-1}(\DeuxLocFondLax{W} \cap \UnCell{\Cat})).
$$
\end{lemme}

\begin{proof}
Un \DeuxFoncteurLax{} $u$ est dans $\DeuxLocFondLax{W}$ si et seulement si le foncteur $\Delta / \NerfLax{(u)}$ l'est (lemme \ref{Gabuzomeu}), donc si et seulement si $\Delta / \NerfLax{(u)}$ est dans $\DeuxLocFondLax{W} \cap \UnCell{\Cat}$, donc si et seulement si $u$ est dans ${\NerfLax}^{-1}i_{\Delta}^{-1}(\DeuxLocFondLax{W} \cap \UnCell{\Cat})$. 
\end{proof}

\begin{lemme}\label{Chamfort}
Pour tout localisateur fondamental $\UnLocFond{W}$ de $\Cat$, 
$$
{\NerfLax}^{-1} (i_{\Delta}^{-1} (\UnLocFond{W})) \cap \UnCell{\Cat} =  \UnLocFond{W}.
$$
\end{lemme}

\begin{proof}
C'est une conséquence immédiate de la proposition \ref{Sade} et du fait que la restriction du nerf lax $\NerfLax$ à $\Cat$ coïncide avec le nerf $\UnNerf$. En formule : 
\[
{\NerfLax}^{-1} (i_{\Delta}^{-1} (\UnLocFond{W})) \cap \UnCell{\Cat} = {\UnNerf}^{-1} (i_{\Delta}^{-1} (\UnLocFond{W})) \cap \UnCell{\Cat} = \UnLocFond{W} \cap \UnCell{\Cat} = \UnLocFond{W}.
\qedhere
\]
\end{proof}

%\begin{lemme}\label{DeuxLocFondLaxInterCat}
%Pour tout localisateur fondamental $\DeuxLocFondLax{W}$ de $\DeuxCatLax$,
%$$
%\DeuxLocFondLax{W} = {\NerfLax}^{-1}i_{\Delta}^{-1}(\DeuxLocFondLax{W} \cap \UnCell{\Cat}).
%$$
%\end{lemme}
%
%\begin{proof}
%C'est une conséquence du lemme \ref{SupLaxNaturel}, de la proposition \ref{SupLaxW} et de la condition $\rm{LF1}_{lax}$.
%\end{proof}

%\begin{prop}[Cisinski]\label{2.2.9.Cisinski}
%On a l'inclusion
%$$
%\EquiQuillen \subset i_{\Delta}^{-1} (\DeuxLocFondLax{W} \cap \UnCell{\Cat})
%$$
%\end{prop}
%
%\begin{proof}
%Cela résulte de la remarque \ref{RemDeuxLocFondLaxInterCat} et de la proposition $2.2.9.$ de \cite{LFM}.
%\end{proof}

%\begin{lemme}\label{Chamfort}
%Pour tout localisateur fondamental $\UnLocFond{W}$ de $\Cat$, 
%$$
%{\NerfLax}^{-1}i_{\Delta}^{-1} \UnLocFond{W} \cap \UnCell{\Cat} =  \UnLocFond{W}.
%$$
%\end{lemme}
%
%\begin{proof}
%C'est une conséquence immédiate de la proposition \ref{Sade} et du fait que la restriction du nerf lax $\NerfLax$ à $\Cat$ coïncide avec le nerf $\UnNerf$. En formules : 
%$$
%{\NerfLax}^{-1}i_{\Delta}^{-1} \UnLocFond{W} \cap \UnCell{\Cat} = {\UnNerf}^{-1}i_{\Delta}^{-1} \UnLocFond{W} \cap \UnCell{\Cat} = \UnLocFond{W} \cap \UnCell{\Cat} = \UnLocFond{W}.
%$$
%\end{proof}

\begin{theo}\label{IsoUnLocFondDeuxLocFondLax}
Les applications
$$
\begin{aligned}
\mathcal{P}(\UnCell{\Cat}) &\to \mathcal{P}(\UnCell{\DeuxCatLax})
\\
\UnLocFond{W} &\mapsto {\NerfLax}^{-1}i_{\Delta}^{-1} \UnLocFond{W} 
\end{aligned}
$$
et
$$
\begin{aligned}
\mathcal{P}(\UnCell{\DeuxCatLax}) &\to \mathcal{P}(\UnCell{\Cat})
\\
\DeuxLocFondLax{W} &\mapsto \DeuxLocFondLax{W} \cap \UnCell{\Cat}
\end{aligned}
$$
induisent des isomorphismes inverses l'un de l'autre entre la classe ordonnée par inclusion des localisateurs fondamentaux de $\Cat$ et la classe ordonnée par inclusion des localisateurs fondamentaux de $\DeuxCatLax$. De plus, pour tout localisateur fondamental $\DeuxLocFond{W}$ de $\DeuxCatLax$, les catégories localisées 
$
\Localisation{\DeuxCatLax}{\DeuxLocFondLax{\DeuxLocFond{W}}}
$ 
et
$
\Localisation{\Cat}{(\DeuxLocFondLax{W} \cap \UnCell{\Cat})}
$
sont équivalentes et, pour tout localisateur fondamental $\UnLocFond{W}$ de $\Cat$, les catégories localisées 
$
\Localisation{\Cat}{\UnLocFond{W}}
$ 
et
$
\Localisation{\DeuxCatLax}{({\NerfLax}^{-1}i_{\Delta}^{-1} \UnLocFond{W})}
$
sont équivalentes, ces équivalences de catégories localisées étant induites par l'inclusion $\Cat \hookrightarrow \DeuxCatLax$.
\end{theo}

\begin{proof}
Ces applications respectant manifestement la relation d'inclusion, il résulte des lemmes \ref{Leroux}, \ref{DeuxLocFondLaxInterCat} et \ref{Chamfort} qu'il s'agit bien d'isomorphismes. La dernière assertion de l'énoncé se déduit du théorème \ref{EqCatLocCatDeuxCatLax}. 
\end{proof}

%\begin{lemme}\label{Souvestre}
%Pour tout localisateur fondamental $\UnLocFond{W}$ de $\Cat$, 
%$$
%{\NerfLax}^{-1}i_{\Delta}^{-1} \UnLocFond{W} \cap \UnCell{\DeuxCat} = {\NerfLaxNor}^{-1}i_{\Delta}^{-1} \UnLocFond{W}
%$$
%\end{lemme}
%
%\begin{proof}
%C'est une conséquence du fait que les flèches verticales du diagramme commutatif
%$$
%\xymatrix{
%\Delta / \NerfLaxNor \mathdeuxcat{A}
%\ar[rr]^{\Delta / \NerfLaxNor (u)}
%\ar[d]
%&& 
%\Delta / \NerfLaxNor \mathdeuxcat{B}
%\ar[d]
%\\
%\Delta / \NerfLax \mathdeuxcat{A}
%\ar[rr]_{\Delta / \NerfLax (u)}
%&& 
%\Delta / \NerfLax \mathdeuxcat{B}
%}
%$$
%déjà considéré dans la remarque \ref{Nini} sont des $\UnLocFondMin$\nobreakdash-équivalences faibles pour tout \DeuxFoncteurStrict{} $u : \mathdeuxcat{A} \to \mathdeuxcat{B}$ et du théorème \ref{CisinskiGrothendieck}.
%\end{proof}

\begin{lemme}\label{Allain}
Pour tout localisateur fondamental $\DeuxLocFond{W}$ de $\DeuxCat$, 
$$
\DeuxLocFondLaxInduit{\DeuxLocFond{W}} \cap \UnCell{\Cat} = \DeuxLocFond{W}
\cap \UnCell{\Cat}.
$$
\end{lemme}

\begin{proof}
En vertu du théorème \ref{IsoDeuxLocFondDeuxLocFondLax}, $ \DeuxLocFondLaxInduit{\DeuxLocFond{W}} \cap \UnCell{\DeuxCat} = \DeuxLocFond{W}$. Ainsi,
\[
 \DeuxLocFondLaxInduit{\DeuxLocFond{W}} \cap \UnCell{\Cat} =
 (\DeuxLocFondLaxInduit{\DeuxLocFond{W}} \cap \UnCell{\DeuxCat}) \cap
 \UnCell{\Cat} = \DeuxLocFond{W} \cap \UnCell{\Cat}.
\qedhere
\]
\end{proof}

\begin{theo}\label{IsoUnLocFondDeuxLocFond}
Les applications
$$
\begin{aligned}
\mathcal{P}(\UnCell{\Cat}) &\to \mathcal{P}(\UnCell{\DeuxCat})
\\
\UnLocFond{W} &\mapsto {\NerfLaxNor}^{-1}i_{\Delta}^{-1} \UnLocFond{W}
\\
&\phantom{bla}(= {\NerfLax}^{-1}i_{\Delta}^{-1} \UnLocFond{W} \cap \UnCell{\DeuxCat})
\end{aligned}
$$
et
$$
\begin{aligned}
\mathcal{P}(\UnCell{\DeuxCat}) &\to \mathcal{P}(\UnCell{\Cat})
\\
\DeuxLocFond{W} &\mapsto \DeuxLocFond{W} \cap \UnCell{\Cat}
\end{aligned}
$$
induisent des isomorphismes inverses l'un de l'autre entre la classe ordonnée par inclusion des localisateurs fondamentaux de $\Cat$ et la classe ordonnée par inclusion des localisateurs fondamentaux de $\DeuxCat$.  De plus, pour tout localisateur fondamental $\DeuxLocFond{W}$ de $\DeuxCat$, les catégories localisées 
$
\Localisation{\DeuxCat}{\DeuxLocFond{W}}
$ 
et
$
\Localisation{\Cat}{(\DeuxLocFond{W} \cap \UnCell{\Cat})}
$
sont équivalentes et, pour tout localisateur fondamental $\UnLocFond{W}$ de $\Cat$, les catégories localisées 
$
\Localisation{\Cat}{\UnLocFond{W}}
$ 
et
$
\Localisation{\DeuxCat}{({\NerfLaxNor}^{-1}i_{\Delta}^{-1} \UnLocFond{W})}
$
sont équivalentes, ces équivalences de catégories localisées étant induites par l'inclusion $\Cat \hookrightarrow \DeuxCat$. 
\end{theo}

\begin{proof}
C'est une conséquence des théorèmes \ref{IsoDeuxLocFondDeuxLocFondLax} et \ref{IsoUnLocFondDeuxLocFondLax}, de la remarque \ref{Nini} et du lemme \ref{Allain}. 
\end{proof}

\begin{rem}
Pour démontrer le théorème \ref{IsoUnLocFondDeuxLocFond}, nous utilisons donc de façon cruciale le théorème \ref{IsoUnLocFondDeuxLocFondLax}. Cela reflète l'importance des morphismes lax en théorie de l'homotopie : les homotopies que l'on rencontre proviennent généralement de morphismes non stricts (voir par exemple \cite[définition 2.19]{ArticleThAMoi}). L'introduction de la notion de localisateur fondamental de $\DeuxCatLax$ ne devrait donc pas sembler mystérieuse. Indépendamment de son utilité dans la démonstration du théorème \ref{IsoUnLocFondDeuxLocFond}, elle s'est en fait imposée à nous comme totalement naturelle dès que nous avons pris connaissance de l'existence d'un analogue pour les \deux{}foncteurs lax du Théorème A de Quillen \cite{NotesDelHoyo}. Les isomorphismes que nous avons dégagés entre les classes des localisateurs fondamentaux de $\Cat$, de $\DeuxCat$ et de $\DeuxCatLax$ permettent de parler de \emph{localisateur fondamental}, sans préciser de catégorie « de base ».
\end{rem}

\begin{rem}
Pour tout localisateur fondamental, on a vu que les inclusions $\Cat \hookrightarrow \DeuxCat$, $\DeuxCat \hookrightarrow \DeuxCatLax$ et $\Cat \hookrightarrow \DeuxCatLax$ induisaient une équivalence de catégories entre les catégories homotopiques associées. Des inverses respectifs sont donnés par les foncteurs induits à ce niveau par les foncteurs $i_{\Delta} \NerfLaxNor : \DeuxCat \to \Cat$, $B : \DeuxCatLax \to \DeuxCat$ (le « foncteur de strictification de Bénabou ») et $i_{\Delta} \NerfLax : \DeuxCatLax \to \Cat$.
 \end{rem}

La notion de \ClasseDeuxLocFondLax{} est stable par intersection. On définit le \emph{localisateur fondamental minimal de $\DeuxCatLax$} comme l'intersection de tous les localisateurs fondamentaux de $\DeuxCatLax$. 

\begin{theo}\label{TheoDeuxLocFondLaxMin}
Le localisateur fondamental minimal de $\DeuxCatLax$ est la classe 
$$
\DeuxLocFondLaxMin{} = {\NerfLax}^{-1}\EquiQuillen.
$$
\end{theo}

\begin{proof}
C'est une conséquence immédiate du théorème \ref{IsoUnLocFondDeuxLocFondLax} et du théorème \ref{CisinskiGrothendieck}.
\end{proof}

La notion de \ClasseDeuxLocFond{} est stable par intersection. On définit le \emph{localisateur fondamental minimal de $\DeuxCat$} comme l'intersection de tous les \ClasseDeuxLocFondS{}.

\begin{theo}\label{TheoDeuxLocFondMin}
Le localisateur fondamental minimal de $\DeuxCat$ est la classe 
$$
\DeuxLocFondMin{} = {\NerfLax}^{-1}\EquiQuillen \cap \UnCell{\DeuxCat} = {\NerfLaxNor}^{-1}\EquiQuillen.
$$
\end{theo}
 
 \begin{proof}
C'est une conséquence immédiate du théorème \ref{IsoUnLocFondDeuxLocFond} et du théorème \ref{CisinskiGrothendieck}. 
 \end{proof}

On termine cette section par quelques énoncés permettant notamment d'assurer que les isomorphismes entre localisateurs fondamentaux de $\Cat$, $\DeuxCat$ et $\DeuxCatLax$ figurant dans l'énoncé des théorèmes \ref{IsoDeuxLocFondDeuxLocFondLax}, \ref{IsoUnLocFondDeuxLocFondLax} et \ref{IsoUnLocFondDeuxLocFond} préservent la propriété d'être engendré par un \emph{ensemble} de morphismes, détail d'importance lorsqu'il s'agit de montrer l'existence de structures de catégories de modèles sur $\Cat$ et $\DeuxCat$ dont la classe des équivalences faibles est donnée par un localisateur fondamental (\emph{cf.} \cite{Ara}). 

\begin{df}\label{def:loc_eng}
On dira qu'un \ClasseUnLocFond{} (resp. un \ClasseDeuxLocFond{}, resp. un \ClasseDeuxLocFondLax{}) est \emph{engendré} par une classe $S$ de morphismes de $\Cat$ (resp. de $\DeuxCat$, resp. de $\DeuxCatLax$) si c'est le plus petit \ClasseUnLocFond{} (resp. \ClasseDeuxLocFond{}, resp. \ClasseDeuxLocFondLax{}) contenant $S$ ou, autrement dit, l'intersection de tous les \ClassesUnLocFond{} (resp. \ClassesDeuxLocFond{}, resp. \ClassesDeuxLocFondLax{}) contenant $S$. 
\end{df}

\begin{prop}\label{Tatata}
Si un localisateur fondamental $\DeuxLocFond{W}$ de $\DeuxCat$ est engendré par une classe $S \subset \UnCell{\DeuxCat}$, alors le localisateur fondamental $\DeuxLocFondLaxInduit{W}$ de $\DeuxCatLax$ est également engendré par $S$. 
\end{prop}

\begin{proof}
On a évidemment $S \subset \DeuxLocFondLaxInduit{W}$. Soit $\DeuxLocFond{W}'$ un localisateur fondamental de $\DeuxCatLax$ contenant $S$. En vertu du théorème \ref{IsoDeuxLocFondDeuxLocFondLax}, l'inclusion $\DeuxLocFondLaxInduit{W} \subset \DeuxLocFond{W}'$ équivaut à $\DeuxLocFondLaxInduit{W} \cap \UnCell{\DeuxCat} \subset \DeuxLocFond{W}' \cap \UnCell{\DeuxCat}$, c'est-à-dire, en vertu de ce même théorème, $\DeuxLocFond{W} \subset \DeuxLocFond{W}' \cap \UnCell{\DeuxCat}$. Cette inclusion résulte de l'hypothèse faite sur $\DeuxLocFond{W}$ et du fait que $\DeuxLocFond{W}' \cap \UnCell{\DeuxCat}$ est un \ClasseDeuxLocFond{} contenant $S$. 
\end{proof}

\begin{df}
Pour toute classe $S \subset \UnCell{\DeuxCatLax}$, on pose
$$
\widetilde{S} = \{ \TildeLax{u}, u \in S \}.
$$
\end{df}

\begin{lemme}\label{Blablabla}
Soit $\DeuxLocFondLax{W}$ un \ClasseDeuxLocFondLax{}. S'il est engendré par $S \subset \UnCell{\DeuxCatLax}$, alors il est engendré par $\widetilde{S}$. 
\end{lemme}

\begin{proof}
Pour tout $u : \mathdeuxcat{A} \to \mathdeuxcat{B}$ dans $S$, on a le diagramme commutatif
$$
\xymatrix{
\TildeLax{\mathdeuxcat{A}}
\ar[r]^{\TildeLax{u}}
&\TildeLax{\mathdeuxcat{B}}
\\
\mathdeuxcat{A}
\ar[u]^{\LaxCanonique{\mathdeuxcat{A}}}
\ar[r]_{u}
&\mathdeuxcat{B}
\ar[u]_{\LaxCanonique{\mathdeuxcat{B}}}
}
$$
dont les flèches verticales sont dans $\DeuxLocFondLax{W}$. Comme $u$ l'est aussi, c'est également le cas de $\TildeLax{u}$, ce qui montre l'inclusion $\widetilde{S} \subset \DeuxLocFondLax{W}$. Soit maintenant $\DeuxLocFondLax{W}'$ un \ClasseDeuxLocFondLax{} contenant $\widetilde{S}$. La considération du même diagramme, dont les flèches verticales sont dans $\DeuxLocFondLax{W}'$, permet d'affirmer $S \subset \DeuxLocFondLax{W}'$, donc $\DeuxLocFondLax{W} \subset \DeuxLocFondLax{W}'$.
\end{proof}

\begin{prop}\label{Tetete}
Soit $\DeuxLocFondLax{W}$ un \ClasseDeuxLocFondLax{}. S'il est engendré par $S \subset \UnCell{\DeuxCatLax}$, alors le localisateur fondamental $\DeuxLocFondLax{W} \cap \UnCell{\DeuxCat}$ de $\DeuxCat$ est engendré par $\widetilde{S}$. 
\end{prop}

\begin{proof}
En vertu du lemme \ref{Blablabla}, $\DeuxLocFondLax{W}$ est engendré par $\widetilde{S}$. On a bien sûr $\widetilde{S} \subset \DeuxLocFondLax{W} \cap \UnCell{\DeuxCat}$. Soit $\DeuxLocFond{W}'$ un \ClasseDeuxLocFond{} contenant $\widetilde{S}$. Comme $\DeuxLocFond{W}' \subset \DeuxLocFondLaxInduit{\DeuxLocFond{W}'}$, l'hypothèse implique $\widetilde{S} \subset \DeuxLocFondLaxInduit{\DeuxLocFond{W}'}$, donc $\DeuxLocFond{W} \subset \DeuxLocFondLaxInduit{\DeuxLocFond{W}'}$, c'est-à-dire $(\DeuxLocFondLax{W} \cap \UnCell{\DeuxCat})_{lax} \subset \DeuxLocFondLaxInduit{\DeuxLocFond{W}'}$, donc $\DeuxLocFondLax{W} \cap \UnCell{\DeuxCat} \subset \DeuxLocFond{W}'$, ce qui permet de conclure. 
\end{proof}

\begin{rem}
On se gardera de croire que, si un localisateur fondamental $\DeuxLocFondLax{W}$ de $\DeuxCatLax$ est engendré par une classe de \DeuxFoncteursLax{} $S$, alors le localisateur fondamental $\DeuxLocFondLax{W} \cap \UnCell{\DeuxCat}$ de $\DeuxCat$ est engendré par ${S} \cap \UnCell{\DeuxCat}$. Pour un contre-exemple, on peut considérer $S = \UnCell{\DeuxCatLax} \backslash \UnCell{\DeuxCat}$, c'est-à-dire la classe des morphismes de $\DeuxCatLax$ qui ne sont pas dans $\DeuxCat$.
\end{rem}

\begin{prop}\label{Tititi}
Soit $\UnLocFond{W}$ un \ClasseUnLocFond{}. S'il est engendré par $S \subset \UnCell{\Cat}$, alors le localisateur fondamental $\NerfLax^{-1} (i_{\Delta}^{-1} (\UnLocFond{W}))$ de $\DeuxCatLax$ est également engendré par $S$. 
\end{prop}

\begin{proof}
On a bien sûr $S \subset \NerfLax^{-1} (i_{\Delta}^{-1} (\UnLocFond{W}))$. Soit de plus $\DeuxLocFondLax{W}$ un \ClasseDeuxLocFondLax{} contenant $S$. L'inclusion $\NerfLax^{-1} (i_{\Delta}^{-1} (\UnLocFond{W})) \subset \DeuxLocFondLax{W}$ équivaut à $\NerfLax^{-1} (i_{\Delta}^{-1} (\UnLocFond{W})) \cap \UnCell{\Cat} \subset \DeuxLocFondLax{W} \cap \UnCell{\Cat}$, c'est-à-dire à $\UnLocFond{W} \subset \DeuxLocFondLax{W} \cap \UnCell{\Cat}$, ce qui résulte du fait que $\DeuxLocFondLax{W} \cap \UnCell{\Cat}$ est un \ClasseUnLocFond{} contenant $S$ et de l'hypothèse faite sur $\UnLocFond{W}$. 
\end{proof}

\begin{lemme}\label{Blebleble}
Soit $\DeuxLocFondLax{W}$ un \ClasseDeuxLocFondLax{}. S'il est engendré par $S \subset \UnCell{\DeuxCatLax}$, alors il est également engendré par $i_{\Delta}(\NerfLax (S))$.
\end{lemme}

\begin{proof}
Pour tout $u : \mathdeuxcat{A} \to \mathdeuxcat{B}$ dans $S$, on a le diagramme commutatif
$$
\xymatrix{
\Delta / \NerfLax \mathdeuxcat{A}
\ar[rr]^{\Delta / \NerfLax (u)}
\ar[d]_{\SupLaxObjet{\mathdeuxcat{A}}}
&&
\Delta / \NerfLax \mathdeuxcat{B}
\ar[d]^{\SupLaxObjet{\mathdeuxcat{B}}}
\\
\mathdeuxcat{A}
\ar[rr]_{u}
&&
\mathdeuxcat{B}
}
$$
dont les flèches verticales sont dans $\DeuxLocFondLax{W}$. C'est donc également le cas de $\Delta / \NerfLax (u)$, ce qui montre $i_{\Delta}(\NerfLax (S)) \subset \DeuxLocFondLax{W}$. Étant donné un localisateur fondamental $\DeuxLocFondLax{W}'$ de $\DeuxCatLax$ contenant $i_{\Delta}(\NerfLax (S))$, la considération du même diagramme permet de conclure $S \subset \DeuxLocFondLax{W}'$, donc $\DeuxLocFondLax{W} \subset \DeuxLocFondLax{W}'$. 
\end{proof}

\begin{prop}\label{Tototo}
Soit $\DeuxLocFondLax{W}$ un \ClasseDeuxLocFondLax{}. S'il est engendré par $S \subset \UnCell{\DeuxCatLax}$, alors le localisateur fondamental $\DeuxLocFondLax{W} \cap \UnCell{\Cat}$ de $\Cat$ est engendré par $i_{\Delta}(\NerfLax (S))$. 
\end{prop}

\begin{proof}
En vertu du lemme \ref{Blebleble}, $\DeuxLocFondLax{W}$ est engendré par $i_{\Delta}(\NerfLax (S))$, donc en particulier $i_{\Delta}(\NerfLax (S)) \subset \DeuxLocFondLax{W} \cap \UnCell{\Cat}$. Soit $\UnLocFond{W}$ un \ClasseUnLocFond{} contenant $i_{\Delta}(\NerfLax (S))$. On a donc l'inclusion $S \subset \NerfLax^{-1} (i_{\Delta}^{-1} (\UnLocFond{W}))$, donc $\DeuxLocFondLax{W} \subset \NerfLax^{-1} (i_{\Delta}^{-1} (\UnLocFond{W}))$ puisque $\NerfLax^{-1} (i_{\Delta}^{-1} (\UnLocFond{W}))$ est un \ClasseDeuxLocFondLax{} et que $\DeuxLocFondLax{W}$ est le plus petit \ClasseDeuxLocFondLax{} contenant $S$. En vertu du lemme \ref{DeuxLocFondLaxInterCat}, cela se récrit $\NerfLax^{-1} (i_{\Delta}^{-1} (\DeuxLocFondLax{W} \cap \UnCell{\Cat})) \subset \NerfLax^{-1} (i_{\Delta}^{-1} (\UnLocFond{W}))$, d'où, en vertu du théorème \ref{IsoUnLocFondDeuxLocFondLax}, $\DeuxLocFondLax{W} \cap \UnCell{\Cat} \subset \UnLocFond{W}$. 
\end{proof}

\begin{prop}\label{prop:bij_acc_1}
Soit $\DeuxLocFond{W}$ un \ClasseDeuxLocFond{}. S'il est engendré par $S \subset \UnCell{\DeuxCat}$, alors le localisateur fondamental $\DeuxLocFond{W} \cap \UnCell{\Cat}$ de $\Cat$ est engendré par $i_{\Delta}(\NerfLaxNor (S))$. 
\end{prop}

\begin{proof}
C'est une conséquence des propositions \ref{Tatata} et \ref{Tototo} et de la remarque \ref{Nini}.
\end{proof}

\begin{prop}\label{prop:bij_acc_2}
Soit $\UnLocFond{W}$ un \ClasseUnLocFond{}. S'il est engendré par $S \subset \UnCell{\Cat}$, alors le localisateur fondamental $\NerfLaxNor^{-1} (i_{\Delta}^{-1}(\UnLocFond{W}))$ de $\DeuxCat$ est également engendré par $S$. 
\end{prop}

\begin{proof}
C'est une conséquence des propositions \ref{Tetete} et \ref{Tititi} et du lemme \ref{BarreTildeW}. 
\end{proof}

\section{Critère local}

\begin{df}
Soit $\UnLocFond{W}$ un \ClasseUnLocFond{}. Un morphisme $u : A \to B$ de $\Cat$ est \emph{$\UnLocFond{W}$-localement constant}, ou plus simplement \emph{localement constant}, si, pour tout morphisme $b \to b'$ de $B$, le morphisme $A/b \to A/b'$ de $\Cat$ est une $\UnLocFond{W}$\nobreakdash-équivalence faible. 
\end{df}

\begin{theo}[Cisinski]\label{CaractTheoBCat}
Le localisateur fondamental minimal $\UnLocFondMin$ de $\Cat$ est le seul localisateur fondamental de $\Cat$ vérifiant les propriétés suivantes : 
\begin{itemize}
\item[(i)] Pour tout morphisme $u : A \to B$ de $\Cat$, si $u$ est une équivalence faible, alors $\pi_{0}(u) : \pi_{0}A \to \pi_{0}B$ est une bijection.
\item[(ii)] Pour tout morphisme $u : A \to B$ de $\Cat$ localement constant, $u$ est une équivalence faible si et seulement s'il est asphérique. 
\end{itemize}
\end{theo}

\begin{proof}
C'est le théorème 2.3.6. de \cite{LFM}. 
\end{proof}

\begin{df}
Soit $\DeuxLocFond{W}$ un \ClasseDeuxLocFond{}. Un morphisme $u : \mathdeuxcat{A} \to \mathdeuxcat{B}$ de $\DeuxCat$ est \emph{$\DeuxLocFond{W}$-lax-localement constant} ou, plus simplement, \emph{lax-localement constant} si, pour tout morphisme $b \to b'$ de $B$, le morphisme $\TrancheLax{\mathdeuxcat{A}}{u}{b} \to \TrancheLax{\mathdeuxcat{A}}{u}{b'}$ de $\DeuxCat$ est une $\DeuxLocFond{W}$\nobreakdash-équivalence faible. 
\end{df}

\begin{paragr}
On rappelle qu'il existe une structure de catégorie de modèles sur $\EnsSimp$ dont les équivalences faibles sont les équivalences faibles simpliciales et dont les cofibrations sont les monomorphismes. Cela permet (même si ce n'est en principe pas indispensable) de donner sens à la notion de carré homotopiquement cartésien dans $\EnsSimp$.
\end{paragr}

\begin{theo}[Cegarra]\label{ThBCegarra}
Soit $u : \mathdeuxcat{A} \to \mathdeuxcat{B}$ un \DeuxFoncteurStrict{} lax-localement constant. Alors, pour tout objet $b$ de $\mathdeuxcat{B}$, le carré canonique
$$
\xymatrix{
\NerfLaxNor (\TrancheLax{\mathdeuxcat{A}}{u}{b})
\ar[r]
\ar[d]
&
\NerfLaxNor (\mathdeuxcat{A})
\ar[d]
\\
\NerfLaxNor (\TrancheLax{\mathdeuxcat{B}}{}{b})
\ar[r]
&
\NerfLaxNor (\mathdeuxcat{B}) 
}
$$
est homotopiquement cartésien. 
\end{theo}

\begin{proof}
C'est un énoncé dual de celui de \cite[théorème 3.2]{Cegarra}. 
\end{proof}

Pour toute petite \deux{}catégorie $\mathdeuxcat{A}$, on note $\pi_{0}\mathdeuxcat{A}$ le quotient de l'ensemble $\Objets{\mathdeuxcat{A}}$ par la relation d'équivalence engendrée par la relation élémentaire « $a \sim a'$ s'il existe une \un{}cellule de $a$ vers $a'$ dans $\mathdeuxcat{A}$ ». Cela permet de définir un foncteur $\pi_{0} : \DeuxCat \to \Ens$.

\begin{theo}\label{CaractTheoBDeuxCat}
Le localisateur fondamental minimal $\DeuxLocFondMin$ de $\DeuxCat$ est le seul localisateur fondamental de $\DeuxCat$ vérifiant les propriétés suivantes.
\begin{itemize}
\item[(i)] Pour tout morphisme $u : \mathdeuxcat{A} \to \mathdeuxcat{B}$ de $\DeuxCat$, si $u$ est une équivalence faible, alors $\pi_{0}(u) : \pi_{0}\mathdeuxcat{A} \to \pi_{0}\mathdeuxcat{B}$ est une bijection.
\item[(ii)] Pour tout morphisme $u : \mathdeuxcat{A} \to \mathdeuxcat{B}$ de $\DeuxCat$ lax-localement constant, $u$ est une équivalence faible si et seulement s'il est lax-asphérique. 
\end{itemize}
\end{theo}

\begin{proof}
Le localisateur fondamental $\DeuxLocFondMin$ de $\DeuxCat$ vérifie par définition la condition $(i)$ de l'énoncé du théorème \ref{CaractTheoBDeuxCat}. On sait déjà qu'un morphisme lax-asphérique de $\DeuxCat$ est une équivalence faible. Réciproquement, si un morphisme $u$ de $\DeuxCat$ est $\DeuxLocFondMin$\nobreakdash-lax-localement constant et que c'est une $\DeuxLocFondMin$\nobreakdash-équivalence faible, alors, en vertu du théorème \ref{ThBCegarra}, il est $\DeuxLocFondMin$\nobreakdash-lax-asphérique. Le localisateur fondamental $\DeuxLocFondMin$ de $\DeuxCat$ vérifie donc la condition $(ii)$. 

Soit $\DeuxLocFond{W}$ un localisateur fondamental de $\DeuxCat$. S'il vérifie les conditions $(i)$ et $(ii)$ de l'énoncé du théorème \ref{CaractTheoBDeuxCat}, le localisateur fondamental $\DeuxLocFond{W} \cap \UnCell{\Cat}$ de $\Cat$ vérifie les conditions $(i)$ et $(ii)$ de l'énoncé du théorème \ref{CaractTheoBCat}. En vertu de ce même théorème \ref{CaractTheoBCat}, $\DeuxLocFond{W} \cap \UnCell{\Cat}$ n'est autre que $\UnLocFondMin$. On en déduit $\DeuxLocFond{W} = \DeuxLocFondMin$ en vertu du théorème \ref{IsoUnLocFondDeuxLocFond}. 
\end{proof}

\pdfbookmark[0]{Références}{bibliography}

\bigskip

\textsc{\footnotesize
Jonathan Chiche}

\footnotesize{\emph{Adresse électronique }: \url{jonathan.chiche@polytechnique.org}}

\footnotesize{\emph{URL }: \url{http://webusers.imj-prg.fr/~jonathan.chiche/}}

\end{document}